\documentclass[12pt,reqno]{amsart}

\usepackage{latexsym,amsfonts,amssymb,amsmath,epsfig,amscd}
\usepackage{float}
\usepackage{url}
\usepackage[all]{xy}
\usepackage{pgf,tikz}
\usepackage{mathrsfs}
\usepackage{graphicx} 
\usetikzlibrary{arrows}
\usetikzlibrary[patterns]

\newcommand{\N}{{\mathbb N}}
\newcommand{\Z}{{\mathbb Z}}
\newcommand{\Q}{{\mathbb Q}}
\newcommand{\C}{{\mathbb C}}
\newcommand{\R}{{\mathbb R}}
\renewcommand{\P}{{\mathbb P}}
\renewcommand{\H}{{\mathbb H}}

\newcommand{\BB}{{\mathcal B}}

\newcommand{\EE}{{\mathcal E}}

\newcommand{\LL}{{\mathcal L}}

\newcommand{\NN}{{\mathcal N}}

\newcommand{\OO}{{\mathcal O}}

\newcommand{\www}{\widetilde}
\newcommand{\oooo}{\overline}
\newcommand{\uuuu}{\underline}
\newcommand{\lcm}{{\rm lcm}}
\newcommand{\pr}{{\rm pr}}
\newcommand{\GL}{{\rm GL}}
\newcommand{\SL}{{\rm SL}}
\newcommand{\PSL}{{\rm PSL}}

\DeclareMathOperator{\Ann}{Ann}
\DeclareMathOperator{\Aut}{Aut}
\DeclareMathOperator{\id}{id}
\DeclareMathOperator{\mmod}{mod}
\DeclareMathOperator{\tr}{tr}

\theoremstyle{plain}
\newtheorem{lemma}{Lemma}[section]
\newtheorem{definition/lemma}[lemma]{Definition/Lemma}
\newtheorem{theorem}[lemma]{Theorem}
\newtheorem{corollary}[lemma]{Corollary}

\theoremstyle{definition}
\newtheorem{definition}[lemma]{Definition}
\newtheorem{remarks}[lemma]{Remarks}
\newtheorem{example}[lemma]{Example}
\newtheorem{examples}[lemma]{Examples}
\newtheorem{notations}[lemma]{Notations}

\begin{document}

\title[Conjugacy classes of regular integer matrices]
{Conjugacy classes of \\regular integer matrices}

\author[C. Hertling and K. Larabi]
{Claus Hertling and Khadija Larabi}
\address{Lehrstuhl f\"ur algebraische Geometrie,
Universit\"at Mannheim, B6, 26, 68159 Mannheim, Germany}
\email{hertling\char64 math.uni-mannheim.de}
\email{khadija.larabi\char64 outlook.com}



\date{October 10, 2026}

\begin{abstract}
This paper is devoted to the theory of $\GL_n(\Z)$-conjugacy classes
of regular integer $n\times n$ matrices and related commutative semigroups. 
Such a matrix is
$\GL_n(\Q)$-conjugate to the companion matrix of its characteristic
polynomial. But the set of $\GL_n(\Z)$-conjugacy classes of regular
integer matrices with a fixed characteristic polynomial $f$ is
usually nontrivial (finite if $f$ has simple roots, infinite if $f$
has multiple roots). It is in 1:1-correspondence to a subsemigroup of
a certain quotient semigroup of the commutative semigroup of
full lattices in the algebra $\Q[t]/(f)$.
In its first part, the paper gives a survey on old and new results
on full lattices and orders in a finite dimensional commutative 
$\Q$-algebra with unit element and on induced semigroups.
In its longer second part, the paper applies this theory to many examples,
essentially all cases with $n=2$, many cases with $n=3$ and two cases
with arbitrary $n$, the case with $n$ different integer eigenvalues 
and the case of a single $n\times n$ Jordan block. 
\end{abstract}

\maketitle

\tableofcontents


\section{Introduction}\label{s1}
\setcounter{equation}{0}
\setcounter{table}{0}
\setcounter{figure}{0}

Let $k$ be a field and $n\in\Z_{\geq 2}$. 
The theory of $\GL_n(k)$-conjugacy classes of matrices in
$M_{n\times n}(k)$ and normal forms in them 
is a corner stone of linear algebra.
In contrast, the theory of $\GL_n(\Z)$-conjugacy classes of
matrices in $M_{n\times n}(\Z)$ and normal forms in them
is beyond linear algebra. It touches algebra, algebraic
number theory and representation theory. 
Rather few general facts are known. The most prominent one is
a consequence of the Jordan-Zassenhaus theorem \cite{Za38}.
It says that the set of semisimple matrices in 
$M_{n\times n}(\Z)$ with a fixed characteristic polynomial 
splits into finitely many $\GL_n(\Z)$-conjugacy 
classes.

The best studied cases are those where the characteristic
polynomial $f\in\Z[t]$ is irreducible.
Then one is in the domain of algebraic number theory.
The set of $\GL_n(\Z)$-conjugacy classes with a fixed
characteristic polynomial $f\in\Z[t]$ is in 
1:1-correspondence with a set of certain equivalence
classes of (fractional) ideals of the order $\Z[t]/(f)$
in the algebraic number field $\Q[t]/(f)$, and it inherits
the structure of a commutative semigroup.
In the case $n=2$ this is very close to the theory of
binary quadratic forms, which was worked out already by
Gau{\ss} \cite{Ga01}. For general $n$,
beyond the results in algebraic number theory books
like \cite{BSh73} and \cite{Ne99}, 
there are results by Taussky (see \cite{Ta78} for references), 
the paper \cite{DTZ62}, and the recent papers \cite{Hu17,
EHO19,Ma20,BHJ22}, see the Remarks \ref{t1.1}. 

The Jordan-Zassenhaus theorem is part of a theory of
orders in separable associative algebras 
\cite{CR62,Re70,Re03,Fa65-1}.
But the separability implies in our commutative situation that
the matrices are semisimple.

This paper is devoted to the theory of $\GL_n(\Z)$-conjugacy
classes of matrices in $M_{n\times n}(\Z)$
which are not necessarily semisimple.
But we restrict to {\it regular} integer matrices. 
A matrix $A\in M_{n\times n}(\C)$ is
{\it regular} if it is $\GL_n(\C)$-conjugate to a matrix 
in $M_{n\times n}(\C)$ in Jordan normal form which has only
one Jordan block for each eigenvalue.
This restriction is motivated by the fact that the
set $S_{1,f}$ of $\GL_n(\Z)$-conjugacy classes of 
{\it regular} integer matrices with fixed 
characteristic polynomial $f\in\Z[t]$ (unitary of degree $n$)
is a commutative semigroup, and comes equipped additionally
with a natural quotient map. 
We are intrigued by these structures and study them in this
paper in many examples. 

This paper is accompanied by our recent paper \cite{HL26}.
It builds upon \cite{DTZ62} and several papers of Faddeev
and develops a theory of full lattices and orders in a 
finite dimensional commutative $\Q$-algebra $A$ with unit element
and of induced commutative semigroups. The algebra $A$ is
not necessarily separable. The paper \cite{HL26}
develops the theory, but does not provide examples.
The present paper is devoted to examples and shows
old and new results in action.

One highlight of \cite{HL26} is an extension of the
Jordan-Zassenhaus to the case of not semisimple
matrices: The set $S_{1,f}$ of $\GL_n(\Z)$-conjugacy classes
of regular matrices in $M_{n\times n}(\Z)$ with a fixed
not semisimple Jordan normal form (i.e. the characteristic
polynomial $f$ has multiple roots) is infinite,
but a {\it single} additional invariant splits it into
infinitely many finite sets. This invariant is the
{\it order} of a {\it full lattice} which is 
associated to a given $\GL_n(\Z)$-conjugacy class
(see Theorem \ref{t5.3}).
A second new point in \cite{HL26} is that each of these
finite sets splits naturally into smaller finite sets
which have all the same size and where one of them
is a subgroup of invertible elements of the semigroup, 
while the others consist of not invertible elements 
(see Theorem \ref{t5.5}).
The subgroups of invertible elements in the semigroup
$S_{1,f}$ play a special role and can be controlled quite
well. Two results from \cite{HL26} allow to compare
these groups with one another (see the Theorems 
\ref{t5.10} and \ref{t5.12}). The second one was known 
before in the case of an algebraic number field \cite{Ne99}.
Another new point in \cite{HL26} is an extension
of the main result in \cite{DTZ62} beyond the
case of algebraic number fields: 
For each full lattice $L$ in $A$ each power
$L^k$ with $k\geq \dim_\Q A-1$ is invertible
(see Theorem \ref{t4.4}). 

This paper splits into two parts. The first and shorter part
is a theoretical survey part,
which presents old and new results, especially on
orders and full lattices in an algebra $A$ as above
and on the induced semigroups.
The second and longer part presents many examples which 
are studied in detail.

The first part consists of the sections \ref{s2}--\ref{s6}.
Section \ref{s2} recalls basic properties of commutative
semigroups. It follows \cite[ch. 4]{HL26}, but this is a 
condensed version of an account in \cite{DTZ62}.
Section \ref{s3} presents the (elementary) structure of
a finite dimensional commutative $\Q$-algebra $A$ with unit
element and discusses special types, Gorenstein algebras 
and cyclic algebras.
The subject of section \ref{s4} is the set 
$$\LL(A):=\{L\subset A\,|\, L\textup{ is a }\Z\textup{-lattice which
generates }A\textup{ over }\Q\}$$
of {\it full lattices} $L$ in $A$. The multiplication of full lattices
makes $\LL(A)$ a commutative semigroup. 
Section \ref{s4} collects old (mainly from papers of Faddeev
\cite{Fa65-1,Fa65-2,Fa67,Fa68} 
and from \cite{DTZ62}) and new (from \cite{HL26}) results 
on this semigroup.
Section \ref{s5} considers also two quotient semigroups
$\EE(A)$ and $W(\LL(A))$ and collects old and new results on their
structure. Finally, section \ref{s6} recalls the old and elementary,
but crucial observation \cite{LMD33} (see also Theorem \ref{t6.3})
that the set $S_{1,f}$  is in 1:1-correspondence to a 
subsemigroup $S_{2,f}\subset \EE(A)$ of the semigroup $\EE(A)$,
where $A:=A_f:=\Q[t]/(f)$ is a cyclic $n$-dimensional commutative
$\Q$-algebra with unit element (which is inseparable if $f$ has
multiple roots).

The second part consists of the section \ref{s7}--\ref{s12}.
The long section \ref{s7} is devoted to the irreducible cases
with $n=2$. Then $S_{1,f}$ is in 1:1-correspondence to the 
$\GL_2(\Z)$-equivalence classes of binary quadratic forms,
so then we are in the case of quadratic number fields.
The theory of binary quadratic forms was worked out by Gau{\ss}
\cite{Ga01} and is treated in many books (of which we cite only
\cite{BSh73} \cite{Tr13} \cite{We17} \cite{Ha21}).
Still we consider it useful to see explicitly the objects
of the sections \ref{s2}--\ref{s6} in these cases.
Also, the rank $n=2$ cases are the only ones where (semi-)normal
forms in $\GL_n(\Z)$-conjugacy classes can be given relatively
easily. In the case of a real quadratic number field we 
use Conway's topograph \cite{Co97} in order to make different 
semi-normal forms transparent.

Section \ref{s8} takes up an irreducible example with $n=3$
in \cite{DTZ62}. Here a certain cubic number field and orders and
full lattices in it are considered. We can build on results
documented in \cite{LMFDB23}.

Section \ref{s9} considers matrices which have $n$ different
integer eigenvalues. The corresponding algebra consists of 
1-dimensional subalgebras, so it is separable, but easier to
treat than algebraic number fields. For example its class number
is easily seen to be 1. Also semi-normal forms can be given 
without difficulty. But we improved them to normal forms only in the
cases $n=2$ and $n=3$. 

Section \ref{s10} considers the extreme case of $n\times n$ 
matrices
with a single Jordan block (with eigenvalue 0). The algebra is
irreducible with a radical of codimension 1. Semi-normal forms
for $n\times n$ matrices can be given easily, 
but improving them to 
normal forms is difficult and is done only for $n=2$ and $n=3$.
In the case $n=3$ the new invariant, the order of a full lattice
which is associated to a $\GL_3(\Z)$-conjugacy class, 
can be read off in an interesting way from the normal form
(see Theorem \ref{t10.7}.

Section \ref{s11} treats $3\times 3$ matrices with a 
$2\times 2$-Jordan block and a $1\times 1$-Jordan block
with different integer eigenvalues. 

The very short section \ref{s12} is offered for completeness
sake. It treats the only 3-dimensional algebra $A$ which
is not cyclic and thus not relevant for $3\times 3$ matrices.
The sections \ref{s10}, \ref{s11} and \ref{s12} 
allow to extend a result of Faddeev \cite{Fa65-2} on 
classes of not invertible full lattices for separable
3-dimensional algebras to all 3-dimensional algebras
(see the Theorems \ref{t5.7} and \ref{t5.8}).
This fits to one purpose of this paper, namely to give a
rather systematic treatment and panorama of the cases
with $n=3$.

At the end of this introduction, we first discuss related
work. Then we fix some notations.

\begin{remarks}\label{t1.1} (Related work)
Here $A$ is a commutative $\Q$-algebra of dimension 
$n\in\N$ with unit element $1_A$. 
\begin{enumerate}
\item[(i)]
The most important references for us are the paper of Dade, 
Taussky and Zassenhaus \cite{DTZ62} and a series of
papers of Faddeev \cite{Fa65-1,Fa65-2,Fa67,Fa68}.
Our own paper \cite{HL26} which is the source of the sections
\ref{s2}, \ref{s4} and \ref{s5} builds upon these papers. 
\cite{DTZ62} and the papers of Faddeev work with the semigroup
$\LL(A)$ and the induced semigroups $\EE(A)$ and $W(\LL(A))$.
Though \cite{DTZ62} gives concrete results only for the 
case when $A$ is an algebraic number field. 
\item[(ii)] The closest much more recent paper which works
with the semigroup $\LL(A)$ (and the quotient
semigroups $\EE(A)$ and $W(\LL(A))$) 
is the paper \cite{Ma20} of Marseglia. 
Though it restricts to the semisimple case, i.e. the case
when $A$ is a sum of algebraic number fields.
For this case, first it collects a part of the results in
our sections \ref{s4} and \ref{s5}. Then it gives an algorithmic 
way to decide whether two regular semisimple matrices 
in $M_n(\Z)$ are $\GL_n(\Z)$-conjugate.
The paper \cite{ZZ94} considers the semigroup $\LL(A)$
for $A$ an algebraic number field. It offers the observation
that for $n\geq 3$ not all fractional ideals are invertible,
and it gives some examples with $n=2$.
\item[(iii)]
The paper \cite{AO81} of Appelgate and Onishi considers
the case of a quadratic number field. It does not make the
connection to binary quadratic forms. But it gives in the
case of a real quadratic number field a direct connection
between one family of semi-normal forms and continued fractions.
The paper \cite{AO82} states normal forms for the reducible 
cases with $n=2$ and semi-normal forms 
for some reducible cases with $n=3$. 
It says nothing about the irreducible cases with $n=3$ and 
nothing about the semigroup structure. Its most elaborate result 
concerns the case $A=A^{(1)}\oplus A^{(2)}$ with 
$A^{(1)}\cong \Q$ and $A^{(2)}$ a quadratic number field. 
We do not treat this case in detail. 
\item[(iv)] 
Several recent papers offer algorithmic ways to decide
whether two matrices in $M_n(\Z)$ are $\GL_n(\Z)$-conjugate,
where these matrices are not necessarily regular.
The doctoral thesis \cite{Hu17} of Husert considers matrices 
which are either semisimple or nilpotent. The paper
\cite{EHO19} of Eick, Hofmann and O'Brien and the paper 
\cite{BHJ22} of Bley, Hofmann and Johnston
offer algorithms for all matrices in $M_n(\Z)$. 
The paper \cite{EHO19} works out an ansatz of Grunewald 
\cite{Gr80}. 
The discussion in \cite[9.5]{BHJ22} shows that the algorithms
in this paper are the fastest ones. 
\item[(v)]
Also the paper \cite{Fr65} of Fr\"ohlich and the old paper 
\cite{St11} of Steinitz concern the general case 
of all matrices in $M_n(\Z)$. But these papers are not used in 
the papers cited above. 
\item[(vi)] Taussky wrote many short papers with some special
observations on integer matrices. Many concern the irreducible
$n=2$ cases. Only \cite{DTZ62} is close to our paper.
The survey \cite{Ta78} gives a list of her other papers. 
\end{enumerate}
\end{remarks}

\begin{notations}\label{t1.2}
\begin{enumerate}
\item[(i)] 
$\N=\{1,2,...\}$.
$\P:=\{p\in\N\,|\, p\textup{ is a prime number}\}$.
For $m\in\N$, denote $\Z_m:=\Z/m\Z$. 
\item[(ii)] 
Throughout the whole paper, $A$ is a finite dimensional
commutative $\Q$-algebra with unit element $1_A$. 
\item[(iii)] 
For any commutative ring with unit element $1_R$, the group
of units in $R$ will be called
$R^{*}=\{a\in R\,|\ \exists\ b\in R\textup{ with }ab=1_R\}$.
\item[(iv)] 
Let $V$ be a $\Q$-vector space, $n\in\N$, 
$\uuuu{v}=(v_1,...,v_n)\in M_{1\times n}(V)$ an $n$-tuple
of elements in $V$ and $M=(m_{ij})\in M_{n\times n}(\Q)$. 
Then also 
$$\uuuu{v}\cdot M:=(\sum_{i=1}^nm_{i1}v_i,...,
\sum_{i=1}^n m_{in})\in M_{1\times n}(V)$$
is an $n$-tuple of elements of $V$. One element
is obtained by multiplying $\uuuu{v}$ with one
column of $M$. This notation will be convenient often.
\item[(v)] 
Define two maps
\begin{eqnarray*}
&&\nu_\Q:\Q-\{0\}\to\N\textup{ and }\\
&&\delta_\Q:\Q-\{0\}\to\N\textup{ by }\\
&&|q|=|\frac{\nu_\Q(q)}{\delta_\Q(q)}|\textup{ and }
\gcd(\nu_\Q(q),\delta_\Q(q))=1\textup{ for }q\in\Q-\{0\}.
\end{eqnarray*}
$\nu_\Q(q)$ and $\delta_\Q(q)$ are the numerator and the 
denominator of $|q|$ if one writes $|q|$ as a reduced fraction.
(v) and (vi) will be used in section \ref{s10}.
\item[(vi)] 
For $(q_1,q_2)\in\Q^2-\{(0,0)\}$ define
${\gcd}_\Q(q_1,q_2)\in\Q_{>0}$ as the unique positive rational
number with 
$$\Z \cdot q_1+\Z \cdot q_2=\Z\cdot{\gcd}_\Q(q_1,q_2).$$
The restricted map ${\gcd}_\Q:\Q_{>0}^2\to\Q_{>0}$ 
defines a commutative semigroup structure on $\Q_{>0}$.
This semigroup is isomorphic to the semigroup 
$\bigl(\{\Z\textup{-lattices}\neq 0\textup{ in }\Q\}, +\bigr).$
\end{enumerate}
\end{notations}

\section{Commutative semigroups}\label{s2}
\setcounter{equation}{0}
\setcounter{table}{0}
\setcounter{figure}{0}

This section is a review of basic properties of
commutative semigroups. We cite the notions and results in
\cite[Ch. 4]{HL26}. They are condensed versions of the 
notions and results in \cite[ch. 1.2]{DTZ62}.

\begin{definition}\label{t2.1}
\begin{enumerate}
\item[(a)] 
A commutative semigroup is a nonempty set $S$ with a 
multiplication map $\cdot:S\times S\to S$ which is commutative
($ab=ba$) and associative ($a(bc)=(ab)c$). 
The semigroup is called $S$, so the multiplication map 
is suppressed.
\item[(b)] 
Let $S$ be a commutative semigroup. An element $a\in S$
is called {\it invertible} if an element $b\in S$ and an element
$c\in S$ with $ab=c$ and $ac=a$ exist.
\item[(c)] 
Let $S$ be a commutative semigroup. An element $c\in S$ is
called an {\it idempotent} if $cc=c$. 
\end{enumerate}
\end{definition}

The following theorem says besides others that each idempotent
is a unit element in a maximal subgroup of $S$,
that the maximal subgroups are disjoint and that their union
is a subsemigroup of $S$ and is the set of all invertible
elements in $S$.

\begin{theorem}\label{t2.2}
\cite[1.2.3, 1.2.10]{DTZ62}\cite[4.2]{HL26}

Let $S$ be a commutative semigroup.

\begin{enumerate}
\item[(a)] 
Let $a\in S$ be invertible. Then there is a unique
element $c\in S$ with the properties $(ac=a,\exists\ b\in S
\textup{ with }ab=c)$. It is called $e_a$. It is idempotent.
There is a unique element $b\in S$ with the properties
$(ab=e_a,be_a=b)$. It is called $a^{-1}$. It is invertible, and 
$e_{a^{-1}}=e_a$. 
\item[(b)] 
An idempotent $c\in S$ is invertible with 
$e_c=c$ and $c^{-1}=c$.
\item[(c)] 
Let $c\in S$ be idempotent. The set 
\begin{eqnarray}\label{2.1}
G(c):=\{a\in S\,|\,a\textup{ invertible with }e_a=c\}
\end{eqnarray}
is a group. It is a maximal subgroup of $S$. 
Any maximal subgroup of $S$ is equal to $G(\www c)$
for some idempotent $\www c\in S$. 
\item[(d)] 
If $a,b\in S$ are invertible, then $ab$ is invertible
with $e_{ab}=e_ae_b$. 
\item[(e)] 
Suppose that $e_1,e_2\in S$ are both idempotent.
Then $e_1e_2$ is idempotent. The sets $G(e_1)e_2$, 
$G(e_2)e_1$ and $G(e_1)G(e_2)$ are subgroups of $G(e_1e_2)$.
The maps 
\begin{eqnarray}\label{2.2}
G(e_1)\to G(e_1)e_2,\ a\mapsto ae_2,\quad\textup{and}\quad
G(e_2)\to G(e_2)e_1,\ b\mapsto be_1,
\end{eqnarray}
are surjective group homomorphisms.
\item[(f)] 
The union $\bigcup_{c\textup{ idempotent}}G(c)\subset S$
of all invertible elements 
is a disjoint union. It is a subsemigroup of $S$. 
It is equal to $S$ if and only if any element of $S$ is invertible. 
\end{enumerate}
\end{theorem}

\begin{definition}\label{t2.3}
\begin{enumerate}
\item[(a)] 
Let $S$ be a commutative semigroup. 
A unit element in it is an element $e$ with $ea=a$ for all $a\in S$. 
Obviously, it is unique if it exists.
A commutative semigroup with a unit element is called a {\it monoid}.
\item[(b)] 
Let $S$ be a commutative semigroup. Two elements 
$a_1,a_2\in S$ are {\it $w$-equivalent} (notation: 
$a_1\sim_w a_2$),
if the following holds:
\begin{eqnarray}\label{2.3} a_1=a_2\quad\textup{or}\quad
\exists\ x_1,x_2\in S\textup{ with }a_1x_1=a_2,a_2x_2=a_1.
\end{eqnarray}
\end{enumerate}
\end{definition}

The $w$-equivalence on a semigroup $S$ gives rise to a 
natural quotient semigroup $S/\sim_w=:W(S)$, as Theorem
\ref{t2.4} shows.

\begin{theorem}\label{t2.4}
\cite[1.2.12, 1.2.16, 1.2.17]{DTZ62}\cite[4.4]{HL26}

Let $S$ be a commutative semigroup.

\begin{enumerate}
\item[(a)] $\sim_w$ is an equivalence relation and is compatible
with the multiplication in $S$, i.e.
\begin{eqnarray}\label{2.4}
a_1\sim_w a_2,b_1\sim_w b_2\Rightarrow a_1b_1\sim_w a_2b_2.
\end{eqnarray}
The equivalence class of $a$ is called $[a]_w$.
Therefore the quotient $S/\sim_w$ is a commutative semigroup. 
This semigroup is called $W(S)$. 
\item[(b)] 
If $a$ is invertible then $[a]_w=G(e_a)$. 
\item[(c)] 
Let $a\in S$. The following three properties are equivalent:
\begin{enumerate}
\item[(i)] $a$ is invertible.
\item[(ii)] $[a]_w$ is invertible (in $W(S)$).
\item[(iii)] $[a]_w$ is idempotent (in $W(S)$).
\end{enumerate}
The only subgroups of $W(S)$ are the sets $\{[c]_w\}$ with
$c\in S$ idempotent. 
\item[(d)] 
The subsemigroup $\bigcup_{c\textup{ idempotent}}G(c)$
(of invertible elements) of $S$ was considered in 
Theorem \ref{t2.2} (e). 
The subsemigroup $W(\bigcup_{c\textup{ idempotent}}G(c))$
of $W(S)$ is isomorphic to the subsemigroup
$\{c\in S\,|\, c\textup{ is idempotent}\}$ of $S$. 
\end{enumerate}
\end{theorem}

\section{Finite dimensional $\Q$-algebras and 
multiplicative metrics}\label{s3}
\setcounter{equation}{0}
\setcounter{table}{0}
\setcounter{figure}{0}

Recall that throughout the paper $A$ is a finite dimensional
commutative $\Q$-algebra with unit element $1_A$. 
Theorem \ref{t3.1} gives its structure.
The decomposition in \eqref{3.1} with (i) and (ii) 
follows from the structure theorem for commutative Artin
rings \cite[Theorem 8.7]{AMD69}. It is also known as
Peirce decomposition. The Wedderburn-Malcev theorem
\cite[(72.19) Theorem]{CR62} applies to each summand and
gives the splitting in part (iii).

\begin{theorem}\label{t3.1} (E.g. \cite[3.2]{HL26})
Let $A$ be a finite dimensional commutative $\Q$-algebra
with unit element $1_A$. Then $A$ has a unique vector space 
decomposition
\begin{eqnarray}\label{3.1}
A=\bigoplus_{j=1}^k A^{(j)} \quad\textup{for some }
k\in\N
\end{eqnarray}
with the following properties.
\begin{enumerate}
\item[(i)]
$A^{(1)},...,A^{(k)}$ are $\Q$-algebras with 
$A^{(i)}\cdot A^{(j)}=\{0\}$ for $i\neq j$.
\item[(ii)]
$1_A=\sum_{j=1}^k 1_{A^{(j)}}$ splits into unit elements 
$1_{A^{(j)}}$ of the algebras $A^{(j)}$. 
\item[(iii)]
$A^{(j)}$ splits uniquely into
\begin{eqnarray}\label{3.2}
A^{(j)}=F^{(j)}\oplus N^{(j)}
\end{eqnarray}
with $F^{(j)}$ a $\Q$-subalgebra and an algebraic number field,
and $N^{(j)}$ the unique maximal ideal of $A^{(j)}$.
$A^{(j)}$ is an $F^{(j)}$-algebra. 
$N^{(j)}$ consists of nilpotent elements and is the radical
of $A^{(j)}$.
\end{enumerate}
\end{theorem}

\begin{remarks}\label{t3.2}
(i) In fact, the Jordan-Chevalley decomposition holds over 
any perfect field \cite[VII \S 5 9. Theorem 1]{Bo90}. 
Therefore we can replace in Theorem \ref{t3.1} 
the field $\Q$ by any perfect field.
But in this paper we need only $\Q$.

(ii) Consider the sums 
\begin{eqnarray}\label{3.3}
F:=\bigoplus_{j=1}^k F^{(k)}\quad\textup{ and }\quad 
R=\bigoplus_{j=1}^k N^{(j)}.
\end{eqnarray}
$F$ is a separable subalgebra of $A$.
$R$ is the radical, namely the intersection of the maximal 
ideals and the set of all nilpotent elements of $A$. 
The algebra $A$ decomposes naturally 
into the direct sum $A=F\oplus R$. The induced projection
$\pr_F:A\to F$ 
respects addition, multiplication and division. 
We will come back to it in Lemma \ref{t5.9} and
Theorem \ref{t5.10}. 

(iii) $A$ is separable if and only if $A=F$, so if and only
if $R=\{0\}$.
\end{remarks}

\begin{definition}\label{t3.3}
Let $A$ be as in Theorem \ref{t3.1}.

\begin{enumerate}
\item[(a)] $A$ is called {\it cyclic} if it contains an element $a$
with $A=\Q[a]$.
\item[(b)] A symmetric $\Q$-bilinear form $\Phi:A\times A\to\Q$ is 
{\it multiplication invariant} if
$$\Phi(ab,c)=\Phi(a,bc)\quad\textup{for }a,b,c\in A.$$
A symmetric $\Q$-bilinear form which is multiplication invariant
and nondegenerate is called {\it multiplicative metric}.
If $A$ has a multiplicative metric then $A$ is a 
{\it Gorenstein ring}.
\end{enumerate}
\end{definition}

\begin{remarks}\label{t3.4}
Let $A$ be as in Theorem \ref{t3.1}.

\begin{enumerate}
\item[(i)] It is well known and easy to see 
that $A$ is a Gorenstein ring if and only if
for each $j\in\{1,...,k\}$ the socle 
$\Ann_{A^{(j)}}(N^{(j)})\subset A^{(j)}$ of $A^{(j)}$ is a 
1-dimensional vector space over $F^{(j)}$. 
In that case any $\Q$-linear form $l:A\to\Q$ with
$$\Ann_{A^{(j)}}(N^{(j)})\not\subset \ker(l)\quad\textup{for }
j\in\{1,...,k\}$$
gives rise to a multiplicative metric via
\begin{eqnarray}\label{3.4}
\Phi(a,b):=l(ab),
\end{eqnarray}
and any multiplicative metric is constructed in this way.
\item[(ii)] 
Any cyclic algebra is Gorenstein. If $A=\Q[a]$, then
the linear form $l:A\to\Q$ with 
\begin{eqnarray*}
l(a^m)&=&0\textup{ for }m\in\{0,1,...,\dim A-2\},\\
l(a^{\dim A-1})&=&1
\end{eqnarray*}
gives rise to a multiplicative metric by \eqref{3.4}.
\item[(iii)] 
Any algebraic number field is a cyclic algebra.
Therefore $A$ is cyclic if $A=F$, so if $A$ is separable.
\item[(iv)] 
The algebra $A=\Q[x,y]/(x^2,y^2)$ is Gorenstein, but not
cyclic.
\item[(v)] 
The algebra $A=\Q[x,y]/(x^2,xy,y^2)$ is not Gorenstein
(so also not cyclic), as it is irreducible with 
2-dimensional socle. See section \ref{s12} for it. 
\item[(vi)] 
Let $A=\Q[a]$ be cyclic. Denote by 
$\mu_a:A\to A$, $b\mapsto ab$, the multiplication by $a$
in $A$. It is a regular endomorphism of $A$ with a 
characteristic polynomial $f\in\Q[t]$. 
Then $A\cong \Q[t]/(f)$. Write 
\begin{eqnarray*}
f&=& \prod_{j=1}^kf_j^{n_j+1}
\end{eqnarray*}
with $n_j\in\Z_{\geq 0}$,
$f_j\in \Q[t]$ unitary and irreducible
and $f_i\neq f_j$ for $i\neq j$. 
Then $A=\bigoplus_{j=1}^kA^{(j)}$ as in Theorem \ref{t3.1}
with unit element $1_{A^{(j)}}$ in $A^{(j)}$ and  
\begin{eqnarray*}
A^{(j)}=\Q[a 1_{A^{(j)}}] &\cong& 
\Q[t]/(f_j^{n_j+1}),\\
F^{(j)}&\cong& \Q[t]/(f_j),
\end{eqnarray*}
Then $\varphi_j:=f_j(a 1_{A^{(j)}})$ is in $N^{(j)}$
and gives rise to the $F^{(j)}$-basis $1_{A^{(j)}},
\varphi_j,\varphi_j^2,...,\varphi_j^{n_j}$ of $A^{(j)}$ with
$$(N^{(j)})^l=\bigoplus_{i=l}^{n_j}F^{(j)}\varphi_j^i$$
for $l\in\{1,...,n_j\}$. So each quotient
$(N^{(j)})^l/(N^{(j)})^{l+1}$ is a 1-dimensional
$F^{(j)}$-vector space. This holds for each $j$ only if
$A$ is cyclic. 
\item[(vii)] 
The cyclic algebra $A\cong\Q[t]/(f)$ in part (vi) 
with $f\in\Z[t]$ unitary will be relevant for the 
conjugacy classes of regular 
integer matrices with characteristic polynomial $f$. 
See section \ref{s6}. 
\end{enumerate}
\end{remarks}

\section{Orders, full lattices and their semigroups}\label{s4}
\setcounter{equation}{0}
\setcounter{table}{0}
\setcounter{figure}{0}

Throughout the paper, $A$ is an algebra as in Theorem \ref{t3.1},
so a finite dimensional commutative $\Q$-algebra with 
unit element $1_A$. 
In this section full lattices and orders in $A$ and the 
semigroup $\LL(A)$ of full lattices in $A$ 
will be defined and discussed.

\begin{definition}\label{t4.1}
\begin{enumerate}
\item[(a)] 
Let $V$ be a finite dimensional $\Q$-vector space.
A full lattice in $V$ is a finitely generated $\Z$-module
in $V$ which generates $V$ over $\Q$. The set of full lattices
in $V$ is called $\LL(V)$.
\item[(b)] 
An order in $A$ is a full lattice $\Lambda$ in $A$ with
\begin{eqnarray}\label{4.1}
1_A\in\Lambda\quad\textup{ and }\quad \Lambda\cdot\Lambda
\subset\Lambda,
\end{eqnarray}
so it is a full lattice in $A$ and a subring of $A$.
Automatically $\Lambda\cdot\Lambda=\Lambda$.
\end{enumerate}
\end{definition}

The following basic facts are elementary. Proofs are offered
for example in \cite{HL26}.

\begin{lemma}\label{t4.2}
(E.g. \cite[5.2, 5.3, 5.5, 8.1]{HL26})

\begin{enumerate}
\item[(a)] 
Let $L,L_1$ and $L_2$ be full lattices in $A$. Then
\begin{eqnarray*}
L_1\cdot L_2&:=& \{\sum_{i\in I}a_ib_i\,|\, I\textup{ a finite 
index set}, a_i\in L_1,b_i\in L_2\}\textup{ and}\\
L_1:L_2 &:=&\{a\in A\,|\, a\cdot L_2\subset L_1\}
\end{eqnarray*}
are full lattices in $A$, and 
\begin{eqnarray*}
\OO(L):= L:L
\end{eqnarray*}
is an order in $A$. It is called {\sf the order of }$L$.
\item[(b)] 
With this multiplication of full lattices, the set $\LL(A)$
becomes a commutative semigroup.  It also comes equipped 
with the division map with $(L_1,L_2)\mapsto L_1:L_2$.
\item[(c)] 
The product $\Lambda_1\Lambda_2$ of two orders $\Lambda_1$
and $\Lambda_2$ is the smallest order which contains 
$\Lambda_1$ and $\Lambda_2$. Therefore the set
$\{\textup{orders in }A\}$ is a subsemigroup of $\LL(A)$. 
But the quotient $\Lambda_1:\Lambda_2$ is not an order if 
$\Lambda_2\not\subset \Lambda_1$ because then 
$1_A\notin \Lambda_1:\Lambda_2$. 
\item[(d)] 
Let $L_1,L_2\in\LL(A)$. Then
\begin{eqnarray}\label{4.2}
\OO(L_1L_2)\supset\OO(L_1)\OO(L_2),\quad 
\OO(L_1:L_2)\supset\OO(L_1)\OO(L_2).
\end{eqnarray}
\end{enumerate}
\end{lemma}

As $\LL(A)$ is a commutative semigroup, we have by
Definition \ref{t2.1} the notions of idempotents
and of invertible elements in $\LL(A)$. 
Theorem \ref{t4.3} discusses them.
It is cited from \cite{HL26}. But the result there is a
condensed version of results in \cite{DTZ62} with one piece
from \cite{Fa65-1} (the implication (iv)$\Rightarrow$(i)
in Theorem \ref{t4.3} (b)).

\begin{theorem}\label{t4.3}\cite[1.3.3, 1.3.6]{DTZ62}
\cite[26.4]{Fa65-1}\cite[5.6]{HL26}

\begin{enumerate}
\item[(a)] 
$\Lambda\in\LL(A)$ is an idempotent in the semigroup
$\LL(A)$ if and only if it is an order.
\item[(b)] 
Let $L\in\LL(A)$. The following four properties 
are equivalent:
\begin{enumerate}
\item[(i)] 
$L$ is invertible in the semigroup $\LL(A)$.
\item[(ii)] 
$L_2\in\LL(A)$ with $L\cdot L_2=\OO(L)$ exists.
\item[(iii)] 
$L\cdot (\OO(L):L)=\OO(L)$. 
\item[(iv)]
$\OO(\OO(L):L)=\OO(L)$. 
\end{enumerate}
If this holds then $e_L=\OO(L)$, $L^{-1}=\OO(L):L$, and
$L^{-1}$ is invertible with $\OO(L^{-1})=\OO(L)$.
\end{enumerate}
\end{theorem}

By Theorem \ref{t2.2} (f), the subset
$$\bigcup_{\Lambda\textup{ order}}G(\Lambda)\subset\LL(A)$$
of invertible elements in $\LL(A)$ is a subsemigroup of
$\LL(A)$ (with $\OO(L_1L_2)=\OO(L_1)\OO(L_2)$ by Theorem 
\ref{t2.2} (d)).
It is much easier to control than the multiplication of
not invertible elements. 

The only strong result for not invertible elements, 
which we know, is Theorem \ref{t4.4}. It was proved for $A$ an
algebraic number field in \cite{DTZ62}, and this is the 
central result of \cite{DTZ62}. A different proof for
$A$ separable was given in \cite{Si70}.
The proof in \cite{DTZ62} was adapted to the case of an
arbitrary $A$ in \cite{HL26}.

\begin{theorem}\label{t4.4} 
\cite[1.5.5]{DTZ62} \cite[Theorem 2]{Si70} 
\cite[Theorem 9.1]{HL26}
Suppose that $A$ has dimension $n\in\Z_{\geq 2}$.
For each full lattice $L\in\LL(A)$ each power $L^k$ with 
$k\geq n-1$ is invertible.
\end{theorem}

A result which is used in the proof of Theorem \ref{t4.4}
in \cite{HL26}, but which is also of independent interest,
is the following.

\begin{theorem}\label{t4.5} \cite[Theorem 9.2]{HL26}
Let $L\in\LL(A)$ be a full lattice. An order 
$\Lambda\supset\OO(L)$ with $\Lambda L\in G(\Lambda)$
(so $\OO(\Lambda L)=\Lambda$ and $\Lambda L$ invertible) exists.
\end{theorem}

Another result which is used in the proof of Theorem \ref{t4.4}
in \cite{HL26} is the observation in Lemma \ref{t4.6}.  
\cite[1.4]{DTZ62} gives the example \ref{t4.7} and ascribes 
the example and the observation to Dedekind \cite{De72}.

\begin{lemma}\label{t4.6}
Let $\Lambda$ be an order, and let $L$ be a full lattice
which is not an order, but which satisfies
$1_A\in L$ and $L\subset\Lambda$. Then there is a unique
$l\in\Z_{\geq 2}$ with
$$L\subsetneqq L^2\subsetneqq ...\subsetneqq L^l =L^{l+k}
\subset\Lambda \textup{ for }k\geq 1.$$
$L^l$ is an order with $L\cdot L^l=L^l$.
The full lattices $L, L^2,...,L^{l-1}$ are not invertible.
\end{lemma}

{\bf Proof:}
$1_A\in L$ implies $L^{m+1}=L\cdot L^m\supset L^m$ for each
$m\in\N$. As each power $L^m$ is contained in $\Lambda$,
there is a minimal $l\in\N$ with $L^{l+1}=L^l$. Then
$L^{l+2}=L\cdot L^{l+1}=L\cdot L^l=L^{l+1}=L^l$, and
inductively $L^{l+k}=L^l$ for $k\geq 1$. 
Then $L^l\cdot L^l=L^{2l}=L^l$ and $1_A\in L^l$, 
so $L^l$ is an order.

If for some $j\in\{1,...,l-1\}$ $L^j$ were invertible, then
$\OO(L^j)=\OO(L^jL^j...L^j)=\OO(L^{jm})=\OO(L^l)=L^l$
for $m\in\N$ with $jm\geq l$, so
$L^j=L^j\cdot\OO(L^j)=L^{j+l}=L^l$, a contradiction.
\hfill$\Box$

\begin{example}\label{t4.7} \cite[1.4]{DTZ62}
Let $\dim A=n\geq 3$ and suppose that 
$\Lambda=\Z[a]\subset A$ is an order. Define
$$L:=\langle 1_A,a,2a^2,...,2a^{n-1}\rangle_\Z\subset\Lambda.$$
Then
\begin{eqnarray*}
L^2&=&  \langle 1_A,a,a^2,2a^3...,2a^{n-1}\rangle_\Z,\\
L^3&=&  \langle 1_A,a,a^2,a^3,2a^4,...,2a^{n-1}\rangle_\Z,\\
&\vdots& \\
L^{n-1}&=&  \langle 1_A,a,a^2,...,a^{n-1}\rangle_\Z=\Lambda,\\
L&\subsetneqq& L^2\subsetneqq L^3\subsetneqq...\subsetneqq
L^{n-1}=L^n=...=\Lambda,
\end{eqnarray*}
and $L$, $L^2$, ..., $L^{n-2}$ are not invertible
by Lemma \ref{t4.6}. 
\end{example}

The following result is well known (e.g. \cite{Ne99}
or \cite{BSh73}) if $A$ is an algebraic number field.
For separable $A$ it follows easily, using the
decomposition $A=\bigoplus_{j=1}^k A^{(j)}$ into 
algebraic number fields $A^{(j)}$. 

\begin{theorem}\label{t4.8}
Let $A$ be separable. 

\begin{enumerate}
\item[(a)] 
Then there is a maximal order
$\Lambda_{max}$ in $A$. 
\item[(b)] 
Each full lattice $L$ with
$\OO(L)=\Lambda_{max}$ is invertible.
\end{enumerate}
\end{theorem}

{\bf Proof:} $\Lambda_{max}:=\bigoplus_{j=1}^k
\Lambda_{max}(A^{(j)})$ is an order in $A$.
If $\Lambda$ is an order in $A$,
then $\Lambda\cdot 1_{A^{(j)}}$ is an order in $A^{(j)}$,
so $\Lambda\cdot 1_{A^{(j)}}\subset \Lambda_{max}(A^{(j)})$,
so $\Lambda\subset\Lambda_{max}$. 
Therefore $\Lambda_{max}$ is the maximal order in $A$.

If $L$ is a full lattice with $\OO(L)=\Lambda_{max}$,
then $L\cdot 1_{A^{(j)}}$ is a full lattice in $A^{(j)}$
with $\OO(L\cdot 1_{A^{(j)}})=\Lambda_{max}(A^{(j)})$, so
it is invertible in $A^{(j)}$. We have 
$L=\bigoplus_{j=1}^k L\cdot 1_{A^{(j)}}$ 
because $1_{A^{(j)}}\in\Lambda_{max}$. Therefore $L$ 
is invertible in $A$.
\hfill$\Box$

\bigskip
Borevi\v{c} and Faddeev \cite{BF65} proved the following
generalization of Theorem \ref{t4.8} (b). It shows that each
full lattice whose order is big enough is invertible
if $A$ is separable.

\begin{theorem}\label{t4.9} \cite[Theorem 1]{BF65}
Let $A$ be separable. Let $\Lambda$ be an order such
that an element $\omega\in \Lambda_{max}$ with 
$\Lambda_{max}=\Lambda+\omega\Lambda$ exists.
Then each $L\in\LL(A)$ with $\OO(L)\supset\Lambda$ is
invertible.
\end{theorem}

Theorem \ref{t4.12} below gives another important result
about invertibility of full lattices. Its proof uses
Lemma \ref{t4.11} and Theorem \ref{t4.14}.

\begin{definition}\label{t4.10}
\begin{enumerate}
\item[(a)] 
An order $\Lambda$ in $A$ is {\it cyclic} if
an $a\in \Lambda$ with $\Lambda=\Z[a]$ exists. 
\item[(b)] Let $\phi:A\times A\to $ be a multiplicative metric.
For $L\in\LL(A)$ denote by $L^\phi$ the full lattice
\begin{eqnarray}\label{4.3}
L^\phi:=\{b\in A\,|\, 
\phi(a,b)\in\Z\ \forall\ a\in L\}.
\end{eqnarray}
\end{enumerate}
\end{definition}

The following properties are elementary.

\begin{lemma}\label{t4.11} 
\begin{enumerate}
\item[(a)] (i) If $\Lambda$ is a cyclic order $\Lambda=\Z[a]$ in $A$ 
then also $A$ is cyclic, namely $A=\Q[a]$.
(ii) Vice versa, any cyclic algebra $A$ contains cyclic
orders.
\item[(b)] 
If $\Lambda=\Z[a]$ is a cyclic order in $A$, then
the multiplicative metric of $A=\Q[a]$ which is induced
via \eqref{3.4} by the linear form $l:A\to\Q$ in
Remark \ref{t3.4} (ii) satisfies $\Lambda^\phi=\Lambda$. 
\end{enumerate}
\end{lemma}

\begin{theorem}\label{t4.12}
Let $\Lambda$ be a cyclic order in $A$. Then each $L\in\LL(A)$
with $\OO(L)=\Lambda$ is invertible.
\end{theorem}

{\bf Proof:} In the case of separable $A$, Theorem \ref{t4.12}
is due to Faddeev \cite[26.7]{Fa65-1}.
The general case follows from Lemma \ref{t4.11} (b)
and Theorem \ref{t4.14} (ii)$\Rightarrow$(i).
\hfill$\Box$

\bigskip
If $A$ is cyclic, many orders in $A$ are cyclic, but
many others are not cyclic. Especially, above a cyclic
order $\Lambda$ often there are other orders which are not 
cyclic. Then all $L$ with $\OO(L)=\Lambda$ are invertible, but 
in general not all $L$ with $\OO(L)\supset\Lambda$
are invertible.

The proof of Theorem \ref{t4.14} needs the following basic
properties of the map $L\mapsto L^\phi$ if $\phi:A\times A\to\Q$
is a multiplicative metric. They are formulated in much
more generality in \cite{Fa65-1}. 
\eqref{4.4}, \eqref{4.5} and \eqref{4.6} are elementary.
Property \eqref{4.7} can also be derived from
\cite[Lemma 2.3 (b)]{HL26}. \eqref{4.7} and \eqref{4.4}
imply \eqref{4.8} and \eqref{4.9}.

\begin{lemma}\label{t4.13} \cite[$8^o$]{Fa65-1}
Let $\phi:A\times A\to\Q$ be a multiplicative metric. 
Fix $L_1,L_2\in\LL(A)$. Then
\begin{eqnarray}\label{4.4}
L_1^{\phi\phi}&=&L_1,\\
(L_1+L_2)^\phi&=& L_1^\phi\cap L_2^\phi,\label{4.5}\\
(L_1\cap L_2)^\phi&=& L_1^\phi+L_2^\phi,\label{4.6}\\
(L_1L_2)^\phi&=& L_1^\phi:L_2 (=L_2^\phi:L_1),\label{4.7}\\
\OO(L_1^\phi)&=&\OO(L_1),\label{4.8}\\
L_1L_1^\phi&=& \OO(L_1)^\phi.\label{4.9}
\end{eqnarray}
\end{lemma}

Theorem \ref{t4.14} for $A$ separable was proved by Faddeev
((i)$\iff$(ii): \cite[24.1]{Fa65-1}, (i)$\iff$(iii): 
\cite[26.6]{Fa65-1}).

\begin{theorem}\label{t4.14} 
Let $A$ have a multiplicative metric $\phi:A\times A\to A$. 
Let $\Lambda\in\LL(A)$ be an order. 
The conditions (i), (ii) and (iii) are equivalent.
\begin{enumerate}
\item[(i)]
Each $L\in\LL(A)$ with $\OO(L)=\Lambda$ is invertible.
\item[(ii)]
$\Lambda^\phi$ is invertible.
\item[(iii)]
$\OO((\Lambda^\phi)^2)=\Lambda.$
\end{enumerate}
\end{theorem}

{\bf Proof:} (i)$\Rightarrow$(ii):
$\OO(\Lambda^\phi)\stackrel{\textup{\eqref{4.8}}}{=}\OO(\Lambda)
=\Lambda$, so by (i) $\Lambda^\phi$ is invertible.

(ii)$\Rightarrow$(i): Consider $L\in\LL(A)$ with 
$\OO(L)=\Lambda$. 
$$L(L^\phi(\Lambda^\phi)^{-1})
\stackrel{\textup{\eqref{4.9}}}{=}
\Lambda^\phi(\Lambda^\phi)^{-1}
=\OO(\Lambda^\phi)\stackrel{\textup{\eqref{4.8}}}{=}
\OO(\Lambda)=\Lambda=\OO(L),$$
so $L$ is invertible. 

(ii)$\iff$(iii): Consider $L_1:=\Lambda^\phi$.
\begin{eqnarray*}
\OO(L_1)=\OO(\Lambda^\phi)\stackrel{\textup{\eqref{4.8}}}{=}
\OO(\Lambda)=\Lambda,\
\textup{and thus}\quad \OO(L_1)^\phi=\Lambda^\phi,
\end{eqnarray*}
so
\begin{eqnarray*}\OO(\Lambda^\phi\Lambda^\phi)
=\OO(\OO(L_1)^\phi L_1)
\stackrel{\textup{\eqref{4.7}}}{=}
\OO((\OO(L_1):L_1)^\phi)
\stackrel{\textup{\eqref{4.8}}}{=}\OO(\OO(L_1):L_1).
\end{eqnarray*}
Theorem \ref{t4.3} (b) (i)$\iff$(iv) gives the equivalence
(ii)$\iff$(iii).
\hfill$\Box$ 

\bigskip
The case $\dim A=2$ is special and famous.
In that case each full lattice is invertible. We can offer now
four different proofs of this fact, three by results
which are cited in this section, the fourth one consists
of direct proofs for the different cases.

\begin{corollary}\label{t4.15}
Let $\dim A=2$. Then each full lattice $L$ is invertible.
\end{corollary}

{\bf Four proofs:} (1) Apply Theorem \ref{t4.4}.

(2) There is an $\omega\in \Lambda_{max}$ with 
$\Lambda_{max}=\Z 1_A+\Z\omega$. Therefore 
$\Lambda_{max}=\Lambda+\Lambda\omega$ for each order $\Lambda$
in $A$. Apply Theorem \ref{t4.9}.

(3) Each order in $A$ is cyclic. Apply Theorem \ref{t4.12}.

(4) For $A$ an algebraic number field, the corollary
is also proved in \cite[1.4.1]{DTZ62}.
For $A=\Q e_1\oplus\Q e_2$ it is part of Theorem \ref{t9.1} (d).
For $A\cong \Q[x]/(x^2)$ it is part of Theorem \ref{t10.3} (b).
\hfill$\Box$

\section{Classes of full lattices, two quotient
semigroups}\label{s5}
\setcounter{equation}{0}
\setcounter{table}{0}
\setcounter{figure}{0}

Throughout the paper, $A$ is an algebra as in Theorem \ref{t3.1},
so a finite dimensional commutative $\Q$-algebra with unit
element $1_A$. 

In this section the quotient semigroup $W(\LL(A))$ from
Theorem \ref{t2.4} and another quotient semigroup $\EE(A)$
are considered. Results from \cite{HL26} are cited.
Most of them are extensions of results from Faddeev,
\cite{DTZ62} and \cite{Ne99}.

Recall that (by Notation \ref{t1.2} (iii)) $A^{*}$ 
is the group of units in $A$. 
The following $\varepsilon$-equivalence of full lattices in $A$ 
is needed for $\GL_n(\Z)$-conjugacy classes 
of regular integer matrices in section \ref{s6}.

\begin{definition}\label{t5.1}
\begin{enumerate}
\item[(a)] 
An equivalence relation $\sim_\varepsilon$ on $\LL(A)$
is defined as follows,
\begin{eqnarray*}
L_1\sim_\varepsilon L_2 &:\iff& a\in A^{*}
\textup{ with }a\cdot L_1=L_2\textup{ exists.}
\end{eqnarray*}
The equivalence class of a full lattice $L$ with respect to
$\sim_\varepsilon$ is called $\varepsilon$-class of $L$ and
is denoted $[L]_\varepsilon$. The set of $\varepsilon$-classes
of full lattices is denoted 
$\EE(A) \ (=\LL(A)/\sim_\varepsilon)$. 
\item[(b)] 
A full lattice $L$ in $A$ with $\OO(L)\supset\Lambda$
for some order $\Lambda$ is called a {\it $\Lambda$-ideal}.
It is called an {\it exact $\Lambda$-ideal} if $\OO(L)=\Lambda$.
\end{enumerate}
\end{definition}

It is easy to see that $\varepsilon$-equivalence is
compatible with multiplication and division in
$\LL(A)$. This and the other statements in Lemma \ref{t5.2}
are discussed in \cite{DTZ62} and \cite{HL26}.

\begin{lemma}\label{t5.2} 
\cite[1.3.7]{DTZ62}\cite[5.5, 5.6]{HL26}

\begin{enumerate}
\item[(a)] 
Multiplication and division map on $\LL(A)$ induce
multiplication and division map on the quotient $\EE(A)$.
So $\EE(A)$ is a commutative semigroup with a division map.
\item[(b)] 
$[L]_\varepsilon$ is idempotent if and only if the class
$[L]_\varepsilon$ contains an order.
\item[(c)] 
$[L]_\varepsilon$ is invertible if and only if $L$ is
invertible. Then $e_{[L]_\varepsilon}=[\OO(L)]_\varepsilon$
and $[L]_\varepsilon^{-1}=[L^{-1}]_\varepsilon$. 
\end{enumerate}
\end{lemma}

Theorem \ref{t5.3} is an important finiteness result.
The special case where the radical $R$ of $A$ satisfies $R^2=0$ 
was proved in \cite{Fa67}, the general case is proved in 
\cite[Theorem 6.5]{HL26}.
Corollary \ref{t5.4} gives an implication for $A$ separable
which also follows from the Jordan-Zassenhaus theorem
\cite{Za38} (see also \cite[(79.1)]{CR62} or 
\cite[(26.4)]{Re03}). 
In this sense, Theorem \ref{t5.3} generalizes the
Jordan-Zassenhaus theorem.

\begin{theorem}\label{t5.3} \cite{Fa67} \cite[6.5]{HL26}
For any order $\Lambda$ in $A$, the set
$$\{[L]_\varepsilon\,|\, L\in\LL(A),\OO(L)= \Lambda\}$$
of $\varepsilon$-classes of exact $\Lambda$-ideals is finite.
\end{theorem}

\begin{corollary}\label{t5.4}
Let $A$ be separable.
For any order $\Lambda$ in $A$, the set
$$\{[L]_\varepsilon\,|\, L\in\LL(A),\OO(L)\supset \Lambda\}$$
of $\varepsilon$-classes of $\Lambda$-ideals is finite.
\end{corollary}

{\bf Proof:} Because the maximal order $\Lambda_{max}$ exists,
there are only finitely many orders $\www{\Lambda}$ with
$\www{\Lambda}\supset\Lambda$. One applies Theorem \ref{t5.3}.
\hfill$\Box$

\bigskip
The $w$-equivalence of full lattices can be rephrased in
different useful ways. It helps to structure the set
$\{L\,|\, L\in\LL(A),\OO(L)=\Lambda\}$ and the finite set 
$\{[L]_\varepsilon\,|\, L\in\LL(A),\OO(L)= \Lambda\}$
for a given order $\Lambda$.
The parts (a)--(c) of the following theorem are cited
from \cite{HL26}, but they are condensed versions of 
results in \cite{DTZ62}. Though the crucial part (d)
is not formulated in \cite{DTZ62}. It is proved in 
\cite[5.8]{HL26}.

\begin{theorem}\label{t5.5}
\cite[1.3.9, 1.3.11]{DTZ62}\cite[5.7, 5.8]{HL26}

\begin{enumerate}
\item[(a)] 
Let $L_1,L_2\in\LL(A)$. The following four properties
are equivalent.
\begin{enumerate}
\item[(i)]
$L_1\sim_w L_2$.
\item[(ii)]
$1_A\in (L_1:L_2)(L_2:L_1)$.
\item[(iii)]
$\OO(L_1)=\OO(L_2)$ and $L_3\in G(\OO(L_1))$ with
$L_1L_3=L_2$ exists.
\item[(iv)]
$[L_1]_\varepsilon\sim_w[L_2]_\varepsilon$.
\end{enumerate}
\item[(b)] Let $L_1,L_2\in\LL(A)$ with $L_1\sim_w L_2$. Then
\begin{eqnarray*}
L_1:L_2\in G(\OO(L_1)),\quad L_2:L_1\in G(\OO(L_1)),
\quad (L_1:L_2)^{-1}=L_2:L_1,\\
L_2=(L_2:L_1)L_1,\quad L_1=(L_1:L_2)L_2,\\
\OO(L_1)=\OO(L_2)=\OO(L_1:L_2)=\OO(L_2:L_1)
=(L_1:L_2)(L_2:L_1).
\end{eqnarray*}
The only full lattice $L_3\in G(\OO(L_1))$ with $L_1L_3=L_2$
is $L_3=L_2:L_1$. 
\item[(c)] 
$W(\LL(A))=W(\EE(A))$, and this semigroup
inherits a division map from the division map on $\LL(A)$. 
\item[(d)] 
Let $\Lambda$ be an order and $L_1\in \LL(A)$
with $\OO(L_1)=\Lambda$. Then 
\begin{eqnarray}
G(\Lambda)=[\Lambda]_w\quad\textup{and}\quad 
G([\Lambda]_\varepsilon)=[[\Lambda]_\varepsilon]_w.\label{5.1}
\end{eqnarray}
The maps 
\begin{eqnarray}
G(\Lambda)\to [L_1]_w,&& L_3\mapsto L_3L_1,\label{5.2}\\
G([\Lambda]_\varepsilon)\to [[L_1]_\varepsilon]_w,
&& [L_3]_\varepsilon\mapsto [L_3]_\varepsilon[L_1]_\varepsilon,
\label{5.3}
\end{eqnarray}
are well defined and bijections. 
\end{enumerate}
\end{theorem}

\begin{remarks}\label{t5.6}
\begin{enumerate}
\item[(i)] 
Let $\Lambda$ be an order. By Theorem \ref{t5.3} the set
$\{[L]_\varepsilon\,|\, L\in\LL(A),\OO(L)=\Lambda\}$ is finite. 
By \eqref{5.3} it splits
into finitely many $w$-classes. 
One of them is the (thus finite) group 
$G([\Lambda]_\varepsilon)$. 
These $w$-classes have all the same size,
because \eqref{5.3} gives bijections from the finite group
$G([\Lambda]_\varepsilon)$ to any other $w$-class in
the finite set 
$\{[L]_\varepsilon\,|\, L\in\LL(A),\OO(L)=\Lambda\}$.
Their number is denoted by 
\begin{eqnarray}\nonumber
\tau(\Lambda)&:=& |\{[[L]_\varepsilon]_w
\,|\, L\in\LL(A),\OO(L)=\Lambda\}|\\
&=& |\{[L]_w\,|\, L\in\LL(A),\OO(L)=\Lambda\}|\in\N.\label{5.4}
\end{eqnarray}
So $\tau(\Lambda)$ is also the number of $w$-classes in 
$\{L\,|\, L\in\LL(A),\OO(L)=\Lambda\}$. One of them is 
$G(\Lambda)=[\Lambda]_w$. \eqref{5.2} gives bijections
from it to any other $w$-class in 
$\{L\,|\, L\in\LL(A),\OO(L)=\Lambda\}$.
\item[(ii)] 
Consider the inclusions
\begin{eqnarray}\label{5.5}
\bigcup_{\Lambda\textup{ order}}G(\Lambda)&\subset& \LL(A),\\
\bigcup_{\Lambda\textup{ order}}G([\Lambda]_\varepsilon)
&\subset& \EE(A),\label{5.6} \\
\{[\Lambda]_w\,|\, \Lambda\textup{ an order}\}
&\subset& W(\LL(A)). \label{5.7}
\end{eqnarray}
On the left there are the sets of invertible elements of the 
semigroups on the right. The sets on the left are subsemigroups
of the semigroups on the right, but in general the division maps
do not restrict to the sets on the left (of course they do 
if the sets on the left and on the right coincide). 
\item[(iii)] 
The Theorems \ref{t5.10} and \ref{t5.12} below give 
some control on the subsemigroup 
$\bigcup_{\Lambda\textup{ order}}
G([\Lambda]_\varepsilon)\subset \EE(A)$ 
on the left of \eqref{5.6}. 
\item[(iv)] 
But little is known about the structure of the
whole semigroups $\EE(A)$ and $W(\LL(A))$ if they contain
not invertible elements. 
The strongest result which we know is Theorem \ref{t4.4}. 
Already the finite number $\tau(\Lambda)\in\N$
for an order $\Lambda$ is difficult to determine in general. 

Theorem \ref{t5.7} of Faddeev \cite[Corollary 4.1]{Fa65-2}
gives a formula for $\tau(\Lambda)$ if $\dim A=3$ and $A$
is separable. It is extended to inseparable $A$
of dimension 3 in Theorem \ref{t10.4} (f), 
Theorem \ref{t11.2} (b) and Lemma \ref{t12.1}, see
Theorem \ref{t5.8}.
\end{enumerate}
\end{remarks}

\begin{theorem}\label{t5.7}
\cite[Corollary 4.1]{Fa65-2}
Let $A$ be separable of dimension 3.
Let $\Lambda$ be an order in $A$ with a $\Z$-basis
$(1_A,\www{\omega}_1,\www{\omega}_2)$ with
$\www{\omega}_1\www{\omega}_2
=k_1\www{\omega}_1+k_2\www{\omega}_2+k_3 1_A$
with $k_1,k_2,k_3\in\Z$. 
Then the $\Z$-basis $(1_A,\omega_1,\omega_2)
=(1_A,\www{\omega}_1-k_2 1_A,\www{\omega}_2-k_1 1_A)$ satisfies
$\omega_1\omega_2\in\Z 1_A$.
There are $a,b,c,d\in\Z$, which are not all 0, with
\begin{eqnarray*}
\omega_1^2&=& b\omega_1+a\omega_2-ac 1_A,\\
\omega_1\omega_2&=& ad 1_A,\\
\omega_2^2&=& d\omega_1+c\omega_2-bd 1_A.
\end{eqnarray*}
Define
\begin{eqnarray*}
\mu&:=& \gcd(a,b,c,d)\in\N,\\
t&:=& |\{p\in\P\,|\, p\textup{ divides }\mu\}|\in\N.
\end{eqnarray*}
The number $\tau(\Lambda)$ of $w$-classes of full lattices
$L$ with $\OO(L)=\Lambda$ is $\tau(\Lambda)=2^t$.
\end{theorem}

The last part of this theorem generalizes also to the algebras
of dimension 3 which are inseparable. This is a consequence
of the results on these algebras in this paper.

\begin{theorem}\label{t5.8}
Let $A$ have dimension 3. Let $\Lambda$ be an order in $A$.
A number $\mu(\Lambda)\in\N$ exists which satisfies
\begin{eqnarray}\label{5.8}
\tau(\Lambda)&=& 2^{t(\Lambda)}\quad\textup{where}\\
t(\Lambda)&=& |\{p\in\P\,|\, p\textup{ divides }\mu(\Lambda)\}|.
\label{5.9}
\end{eqnarray}
\end{theorem}

{\bf Proof:} For separable $A$ this is part of Theorem \ref{5.7}.
There are only three 3-dimensional algebras which are not
separable. 
For $A\cong \Q[x]/(x^3)$ see Lemma \ref{t10.1} and 
Theorem \ref{t10.4} (f), where $\mu(\Lambda):=\Delta(\Lambda)$. 
For $A=A^{(1)}\oplus A^{(2)}$ with $A^{(1)}\cong \Q[x]/(x^2)$
and $A^{(2)}\cong\Q$ see Theorem \ref{t11.2} (b)(i).
For $A\cong\Q[x,y]/(x^2,xy,y^2)$ \eqref{5.8} and \eqref{5.9}
hold with $(\mu(\Lambda),t(\Lambda),\tau(\Lambda))=(1,0,1)$ 
by Lemma \ref{t12.1}.
\hfill$\Box$

\bigskip
By Theorem \ref{t3.1} the $\Q$-algebra $A$ 
decomposes canonically
as $A=F\oplus R$, so into the separable subalgebra $F$,
which is a sum of algebraic number fields, and the radical $R$,
which consists of all nilpotent elements. The decomposition
induces the projection
\begin{eqnarray}\label{5.10}
\pr_F:A\to F.
\end{eqnarray}
Lemma \ref{t5.9} and Theorem \ref{t5.10} make statements about
this projection.

\begin{lemma}\label{t5.9} \cite[Lemma 9.1]{HL26}
\begin{enumerate}
\item[(a)] 
The image $\pr_FL$ of a full lattice in $A$ is a full
lattice in $F$. The projection $\pr_F:A\to F$ is compatible 
with the multiplication in $\LL(A)$ and $\LL(F)$, 
\begin{eqnarray}
\pr_F(L_1\cdot L_2)=\pr_F L_1\cdot \pr_F L_2\quad
\textup{for }L_1,L_2\in\LL(A).\label{5.11}
\end{eqnarray}
It induces a surjective homomorphism 
$\pr_F:\LL(A)\to\LL(F)$ of semigroups. 
Especially, it maps idempotents to idempotents,
so orders to orders. So it restricts to a surjective
homomorphism
\begin{eqnarray}\label{5.12}
\pr_F:\{\textup{orders in }A\}\to\{\textup{orders in }F\}
\end{eqnarray}
of semigroups. 
Though it does not respect the division maps. In general 
we only have 
\begin{eqnarray}\label{5.13}
\pr_F(L_1: L_2)\subset\pr_F L_1:\pr_FL_2\quad
\textup{for }L_1,L_2\in\LL(A).
\end{eqnarray}
\item[(b)] 
The projection $\pr_F:\LL(A)\to\LL(F)$ respects
$\varepsilon$-equivalence and $w$-equivalence. Therefore it 
induces surjective homomorphisms
\begin{eqnarray}\label{5.14}
\pr_F:\EE(A)&\to& \EE(F),\\
\pr_F:W(\LL(A))&\to& W(\LL(A))\label{5.15}
\end{eqnarray}
of semigroups. 
\item[(c)] 
In general, for $L\in\LL(A)$ we only have an inclusion
$\pr_F\OO(L)\subset\OO(\pr_FL)$. But if $L$ is invertible
then $\pr_FL$ is invertible with 
$\OO(\pr_FL)=\pr_F(\OO(L))$ and $(\pr_FL)^{-1}=\pr_F(L^{-1})$.
\end{enumerate}
\end{lemma}

\bigskip
The following theorem is surprising. It says especially
that the finite group $G([\Lambda]_\varepsilon)$ for an order 
$\Lambda$ in $A$ is isomorphic to the finite group
$G([\Lambda_0]_\varepsilon)$ for the induced order
$\Lambda_0=\pr_F\Lambda$ in the separable part $F$ of $A$. 

The more difficult part of Theorem \ref{t5.10}, namely the
injectivity of the map in \eqref{5.16}, is due to 
Faddeev \cite[Theorem 3]{Fa68}. The proofs of injectivity
and surjectivity use the localization in 
\cite[section 7]{HL26}.

\begin{theorem}\label{t5.10} 
\cite[Theorem 3]{Fa68} \cite[10.2]{HL26}
Let $\Lambda$ be an order in $A$. Write 
$\Lambda_0:=\pr_F\Lambda$ for the induced order in $F$.
Part (c) in Lemma \ref{t5.10} 
gives rise to two group homomorphisms,
\begin{eqnarray}\label{5.16}
G(\Lambda)&\to& G(\Lambda_0),\quad L\mapsto \pr_FL,\\
G([\Lambda]_\varepsilon)&\to& G([\Lambda_0]_\varepsilon),\quad
[L]_\varepsilon\mapsto [\pr_FL]_\varepsilon.\label{5.17}
\end{eqnarray}
The group homomorphism in \eqref{5.16} is surjective.
The group homomorphism in \eqref{5.17} is an isomorphism.
\end{theorem}

\begin{remarks}\label{t5.11}
\begin{enumerate}
\item[(i)] 
If $R\neq \{0\}$, then there are infinitely many 
orders in $A$ which project to the same order in $F$.
Let $\Lambda$ be one order in $A$. The following is a
construction of an infinite sequence $(\Lambda_m)_{m\in\N}$
of orders $\Lambda_m$ 
with $\pr_F\Lambda_m=\pr_F\Lambda$ and 
$\Lambda\subsetneqq\Lambda_1\subsetneqq ...\subsetneqq
\Lambda_m\subsetneqq\Lambda_{m+1}\subsetneqq ...$:
\begin{eqnarray}\label{5.18}
\Lambda_m:=\Lambda+\sum_{l=1}^{n_{max}}2^{-lm}(\Lambda\cap R)^l.
\end{eqnarray}
Here $n_{max}=\max(l\in\N\,|\, R^l\neq \{0\})$. The construction
is due to \cite[Proposition 25.1]{Fa65-1}, and it is recalled
in \cite[6.3]{HL26}.
\item[(ii)] 
By Theorem \ref{t5.10} for each two orders $\Lambda$
and $\www{\Lambda}$ with $\pr_F\Lambda=\pr_F\www{\Lambda}$,
the finite groups $G([\Lambda]_\varepsilon)$ and
$G([\www{\Lambda}]_\varepsilon)$ are canonically isomorphic
to the group $G([\pr_F\Lambda]_\varepsilon)$ and thus also
to one another. 
Therefore in order to determine their size and structure, 
one can restrict to the group $G([\pr_F\Lambda]_\varepsilon)$. 
But the numbers $\tau(\Lambda)$, $\tau(\www{\Lambda})$ and
$\tau(\pr_F\Lambda)$ are in general not equal.
\end{enumerate}
\end{remarks}

The following theorem compares the groups 
$G(\Lambda_1)$ and $G(\Lambda_2)$ respectively
$G([\Lambda_1]_\varepsilon)$ and $G([\Lambda_2]_\varepsilon)$
for two orders $\Lambda_1$ and $\Lambda_2$ in $A$ with
$\Lambda_1\supset\Lambda_2$. It turns out that there are
natural surjective group homomorphisms
$G(\Lambda_2)\to G(\Lambda_1)$ and
$G([\Lambda_2]_\varepsilon)\to G([\Lambda_1]_\varepsilon)$.
Theorem \ref{t5.12} embeds them into exact sequences.
In the case when $A$ is an algebraic number field,
this is proved in \cite[Ch. I \S 12]{Ne99}, using localization
by prime ideals. In the general case, it is proved in 
\cite{HL26} using the localization in 
\cite[section 7]{HL26}.

\begin{theorem}\label{t5.12} 
\cite[Ch. I \S 12]{Ne99}\cite[8.2]{HL26}
Let $\Lambda_1$ and $\Lambda_2$ be two orders in $A$ with
$\Lambda_2\subsetneqq\Lambda_1$. The full lattice 
$C:=\Lambda_2:\Lambda_1$ is called {\sf conductor} of the
pair $(\Lambda_1,\Lambda_2)$.

\begin{enumerate}
\item[(a)] 
$\OO(C)\supset\Lambda_1$. The conductor $C$ is the biggest
$\Lambda_1$-ideal in $\Lambda_2$.
\item[(b)] 
If $L\in G(\Lambda_2)$ then $\Lambda_1L\in G(\Lambda_1)$.
This leads to group homomorphisms
\begin{eqnarray}
G(\Lambda_2)\to G(\Lambda_1), &&L\mapsto \Lambda_1 L,
\label{5.19}\\
G([\Lambda_2]_\varepsilon)\to G([\Lambda_1]_\varepsilon),&& 
[L]_\varepsilon \mapsto [\Lambda_1 L]_\varepsilon.\label{5.20}
\end{eqnarray}
\item[(c)] 
The following sequence is exact,
\begin{eqnarray}\label{5.21}
1\to (\Lambda_2/C)^{*} \to (\Lambda_1/C)^{*}
\to G(\Lambda_2) \to G(\Lambda_1)\to 1.
\end{eqnarray}
Especially, the group homomorphism in \eqref{5.19} is surjective.
Here the image in $\ker(G(\Lambda_2)\to G(\Lambda_1))$ 
of an element $a+C\in (\Lambda_1/C)^{*}$ with 
$a\in \Lambda_1$ is $C+a\Lambda_2$.
\item[(d)] 
The maps in the following sequence are the natural ones,
the sequence is exact,
\begin{eqnarray}\label{5.22}
1\to \Lambda_2^{*}\to \Lambda_1^{*}\to 
\frac{(\Lambda_1/C)^{*}}{(\Lambda_2/C)^{*}}
\to G([\Lambda_2]_\varepsilon)\to G([\Lambda_1]_\varepsilon)
\to 1.
\end{eqnarray}
Especially, the group homomorphism in \eqref{5.20} is surjective.
\item[(e)] 
The group $\Lambda_2^{*}$ has finite index in the group
$\Lambda_1^{*}$. The groups $(\Lambda_1/C)^{*}$
and $(\Lambda_2/C)^{*}$ are finite. 
By Theorem \ref{t5.3} the groups $G([\Lambda_2]_\varepsilon)$
and $G[\Lambda_1]_\varepsilon)$ are finite. 
The size of one of them can be calculated by the size 
of the other one with the following formula, 
\begin{eqnarray}\label{5.23}
\frac{|G([\Lambda_2]_\varepsilon)|}{|G([\Lambda_1]_\varepsilon)|}
&=& \frac{|(\Lambda_1/C)^{*}|}{|(\Lambda_2/C)^{*}|}
\cdot \frac{1}{[\Lambda_1^{*}:\Lambda_2^{*}]}.
\end{eqnarray}
\end{enumerate}
\end{theorem}

\begin{remarks}\label{t5.13}
\begin{enumerate}
\item[(i)] 
If $A$ is separable, the maximal order $\Lambda_{max}$ 
exists. The group $G([\Lambda]_\varepsilon)$ for any order
$\Lambda$ maps surjectively to $G([\Lambda_{max}]_\varepsilon)$.
Formula \eqref{5.23} allows to calculate fairly easily
$|G([\Lambda]_\varepsilon)|$ if one knows
$|G([\Lambda_{max}]_\varepsilon)|$.
\item[(ii)] 
If $A$ is inseparable, nevertheless its separable part
$F$ has the maximal order $\Lambda_{max}(F)$.
The group $G([\Lambda]_\varepsilon)$ for any order $\Lambda$
is isomorphic to the group $G([\pr_F\Lambda]_\varepsilon)$
by Theorem \ref{t5.10}, and this group maps surjectively 
to the group $G([\Lambda_{max}(F)]_\varepsilon)$ by Theorem
\ref{t5.12}.
\end{enumerate}
\end{remarks}

\section{Conjugacy classes of matrices and
$\varepsilon$-classes of full lattices}\label{s6}
\setcounter{equation}{0}
\setcounter{table}{0}
\setcounter{figure}{0}

The main body of this paper, the sections \ref{s7}--\ref{s11},
is about the set $S_{1,f}$ of $\GL_n(\Z)$-conjugacy 
classes of regular integer matrices in $M_{n\times n}(\Z)$ 
with a fixed characteristic polynomial $f\in \Z[t]$. 
Theorem \ref{t6.3} explains why the sections \ref{s2}--\ref{s5}
are relevant for this. 
It gives a 1:1 correspondence between the set 
$S_{1,f}$ and the set $S_{2,f}:=\{[L]_\varepsilon\,|\, L\in 
\LL(A_f),\ \OO(L)\supset \Lambda_f\}\subset\EE(A_f)$ where
\begin{eqnarray}\label{6.1}
A_f:=\frac{\Q[t]}{(f(t))_{\Q[t]}}&\supset&
\Lambda_f:=\frac{\Z[t]}{(f(t))_{\Z[t]}}.
\end{eqnarray}
$A_f$ is an $n$-dimensional commutative $\Q$-algebra with
unit element $1_{A_f}=[1]$, and $\Lambda_f$ is a cyclic order in $A_f$. Before giving Theorem \ref{t6.3}, we fix some notations 
and recall some basic facts from linear algebra.

\begin{definition/lemma}\label{t6.1}
Let $k$ be a field and $\oooo{k}$ its algebraic closure.

\begin{enumerate}
\item[(a)] 
(Definition) The characteristic polynomial of a matrix
$B\in M_{n\times n}(k)$ is denoted by 
$p_B\in k[t]$. (Lemma) It is unitary of degree $n$.
\item[(b)] 
(Definition) Let $f(t)=t^n+f_{n-1}t^{n-1}+...+f_1t+f_0
\in k[t]$ be unitary of degree $n$. The
{\sf companion matrix} $M_f\in M_{n\times n}(k)$ is
\begin{eqnarray}\label{6.2}
M_f:=\begin{pmatrix}0& & & -f_0\\1 & \ddots & & -f_1\\
 & \ddots & 0 & \vdots \\ & & 1 & -f_{n-1}\end{pmatrix}.
\end{eqnarray}
(Lemma) It satisfies $p_{M_f}(t)=f(t)$. 
\item[(c)] 
(Lemma) Let $B\in M_{n\times n}(k)$. The following conditions
are equivalent.
\begin{enumerate}
\item[(i)]
There is a vector $v\in M_{n\times 1}(k)$ (called 
{\sf cyclic generator}) with
\begin{eqnarray*}
M_{n\times 1}(k) = \bigoplus_{l=0}^{n-1}k\cdot B^lv.
\end{eqnarray*}
\item[(ii)]
$B$ is $\GL_n(k)$-conjugate to $M_{p_B}$, i.e. a matrix
$C\in \GL_n(k)$ with $C^{-1}BC=M_{p_B}$ exists.
\item[(iii)]
There is a matrix $C\in \GL_n(\oooo{k})$ such that the matrix
$C^{-1}BC\in M_{n\times n}(\oooo{k})$ is in Jordan normal form 
and has for each eigenvalue only one Jordan block.
\item[(iv)]
$\{C\in M_{n\times n}(k)\,|\, BC=CB\}=k[B].$
\end{enumerate}
\item[(d)] 
(Definition) A matrix $B\in M_{n\times n}(k)$ which 
satisfies the equivalent conditions in part (c) is called
{\sf regular}.
\item[(e)] 
(Definition) For any matrix $B\in M_{n\times n}(\Z)$ denote
by 
\begin{eqnarray*}
[B]_\Z:= \{C^{-1}BC\,|\, C\in \GL_n(\Z)\}
\end{eqnarray*}
the $\GL_n(\Z)$-conjugacy class of $B$.
\end{enumerate}
\end{definition/lemma}

\begin{remarks}\label{t6.2}
\begin{enumerate}
\item[(i)] By Lemma \ref{t6.1} (c) (ii), the regular matrices in
$M_{n\times n}(k)$ with fixed characteristic polynomial form
a single $\GL_n(k)$-conjugacy class.
\item[(ii)] 
Especially, a regular integer $n\times n$-matrix $B$
is $\GL_n(\Q)$-conjugate to the companion matrix
$M_{p_B}\in M_{n\times n}(\Z)$. 
\item[(iii)] 
But the set of regular integer $n\times n$ matrices with
fixed characteristic polynomial $f\in \Z[t]$ may split
into many $\GL_n(\Z)$-conjugacy classes.
The following theorem is a simple and basic, but very useful
observation on this situation. In the regular and 
semisimple case it was already made in \cite{LMD33}. 
In \cite[Theorem 1.4]{Hu17} and \cite[Theorem 8.1]{Ma20}
the result in \cite{LMD33} is generalized to a 1:1 
correspondence for semisimple matrices which are not 
necessarily regular. Theorem \ref{t6.3} concerns the case
of regular, but not necessarily semisimple matrices.
\end{enumerate}
\end{remarks}

\begin{theorem}\label{t6.3} (\cite{LMD33} for the semisimple
case)
Let $n\in\N$ and let $f\in\Z[t]$ be unitary of degree $n$.
There is a natural 1:1 correspondence between the sets
$S_{1,f}$ and $S_{2,f}$ where $S_{1,f}$ and $S_{2,f}$ 
are as follows,
\begin{eqnarray*}
S_{1,f}&:=& \{[B]_\Z\,|\, B\in M_{n\times n}(\Z)
\textup{ is regular with }p_B=f\},\\
S_{2,f}&:=& \{[L]_\varepsilon\,|\, L\in \LL(A_f)\textup{ with }
\OO(L)\supset \Lambda_f\}.
\end{eqnarray*}
$S_{1,f}$ is the set of $\GL_n(\Z)$-conjugacy classes of
regular integer matrices with characteristic polynomial $f$,
and $S_{2,f}$ is the subsemigroup of $\EE(A_f)$ which consists
of the $\varepsilon$-classes of full lattices in $A$ which are 
$\Lambda_f$-ideals.

\begin{enumerate}
\item[(i)] 
From $B$ with $[B]_\Z\in S_{1,f}$ to $L$ with 
$[L]_\varepsilon\in S_{2,f}$: 
Choose $C\in \GL_n(\Q)$ with $CBC^{-1}=M_f$. 
Denote $\oooo{t}:=[t]\in \Lambda_f\subset A_f$ and define
\begin{eqnarray}\label{6.3}
(b_1,...,b_n):=\BB&:=&(1,\oooo{t},\oooo{t}^2,...,
\oooo{t}^{n-1})\cdot C,\\
L&:=& \bigoplus_{i=1}^n\Z\cdot b_i.\label{6.4}
\end{eqnarray}
\item[(ii)] 
From $L$ with $[L]_\varepsilon\in S_{2,f}$ to
$B$ with $[B]_\Z\in S_{1,f}$. 
Choose a $\Z$-basis $(b_1,...,b_n)$ of $L$. Then
\begin{eqnarray}\label{6.5}
\oooo{t}(b_1,...,b_n):=(\oooo{t}b_1,...,\oooo{t}b_n)
=(b_1,...,b_n)\cdot B
\end{eqnarray}
for a unique matrix $B\in M_{n\times n}(\Z)$.
\end{enumerate}
\end{theorem}

{\bf Proof:} The endomorphism
\begin{eqnarray*}
\mu_{\oooo{t}}:=(\textup{multiplication with }\oooo{t}):
A_f\to A_f,\ a\mapsto \oooo{t}a,
\end{eqnarray*}
is regular with characteristic polynomial $f$.
Any full lattice $L\in \LL(A_f)$ induces by \eqref{6.5}
a $\GL_n(\Z)$-conjugacy class of regular matrices in
$M_{n\times n}(\Q)$ with characteristic polynomial $f$.
It consists of matrices in $M_{n\times n}(\Z)$ if and only if
$\oooo{t}L\subset L$, so if and only if $\OO(L)\supset\Lambda_f$.

Any class $[B]_\Z\in S_{1,f}$ is obtained by this construction
because in part (i) 
\begin{eqnarray*}
\mu_{\oooo{t}}(1,\oooo{t},...,\oooo{t}^{n-1})&=&
(\oooo{t},\oooo{t}^2,...,\oooo{t}^n)
=(1,\oooo{t},...,\oooo{t}^{n-1})\cdot M_f,\\
\mu_{\oooo{t}}\BB&=&\BB\cdot B.
\end{eqnarray*}

Suppose that two full lattices $L_1$ and $L_2$ in $A_f$
with $\OO(L_1)\supset \Lambda_f$ and $\OO(L_2)\supset \Lambda_f$
give the same conjugacy class $[B]_\Z\in S_{1,f}$.
Then there are $\Z$-bases $\BB_1$ of $L_1$ and
$\BB_2$ of $L_2$ with
\begin{eqnarray*}
\mu_{\oooo{t}}\BB_1=\BB_1\cdot B\quad\textup{and}\quad 
\mu_{\oooo{t}}\BB_2=\BB_2\cdot B.
\end{eqnarray*}
Write $\BB_2=\BB_1\cdot C$ for some matrix $C\in \GL_n(\Q)$.
Then $BC=CB$, so by Lemma \ref{t6.1} (c) (iv)
$C=\sum_{l=0}^{n-1}c_lB^l\in\Q[B]$. The element
$$c:=\sum_{l=0}^{n-1}c_l\oooo{t}^l$$
satisfies
\begin{eqnarray*}
c\in A_f^{*},\quad c\BB_1=\BB_2,\quad cL_1=L_2,
\end{eqnarray*}
so $[L_1]_\varepsilon =[L_2]_\varepsilon$.
All statements in Theorem \ref{t6.3} are proved.
\hfill$\Box$

\begin{remarks}\label{t6.4} 
Let $n\in\N$ and let $f\in \Z[t]$ be unitary of degree $n$.

\begin{enumerate}
\item[(i)] 
$S_{2,f}\subset \EE(A_f)$ is a subsemigroup of the
semigroup $\EE(A_f)$. It has the unit element
$[\Lambda_f]_\varepsilon$. So it is a monoid.
Also the division map in $\EE(A_f)$ restricts to $S_{2,f}$ 
because of \eqref{4.2}. 
Therefore also $S_{1,f}$ is a monoid with division map.
\item[(ii)] 
Though, given matrices $B_1$ and $B_2$ with 
$[B_1]_\Z$ and $[B_2]_\Z\in S_{1,f}$, it is not easy to find
matrices $B_3$ and $B_4$ with
\begin{eqnarray}\label{6.6}
[B_3]_\Z = [B_1]_\Z\cdot [B_2]_\Z\quad\textup{and}\quad
[B_4]_\Z = [B_1]_\Z:[B_2]_\Z.
\end{eqnarray}
Taussky gives in \cite[5.]{Ta60} and in \cite[3.]{Ta74}
two ways to obtain $B_3$ from $B_1$ and $B_2$. But they just 
encode two obvious ways how to go from $B_1$ and $B_2$ 
via $L_i\in\LL(A_f)$ with $[L_i]_\varepsilon\sim [B_i]_\Z$
and $L_3=L_1L_2$ to $B_3$. 
\item[(iii)] 
The unit element $[\Lambda_f]$ in $S_{2,f}$ corresponds
to the $\GL_n(\Z)$-conjugacy class $[M_f]_\Z$ in $S_{1,f}$.
\item[(iv)] 
The following conditions are equivalent:
\begin{enumerate}
\item[($\alpha$)]
$A_f$ is separable.
\item[($\beta$)]
$f$ does not have multiple roots in $\C$.
\item[($\gamma$)]
$M_f$ is semisimple.
\item[($\delta$)]
Any regular matrix in $M_{n\times n}(\Z)$ with characteristic
polynomial $f$ is semisimple.
\item[($\epsilon$)]
$S_{1,f}$ is finite.
\end{enumerate}
Here $(\alpha)\Rightarrow(\epsilon)$ follows from 
Corollary \ref{t5.4}. For the inverse implication
$(\epsilon)\Rightarrow(\alpha)$, Remark \ref{t5.11} (i) is 
needed.
\item[(v)] 
In any case, the set $S_{1,f}$
splits into the disjoint union
\begin{eqnarray}\label{6.7}
\dot\bigcup_{\Lambda\supset\Lambda_f\textup{ order}}
\{[L]_\varepsilon\,|\, L\in\LL(A_f),\OO(L)=\Lambda\}
\end{eqnarray}
of sets which are finite by Theorem \ref{t5.3}. 
If the conditions in (iv) do not hold, it is
a union of infinitely many finite sets.
In any case, for $B$ with $[B]_\Z\in S_{1,f}$ the order 
$\Lambda\supset\Lambda_f$ with $\Lambda=\OO(L)$ and
$[B]_\Z\sim [L]_\varepsilon$ is an interesting
invariant of the $\GL_n(\Z)$-conjugacy class $[B]_\Z$ of $B$.
\end{enumerate}
\end{remarks}

Finally we make a few comments on integer $n\times n$ matrices
which are not necessarily regular. Part (a) of Theorem 
\ref{t6.5} is linear algebra. Part (b) follows from the
Jordan-Zassenhaus theorem \cite{Za38} 
(see also \cite[(79.1)]{CR62} or \cite[(26.4)]{Re03}). 
Part (c) can be proved together with the derivation of 
part (b) from the Jordan-Zassenhaus theorem.

\begin{theorem}\label{t6.5}
Fix a matrix $C\in M_{n\times n}(\C)$ in Jordan normal form
whose characteristic polynomial $f$ is in $\Z[t]$. Denote by
$[C]_\C$ its $\GL_n(\C)$-conjugacy class.

\begin{enumerate}
\item[(a)] 
$M_{n\times n}(\Q)\cap [C]_\C$ is a single 
$\GL_n(\Q)$-conjugacy class. It contains block diagonal matrices
in $M_{n\times n}(\Z)$ whose diagonal blocks are companion 
matrices for certain divisors in $\Z[t]$ of $f$.
\item[(b)] 
If $C$ is semisimple then $M_{n\times n}(\Z)\cap[C]_\C$
splits into finitely many $\GL_n(\Z)$-conjugacy classes.
\item[(c)] 
If $C$ is not semisimple then $M_{n\times n}(\Z)\cap [C]_\C$
splits into infinitely many $\GL_n(\Z)$-conjugacy classes.
\end{enumerate}
\end{theorem}

More detailed studies of $\GL_n(\Z)$-conjugacy classes of 
integer matrices which are not necessarily regular 
are made in \cite{Gr80}, \cite{Hu17}, \cite{EHO19} 
and \cite{BHJ22}, see the Remarks \ref{t1.1} (v).

\section{The irreducible rank 2 cases}\label{s7}
\setcounter{equation}{0}
\setcounter{table}{0}
\setcounter{figure}{0}

Lemma \ref{t7.2} (b) will give an elementary 1:1 correspondence
between $\SL_2(\Z)$-conjucagy classes of matrices in
$M_{2\times 2}(\Z)$ and proper equivalence classes of 
binary quadratic forms.
The theory of binary quadratic forms is old, going back to
Lagrange, Legendre and especially Gau{\ss} who gave the first
systematic treatment \cite{Ga01}.
Nowadays, it is covered in many books, e.g.
\cite{Zag81} \cite{Tr13} \cite{We17} \cite{Ha21}.
Also the thesis \cite{Ku12} says quite a lot on the
rank 2 cases. 
Therefore we could stop after Lemma \ref{t7.2}.
But we go on as we consider it useful to see how the objects 
in the sections \ref{s2}--\ref{s6} look in the rank 2 cases.

Subsection \ref{s7.1} discusses the different types of
integer $2\times 2$ matrices and provides Lemma \ref{t7.2}.
The subsections \ref{s7.2}--\ref{s7.4} treat the cases
where $A$ is a quadratic number field.

\subsection{Preliminaries, a correspondence with binary
quadratic forms}\label{s7.1}

A matrix $B\in M_{2\times 2}(\Z)$ with eigenvalues 
$\lambda_1,\lambda_2\in\C$ is of one of the following five types.

\begin{enumerate}
\item[Type I:] 
$B$ is semisimple with $\lambda_1=\lambda_2\in\Z$.
Then $B=\lambda_1\cdot E_2$. 
\item[Type II:]
$B$ has a $2\times 2$ Jordan block with eigenvalue 
$\lambda_1=\lambda_2\in\Z$.
\item[Type III:]
$B$ is semisimple with $\lambda_1,\lambda_2\in\Z$ and
$\lambda_1\neq\lambda_2$.
\item[Type IV:]
$B$ is semisimple with $\lambda_1,\lambda_2\in\R-\Z$ and
$\lambda_1\neq\lambda_2$.
\item[Type V:]
$B$ is semisimple with $\lambda_1,\lambda_2\in\C-\R$ and
$\lambda_2=\oooo{\lambda_1}$.
\end{enumerate}

Only $B$ of type I is not regular. There is nothing more to
say on type I. 
The types II and III are treated in the subsection \ref{s10.2}
and \ref{s9.2}.

Below we treat the types IV and V. For them the algebra $A$,
the orders and the semigroups $\EE(A)$
have much in common and are treated together.
Because of the relationship to binary quadratic forms,
we consider also $\SL_2(\Z)$-conjugacy classes, not only
$\GL_2(\Z)$-conjugacy classes.

An $\SL_2(\Z)$-conjugacy class of type V has a canonical
representative. This is related to the action of $\PSL_2(\Z)$
on the upper half plane $\H$. The question on (semi-)normal forms
for an $\SL_2(\Z)$-conjugacy class of type IV is more 
difficult. First we study it together with type V.
But then we turn to a point of view which we find especially
attractive, Conway's topograph. It allows a clear
overview on different semi-normal forms for type IV
in the literature.

\begin{definition}\label{t7.1}
\begin{enumerate}
\item[(a)] 
Define $\Z^2:=M_{2\times 1}(\Z)$.
We write $\uuuu{x}=\begin{pmatrix}x_1\\x_2\end{pmatrix}
\in \Z^2$.
\item[(b)] 
A {\it binary quadratic form} $q$ is a map
\begin{eqnarray*}
q=[a,h,b]_{quad}:\Z^2\to\Z\ \textup{ with }\ 
q(\uuuu{x})=q(\begin{pmatrix}x_1\\x_2\end{pmatrix})
=ax_1^2+hx_1x_2+bx_2^2
\end{eqnarray*}
for some $a,h,b\in\Z$.
\item[(c)] 
Two binary quadratic forms $q_1$ and $q_2$ are 
{\it properly equivalent} if a matrix $A\in \SL_2(\Z)$ with
$q_2(\uuuu{x})
=q_1(A\uuuu{x})$ exists.
They are $\GL_2(\Z)$-equivalent if a matrix $A\in \GL_2(\Z)$ with
$q_2(\uuuu{x})
=\det A\cdot q_1(A\uuuu{x})$ exists.
\end{enumerate}
\end{definition}

\begin{lemma}\label{t7.2}
\begin{enumerate}
\item[(a)] 
There is a 1:1 correspondence between binary quadratic forms
$[a,h,b]_{quad}$ and symmetric $\Z$-bilinear forms
$[a,h,b]_{sym}:\Z^2\times \Z^2\to\frac{1}{2}\Z$ of the shape
\begin{eqnarray*}
[a,h,b]_{sym}(\begin{pmatrix}x_1\\x_2\end{pmatrix},
\begin{pmatrix}y_1\\y_2\end{pmatrix})
=\begin{pmatrix}x_1&x_2\end{pmatrix}
\begin{pmatrix}a&h/2\\h/2&b\end{pmatrix}
\begin{pmatrix}y_1\\y_2\end{pmatrix}.
\end{eqnarray*}
$q=[a,h,b]_{quad}$ is obtained from $[a,h,b]_{sym}$ by 
$q(\uuuu{x})=[a,h,b]_{sym}(\uuuu{x},\uuuu{x})$.
$[a,h,b]_{sym}$ is obtained from $q=[a,h,b]_{quad}$ by
\begin{eqnarray*}
[a,h,b]_{sym}(\uuuu{x},\uuuu{y})= 
\frac{1}{2}(q(\uuuu{x}+\uuuu{y})-q(\uuuu{x})-q(\uuuu{y})).
\end{eqnarray*}
\item[(b)] 
Fix $r\in\Z$. There is a 1:1 correspondence between matrices
$A\in M_{2\times 2}(\Z)$ with trace $\tr A=r$ and binary
quadratic forms. It maps 
$B=\begin{pmatrix}a&b\\c&d\end{pmatrix}$ to $[c,d-a,-b]_{quad}$.
One $\SL_2(\Z)$-conjugacy class is mapped to one proper 
equivalence class. One $\GL_2(\Z)$-conjugacy class is mapped
to one $\GL_2(\Z)$-equivalence class.
\end{enumerate}
\end{lemma}

{\bf Proof:} 
(a) Trivial.

(b) Consider the injective map
\begin{eqnarray*}
\Psi:\{B\in M_{2\times 2}(\Z)\,|\, \tr B=r\}\to
\{C\in M_{2\times 2}(\frac{1}{2}\Z)\,|\, C^t=C\},\\
B=\begin{pmatrix}a&b\\c&d\end{pmatrix}\mapsto \Psi(B):=
(B-\frac{r}{2}E_2)^t\begin{pmatrix}0&-1\\1&0\end{pmatrix}
=\begin{pmatrix}c&\frac{d-a}{2}\\ \frac{d-a}{2}&-b\end{pmatrix}.
\end{eqnarray*}
$\Psi(B)$ is equivalent to $[c,d-a,-b]_{sym}$ and to 
$[c,d-a,-b]_{quad}$. One recovers $B$ from $\Psi(B)$
and $\tr B=r$. 

Recall $\SL_2(\Z)=Sp_2(\Z)$ and the generalization to 
$\GL_2(\Z)$,
\begin{eqnarray*}
A^t\begin{pmatrix}0&-1\\1&0\end{pmatrix}A
=\det A\cdot \begin{pmatrix}0&-1\\1&0\end{pmatrix}
\quad\textup{for }A\in \GL_2(\Z).
\end{eqnarray*}
It implies for $A\in \GL_2(\Z)$
\begin{eqnarray*}
\Psi(A^{-1}BA)&=& A^t(B-\frac{r}{2}E_2)^tA^{-t}
\begin{pmatrix}0&-1\\1&0\end{pmatrix}\\
&=& \det A\cdot A^t\Psi(B)A.\hspace*{3cm}\Box
\end{eqnarray*}

\begin{remarks}\label{t7.3}
Consider a unitary polynomial $f=f(t)=t^2-rt+s\in\Z[t]$
of degree 2. 

\begin{enumerate}
\item[(i)] 
By Theorem \ref{t6.3} 
the problem to find all $\GL_2(\Z)$-conjugacy classes of matrices
in $M_{2\times 2}(\Z)$ with characteristic polynomial $f$
leads to the following objects.
\begin{enumerate}
\item[($\alpha$)]
The two-dimensional $\Q$-algebra $A=A_f=\Q[t]/(f)$ with
generator $\oooo{t}=(\textup{class of }t\textup{ in }A)$,
so $A=\Q\cdot 1+\Q\cdot \oooo{t}=\Q[\oooo{t}]$. 
\item[($\beta$)]
The cyclic order $\Lambda_f= \Z[\oooo{t}]=\Z\cdot 1+
\Z\cdot\oooo{t} \subset A$.
\item[($\gamma$)]
The full lattices $L\in\LL(A)$ with $\OO(L)\supset\Z[\oooo{t}]$
and the $\varepsilon$-classes of them. 
\end{enumerate}
Then one $\GL_2(\Z)$-conjugacy class of matrices in 
$M_{2\times 2}(\Z)$ with characteristic polynomial $f$
corresponds to one $\varepsilon$-class of such full lattices
$L$ with $\OO(L)\supset\Z[\oooo{t}]$. Given such a full lattice
$L$, the choice of a $\Z$-basis $(b_1,b_2)$ provides a matrix
$M\in M_{2\times 2}(\Z)$ with characteristic polynomial $f$ 
by $(\oooo{t}b_1,\oooo{t}b_2)=(b_1,b_2)\cdot M$.
\item[(ii)] 
In order to catch the $\SL_2(\Z)$-conjugacy classes of 
matrices, one has to take into account orientations of bases.
The set of $\Q$-bases of $A$ splits into two disjoint subsets
$o_1$ and $o_2$. A $\Q$-basis of $A$ is in $o_1$ 
[respectively $o_2$] if the base change matrix to the
$\Q$-basis $(1,\oooo{t})$ has positive [respectively negative]
determinant. 

The norm $N(a)\in\Q$ of an element $a\in A$ is
\begin{eqnarray*}
N(a):=\det\bigl((\textup{multiplication with }a):A\to A\bigr).
\end{eqnarray*}
This fits to the classical definition of the norm in a 
quadratic number field. 

One considers pairs $(L,o_i)\in\LL(A)\times\{o_1,o_2\}$
and an $(\varepsilon+)$-equivalence on them which is defined by
\begin{eqnarray*}
&&(L,o_i)\sim_{\varepsilon+}(\www{L},o_j)\\
&\iff& \textup{an element }a\in A^{*}\textup{ exists with }
\www{L}=aL\textup{ and}\\
&& N(a)>0\textup{ if }o_i=o_j\textup{ respectively }
N(a)<0\textup{ if }o_i\neq o_j.
\end{eqnarray*}
Then one $\SL_2(\Z)$-conjugacy class of matrices in 
$M_{2\times 2}(\Z)$ with characteristic polynomial $f$
corresponds to an $(\varepsilon+)$-class of pairs
$(L,o_i)$ with $\OO(L)\supset\Z[\oooo{t}]$. 
Given a pair $(L,o_i)$, the choice of a $\Z$-basis
$(b_1,b_2)\in o_i$ of $L$ provides a matrix 
$M\in M_{2\times 2}(\Z)$ with characteristic polynomial $f$ by
$(\oooo{t}b_1,\oooo{t}b_2)=(b_1,b_2)\cdot M$.
\end{enumerate}
\end{remarks}

It will be useful to consider binary quadratic forms not
only on $\Z^2$ (as in Definition \ref{t7.1}), but also on
abstract rank 2 $\Z$-lattices.

\begin{definition}\label{t7.4}
An {\it abstract binary quadratic form} is a triple
$(H_\Z,o_1,q)$ where $H_\Z$ is a $\Z$-lattice
of rank 2, $o_1$ is one of the
two orientation classes of $\Z$-bases of $H_\Z$ and
$q:H_\Z\to\Z$ is a map which takes with respect to some
$\Z$-basis $(v_1,v_2)\in o_1$ the shape
\begin{eqnarray*}
q(x_1v_1+x_2v_2)=ax_1^2+hx_1x_2+bx_2^2\quad\textup{for }
x_1,x_2\in\Z
\end{eqnarray*}
for some $a,h,b\in\Z$.
\end{definition}

\begin{remarks}\label{t7.5}
\begin{enumerate}
\item[(i)]
Let an abstract binary quadratic form $(H_\Z,o_1,q)$ be given.
Each $\Z$-basis $\uuuu{v}=(v_1,v_2)$ in $o_1$ induces an 
isomorphism $\beta_{\uuuu{v}}:H_\Z\to\Z^2$ which maps $\uuuu{v}$
to the standard $\Z$-basis of $\Z^2$. 
It gives a binary quadratic form 
$q_{\uuuu{v}}:=q\circ\beta_{\uuuu{v}}^{-1}:\Z^2\to\Z$. Running through all
$\Z$-bases in $o_1$ one obtains all binary quadratic forms
in one proper equivalence class.
\item[(ii)] 
We call a binary quadratic form 
{\it of type I, II, III, IV or V} 
if some matrix which corresponds to it by Lemma \ref{t7.2} (b)
is of the same type. This is possible as the type of a matrix
$B\in M_{2\times 2}(\Z)$ does not change if one adds a multiple
of $E_2$. We call an abstract binary quadratic form {\it of type
I, II, III, IV or V} if some (equivalent: any) binary form in the
corresponding proper equivalence class is of the same type. 
Lemma \ref{t7.6} says what the types mean for the
binary quadratic forms.
\end{enumerate}
\end{remarks}

\subsection{Algebras, orders, and $\SL_2(\Z)$-conjugacy classes}
\label{s7.2}

Lemma \ref{t7.6} collects basic facts on matrices in
$M_{2\times 2}(\Z)$ and the associated binary quadratic forms.

\begin{lemma}\label{t7.6}
Let $B=\begin{pmatrix}a&b\\c&d\end{pmatrix}\in 
M_{2\times 2}(\Z)$, and let $q=[c,d-a,-b]_{quad}$ be the
associated binary quadratic form.

\begin{enumerate}
\item[(a)] 
The characteristic polynomial $p_B(t)=t^2-rt+s\in\Z[t]$
satisfies
$$r=\tr B=a+d,\quad s=\det B=ad-bc.$$
The eigenvalues $\lambda_1$ and $\lambda_2$ of $B$ are 
\begin{eqnarray*}
\lambda_{1/2}=\frac{r}{2}\pm\sqrt{D}\quad\textup{with}\quad
D=\frac{r^2}{4}-s=\bigl(\frac{a-d}{2}\bigr)^2+bc.
\end{eqnarray*}
Here $\sqrt{D}$ is the square root with
$\sqrt{D}\in\R_{\geq 0}\,\cup\, i\R_{>0}$. 
One row eigenvector $v_1$ of $B$ with eigenvalue $\lambda_1$,
so with $v_1B=\lambda_1 v_1$, is
\begin{eqnarray}\label{7.1}
v_1=(c,-a+\lambda_1)=(c,\frac{d-a}{2}+\sqrt{D})
\end{eqnarray}
if $v_1\neq 0$.
\item[(b)] 
$B$ is of type I $\iff$ $q=0$.
\item[(c)] 
$B$ is of type II $\iff$ $D=0$ and $q\neq 0$. 
Then $q$ is semidefinite with $0\in q(\Z^2-\{0\})$.
\item[(d)] 
$B$ is of type III $\iff$ $D>0$ and $4D\in\N$ is a square.
Then $q$ is indefinite with $0\in q(\Z^2-\{0\})$. 
\item[(e)] 
$B$ is of type IV $\iff$ $D>0$ and $4D\in\N$ is not a square.
Then $A=\Q[\sqrt{\delta}]$ for the square free $\delta\in\N$
with $4D=\delta\cdot(\textup{a square})$. 
We choose $\oooo{t}=\lambda_1$. Then $q$ is indefinite with
$bc\neq 0$ and $0\notin q(\Z^2-\{0\})$. 
\item[(f)] 
$B$ is of type V $\iff$ $D<0$. Then $A=\Q[\sqrt{\delta}]$
for the $\delta\in \Z_{<0}$ with $|\delta|$ square free
and $4D=\delta\cdot(\textup{a square})$. 
We choose $\oooo{t}=\lambda_1$. Then $bc<0$,
$q$ is positive definite if $c>0$, and $q$ is negative definite
if $c<0$. 
\end{enumerate}
\end{lemma}

{\bf Proof:}
Part (a) and most of the parts (b) to (f) are obvious. 
Observe
$$\det\begin{pmatrix}c&\frac{d-a}{2}\\ \frac{d-a}{2}&-b
\end{pmatrix}=-D.$$
Therefore $q=[c,d-a,-b]_{quad}$ is indefinite of $D>0$ and
positive or negative definite of $D<0$.\hfill$\Box$

\bigskip
Theorem \ref{t7.7} collects well known facts on quadratic 
number fields and their orders 
(see e.g. \cite[Ch. 2 7.2 and 7.3]{BSh73}\cite[4.2, 6.7]{Tr13}).

\begin{theorem}\label{t7.7}
Let $\delta\in\Z-\{0;1\}$ be square free. Then 
$\delta\equiv 1\textup{ or }2\textup{ or }3(4)$.
Consider the quadratic number field $A=\Q[\sqrt{\delta}]$.

\begin{enumerate}
\item[(a)] 
The maximal order (= the ring of algebraic
integers) in $A$ is $\Lambda_1=\langle 1,\omega\rangle_\Z$ with
\begin{eqnarray*}
\omega =\left\{\begin{array}{ll}
\sqrt{\delta}&\textup{ if }\delta\equiv 2\textup{ or }3(4),\\
\frac{1+\sqrt{\delta}}{2}& \textup{ if }\delta\equiv 1(4).
\end{array}\right.
\end{eqnarray*}
\item[(b)] 
The orders in $A$ are the full lattices 
$$\Lambda_n:=\langle 1,n\omega\rangle_\Z\quad \textup{ for }
n\in\N.$$
\item[(c)] 
Consider $\delta<0$. Then the unit group of $\Lambda_n$ is
\begin{eqnarray*}
\Lambda_n^{*}=\left\{\begin{array}{ll}
\{\pm 1\}&\textup{ if }n\geq 2\textup{ or if }n=1\textup{ and }
\delta\notin\{-1;-3\},\\
\{\pm 1,\pm i\}&\textup{ if }n=1\textup{ and }\delta=-1,\\
\{\pm 1,e^{\pm 2\pi i/3},e^{\pm 2\pi i/6}\}&\textup{ if }
n=1\textup{ and }\delta=-3.
\end{array}\right.
\end{eqnarray*}
\item[(d)] 
Consider $\delta>0$. Then the unit groups of $\Lambda_1$ and
$\Lambda_n$ for $n\geq 2$ are 
\begin{eqnarray*}
\Lambda_1^{*}&=& \{\pm 1\}\times \{\varepsilon^l\,|\, l\in\Z\}
\quad\textup{ for a unique unit }\varepsilon>1,\\
\Lambda_n^{*}&=& \{\pm 1\}\times \bigl(
\textup{a subgroup of finite index in }\{\varepsilon^l\,|\, 
l\in\Z\}\bigr).
\end{eqnarray*}
\end{enumerate}
\end{theorem}

Theorem \ref{t7.8} treats the semigroups $\EE(A)$ and
$W(\EE(A))=W(\LL(A))$.

\begin{theorem}\label{t7.8}
Let $\delta\in\Z-\{0;1\}$ be square free, and let 
$A=\Q[\sqrt{\delta]}$ be as in Theorem \ref{t7.7}.

\begin{enumerate}
\item[(a)] 
Each full lattice $L\in\LL(A)$ is invertible.
\item[(b)] 
$\EE(A)=\bigcup_{n\in\N}G([\Lambda_n]_\varepsilon)$
is a union of finite groups. 
\item[(c)] 
The semigroup of $w$-classes is
\begin{eqnarray*}
(W(\EE(A)),\cdot)\cong (\{\textup{orders in }A\},\cdot)
\cong (\N,\gcd).
\end{eqnarray*}
\item[(d)] 
\begin{enumerate}
\item[(i)] 
$\Lambda_{n_1}\supset \Lambda_{n_2}$ if and only if
$n_1$ divides $n_2$.
\item[(ii)] In that case the map
$$G([\Lambda_{n_2}]_\varepsilon)\to 
G([\Lambda_{n_1}]_\varepsilon),\quad 
[L]_\varepsilon\mapsto [\Lambda_{n_1}L]_\varepsilon,$$
is a surjective group homomorphism of finite groups, and
\begin{eqnarray}\label{7.2}
\frac{|G([\Lambda_{n_2}]_\varepsilon)|}
{|G([\Lambda_{n_1}]_\varepsilon)|}
&=& \frac{|(\Lambda_{n_1}/C)^{*}|}
{|(\Lambda_{n_2}/C)^{*})|}\cdot 
\frac{1}{|\Lambda_{n_1}^{*}:\Lambda_{n_2}^{*}]}.
\end{eqnarray}
Here $C:=\Lambda_{n_2}:\Lambda_{n_1}\in\LL(A)$ 
is their conductor.
\item[(iii)] 
The conductor $C$ satisfies
\begin{eqnarray*}
C=\langle \frac{n_2}{n_1},n_2\omega\rangle_\Z
=\frac{n_2}{n_1}\langle 1,n_1\omega\rangle_\Z
=\frac{n_2}{n_1}\Lambda_{n_1},\\
(\Lambda_{n_2}/C)^{*}\cong\Z_{n_2/n_1}^{*},\quad 
|(\Lambda_{n_2}/C)^{*}|=\varphi(n_2/n_1).
\end{eqnarray*}
The class in $\Lambda_{n_1}/C$ of an element $a\in\Lambda_{n_1}$
is in $(\Lambda_{n_1}/C)^{*}$ if and only if 
$\gcd(N(a),\frac{n_2}{n_1})=1$. This makes also the 
calculation of $|(\Lambda_{n_1}/C)^{*}|$ fairly easy.
\item[(iv)] 
The index $[\Lambda_{n_1}^{*}:\Lambda_{n_2}^{*}]$ 
can also be calculated fairly easily, see Theorem \ref{t7.7}
(c) and (d). The most difficult part in applying formula 
\eqref{7.2} is to determine the class number 
$|G([\Lambda_{n_1}]_\varepsilon)|\in\N$ of $A$.
\end{enumerate}
\end{enumerate}
\end{theorem}

{\bf Proof:} (a) See Corollary \ref{t4.15}.

(b) and (c) The product of two orders $\Lambda_{n_1}$ and
$\Lambda_{n_2}$ is
\begin{eqnarray*}
\Lambda_{n_1}\cdot \Lambda_{n_2}
=\langle 1,n_1\omega,n_2\omega,n_1n_2\omega^2\rangle_\Z
=\langle 1\gcd(n_1,n_2)\omega\rangle_\Z
=\Lambda_{\gcd(n_1,n_2)}
\end{eqnarray*}
because
\begin{eqnarray*}
\omega^2&=&\left\{\begin{array}{ll}
\delta & \textup{ if }\delta\equiv 2\textup{ or }3(4),\\
\frac{\delta-1}{4}+\omega&\textup{ if }\delta\equiv 1(4).
\end{array}\right.
\end{eqnarray*}
Everything else follows from part (a) and Theorem \ref{t5.3}.

(d) Part (i) is trivial.
Part (ii) is a special case of Theorem \ref{t5.12}.

(iii) The conductor $C=\Lambda_{n_2}:\Lambda_{n_1}$ obviously
contains $\langle \frac{n_2}{n_1},n_2\omega\rangle_\Z$. 

If $a+b\omega\in C$ for some $a,b\in\Z$, then
\begin{eqnarray*}
\Lambda_{n_2}\owns (a+b\omega)\cdot 1 =a+b\omega,
\quad\textup{so }n_2|b.
\end{eqnarray*}
Then $a\in C$, so $a\cdot n_1\omega\in \Lambda_{n_2}$, so 
$\frac{n_2}{n_1}|a$. Therefore 
$C=\langle\frac{n_2}{n_1},n_2\omega\rangle_\Z$.

The isomorphisms $\Lambda_{n_2}/C\cong \Z_{n_2/n_1}$
(as rings) and $(\Lambda_{n_2}/C)^{*}\cong \Z_{n_2/n_1}^{*}$
(as groups) are clear. 

Consider $a\in\Lambda_{n_1}$, the algebraically conjugate
element $a'$ and the norm $N(a)=aa'\in\Z$. Then
\begin{eqnarray*}
&&\textup{The class in }\Lambda_{n_1}/C\textup{ of }a
\textup{ is in }(\Lambda_{n_1}/C)^{*}\\
&\iff& \textup{the class in }\Lambda_{n_1}/C\textup{ of }a'
\textup{ is in }(\Lambda_{n_1}/C)^{*}\\
&\iff& \textup{the class in }\Lambda_{n_1}/C\textup{ of }
N(a)=aa'\textup{ is in }(\Lambda_{n_1}/C)^{*}\\
&\iff& \gcd(N(a),n_2/n_1)=1.
\end{eqnarray*}

(iv) Here nothing has to be proved. \hfill$\Box$

\bigskip
Theorem \ref{t7.9} compares $\SL_2(\Z)$-conjugacy classes and
$\GL_2(\Z)$-conjugacy classes.

\begin{theorem}\label{t7.9}
Consider an irreducible unitary polynomial
$f=f(t)=t^2-rt+s\in\Z[t]$ of degree 2. It has roots
\begin{eqnarray*}
\lambda_{1/2}=\frac{r}{2}\pm\sqrt{D}&&\textup{with }
D=\frac{r^2}{4}-s\in\frac{1}{4}\Z-\{0\}\\
&&\textup{and }\sqrt{D}\in\R_{> 0}\,\cup\, i\R_{>0}.
\end{eqnarray*}
Then $A=\Q[\lambda_1]=\Q[\sqrt{\delta}]$ where 
$\delta\in\Z-\{0;1\}$ is unique with 
$4D=\delta\cdot (\textup{a square})$ and $\delta$ square free.
$\sqrt{\delta}$ denotes the root in $\R_{>0}\,\cup\, 
i\R_{>0}$.

$o_1$ denotes the set of $\Q$-bases of $A$ with the same 
orientation as the $\Q$-bases $(1,\lambda_1)$ and 
$(1,\sqrt{\delta})$. $o_2$ denotes the set of $\Q$-bases 
of $A$ with the other orientation.

\begin{enumerate}
\item[(a)] Let $B=\begin{pmatrix}a&b\\c&d\end{pmatrix}\in 
M_{2\times 2}(\Z)$ with characteristic polynomial $f$, so
$r=\tr B=a+d$, $s=\det B=ad-bc$, $D=\bigl(\frac{a-d}{2}\bigr)^2
+bc$.

\begin{enumerate}
\item[(i)] 
The full lattice 
$$L:=\langle c,-a+\lambda_1\rangle_\Z\quad \textup{with}\quad
\Z\textup{-basis }(c,-a+\lambda_1)$$
satisfies $L\subset\Z[\lambda_1]$, $\OO(L)\supset\Z[\lambda_1]$
and
\begin{eqnarray}\label{7.3}
\lambda_1\cdot (c,-a+\lambda_1)=(c,-a+\lambda_1)\cdot B,
\end{eqnarray}
so the $(\varepsilon+)$-class $[(L,o_i)]_{\varepsilon+}$
with $(c,-a+\lambda_1)\in o_i$ corresponds to the 
$\SL_2(\Z)$-conjugacy class of $B$ and to the proper equivalence
class of the binary quadratic form $[c,d-a,-b]_{quad}$.
Here 
\begin{eqnarray}\label{7.4}
(c,-a+\lambda_1)\in o_1\textup{ if }c>0\quad\textup{and}\quad
(c,-a+\lambda_1)\in o_2\textup{ if }c<0.
\end{eqnarray}
\item[(ii)] 
One obtains $[c,d-a,-b]_{quad}$ directly from $L$ and
its $\Z$-basis $(\alpha,\beta):=(c,-a+\lambda_1)$ by the
following formula, which is probably due to Gau{\ss}
(see e.g. \cite[\S 10 (15)]{Zag81})
\begin{eqnarray}\nonumber
&&[c,d-a,-b]_{quad}(\begin{pmatrix}x_1\\x_2\end{pmatrix})\\
&=& \frac{\textup{sign}(c)}{[\Z[\lambda_1]:L]}\bigl(
\alpha\alpha'x_1^2+(\alpha\beta'+\alpha'\beta)x_1x_2+
\beta\beta'x_2^2\bigr).\label{7.5}
\end{eqnarray}
\item[(iii)] 
The condition $\gcd(c,b,a-d)=1$ (then the binary 
quadratic form $[c,d-a,-b]_{quad}$ is {\sf primitive})
is equivalent to $\OO(L)=\Z[\lambda_1]$. 
\end{enumerate}
\item[(b)] 
Consider the case $D<0$ (and $\delta<0$). 
Each $\GL_2(\Z)$-conjugacy class of matrices with 
characteristic polynomial $f$ splits into two $\SL_2(\Z)$ 
conjugacy classes. One contains matrices 
$\begin{pmatrix}a&b\\c&d\end{pmatrix}$ with $c>0$ and 
corresponds to $(\varepsilon+)$-classes of pairs
$(L,o_1)$ and to proper equivalence classes of positive definite
binary quadratic forms.
The other contains matrices 
$\begin{pmatrix}a&b\\c&d\end{pmatrix}$ with $c<0$ and
corresponds to $(\varepsilon+)$-classes of pairs $(L,o_2)$
and to proper equivalence classes of negative definite 
binary quadratic forms. 
\item[(c)] 
Consider the case $D>0$ (and $\delta>0$).
By Theorem \ref{t6.3}
one $\GL_2(\Z)$-conjugacy class of matrices with characteristic
polynomial $f$ corresponds to one $\varepsilon$-class
$[L]_\varepsilon$ of full lattices with $\OO(L)\supset
\Z[\lambda_1]$. The $\GL_2(\Z)$-conjugacy class splits into two
$\SL_2(\Z)$-conjugacy classes if and only if $\OO(L)$ does not
contain an invertible element $u$ with $N(u)=-1$.
\end{enumerate}
\end{theorem}

{\bf Proof:}
(a) (i) $L\subset\Z[\lambda_1]$ is clear. As noted in Lemma
\ref{t7.6}, $v_1=(c,-a+\lambda_1)$ is a row eigenvector of $B$,
so $\lambda_1v_1=v_1B$, which is \eqref{7.3}.
It implies $\OO(L)\supset\Z[\lambda_1]$.
Also \eqref{7.4} is clear.

(ii) For \eqref{7.5} observe 
\begin{eqnarray*}
[\Z[\lambda_1]:L]=|c|,\quad \alpha\alpha'=c^2,\\
\alpha\beta'+\alpha'\beta=c(d-a),\quad 
\beta\beta'=\bigl(\frac{a-d}{2}\bigr)^2-D=-bc.
\end{eqnarray*}

(iii) Write 
\begin{eqnarray}\label{7.6}
(a-d)^2+4bc=4D=\delta\cdot g^2 \quad\textup{for some }
g\in\N.
\end{eqnarray}

The case $\delta\equiv 2\textup{ or }3(4)$, so 
$\omega=\sqrt{\delta}$: Then $a-d\equiv 0(2)$, $g\equiv 0(2)$,
$\frac{g}{2}\sqrt{\delta}=\sqrt{D}=\lambda_1-\frac{a+d}{2}$, 
\begin{eqnarray*}
\frac{g}{2}\sqrt{\delta}(c,-a+\lambda_1)=(c,-a+\lambda_1)
\begin{pmatrix}\frac{a-d}{2}& b\\c&\frac{d-a}{2}\end{pmatrix}.
\end{eqnarray*}
Then 
$$\OO(L)=\Z[\lambda_1]=\Z[\frac{g}{2}\sqrt{\delta}]\iff
\gcd(c,b,\frac{a-d}{2})=1.$$
Suppose $\gcd(c,b,\frac{a-d}{2})=1$, but 
$\gcd(c,b,a-d)>1$. Then
$a-d\equiv 2(4)$, $c\equiv 0(2)$, $b\equiv 0(2)$,
$D\equiv 1(4)$. This contradicts
$D=\delta\cdot (\frac{g}{2})^2$ and 
$\delta\equiv 2\textup{ or }3(4)$. Therefore
$\gcd(c,b,a-d)=1\iff\gcd(c,b,\frac{a-d}{2})=1$.

The case $\delta\equiv 1(4)$, so 
$\omega=\frac{1+\sqrt{\delta}}{2}$: 
Then $\lambda_1=\frac{a+d-g}{2}+g\omega$, 
$a+d-g\equiv 0(2)$,
\begin{eqnarray*}
g\omega(c,-a+\lambda_1)=(c,-a+\lambda_1)
\begin{pmatrix}\frac{a-d+g}{2} & b \\ c & \frac{d-a+g}{2}
\end{pmatrix}.
\end{eqnarray*}
Then
$$\OO(L)=\Z[\lambda_1]=\Z[g\omega]\iff
\gcd(c,b,\frac{a-d+g}{2},g)=1.$$
Suppose $\gcd(c,b,\frac{a-d+g}{2},g)=1$, but 
$\gcd(c,b,a-d)=:g_2>1$. 
Because of \eqref{7.6} and $\delta$ square free $g_2$
divides $g$, so $\gcd(c,b,a-d,g)=g_2$. This and
$\gcd(c,b,\frac{a-d+g}{2},g)=1$ imply
$g_2=2$, $a-d+g\equiv 2(4)$, $c\equiv 0(2)$, $b\equiv 0(2)$,
$g\equiv 0(2)$, $a-d\equiv 0(2)$. Then \eqref{7.6} implies
$(a-d+g)(a-d-g)\equiv 0(16)$, which contradicts
$a-d+g\equiv 2(4)$, $g\equiv 0(2)$. 
Therefore 
$\gcd(c,b,a-d)=1\iff\gcd(c,b,\frac{a-d+g}{2},g)=1$.

(b) The binary quadratic forms in one proper equivalence class
are either all positive definite or all negative definite.
Together with Lemma \ref{t7.2} (b) this implies the statements
on $\GL_2(\Z)$-conjugacy classes and $\SL_2(\Z)$-conjugacy classes.
Because any element $a\in A-\{0\}$ satisfies $N(a)>0$,
$(L,o_i)\sim_{\varepsilon+}(\www{L},o_j)$ implies 
$o_i=o_j$.

(c) The $\GL_2(\Z)$-conjugacy class of 
$\begin{pmatrix}a&b\\c&d\end{pmatrix}$ is a single 
$\SL_2(\Z)$-conjugacy class if and only if
the full lattice $L=\langle c,-a+\lambda_1\rangle_\Z$ from
part (a) satisfies $(L,o_1)\sim_{\varepsilon+}(L,o_2)$. 
This holds if and only if there is an invertible element 
$u\in\OO(L)$ with $N(u)=-1$. \hfill$\Box$

\subsection{Normal forms and semi-normal forms}
\label{s7.3}

In this section we fix one irreducible unitary polynomial
$f=f(t)=t^2-rt+s\in\Z[t]§$ of degree 2, as in Theorem \ref{t7.9}.
Recall its roots
\begin{eqnarray*}
\lambda_{1/2}=\frac{r}{2}\pm \sqrt{D}\quad\textup{with}\quad
D=\frac{r^2}{4}-s\in\frac{1}{4}\Z-\{0\}\\
\textup{and}\quad \sqrt{D}\in\R_{>0}\,\cup\, i\R_{>0}.
\end{eqnarray*}

Integer matrices $B=\begin{pmatrix}a&b\\c&d\end{pmatrix}$
with characteristic polynomial $f$ are of type 
IV if $D>0$ and of type V if $D<0$. We are interested in 
normal forms or semi-normal forms within each $\SL_2(\Z)$
conjugacy class. First we will construct simultaneously 
semi-normal forms for both types. In the case of type V they will
turn out to be normal forms. In the case of type IV we will use
in subsection \ref{s7.4} Conway's topograph 
\cite[ch. 1]{Co97} in order to understand the semi-normal forms
and compare them with two  other semi-normal forms.

Consider the matrices 
\begin{eqnarray*}
E_2=\begin{pmatrix}1&0\\0&1\end{pmatrix},\ 
T=\begin{pmatrix}1&1\\0&1\end{pmatrix},\ 
S=\begin{pmatrix}0&-1\\1&0\end{pmatrix}\ \textup{and}\ 
D_{1,-1}=\begin{pmatrix}1&0\\0&-1\end{pmatrix}.
\end{eqnarray*}
It is well known that $\SL_2(\Z)$ is generated by $S$, $T$ and
$-E_2$ and that $\GL_2(\Z)$ is generated by $S$, $T$, $-E_2$ and
$D_{1,-1}$. The matrices $S$, $T$ and $D_{1,-1}$ act on
$M_{2\times 2}(\Z)$ by conjugation as follows,
\begin{eqnarray*}
T^{-1}\begin{pmatrix}a&b\\c&d\end{pmatrix}T &=&
\begin{pmatrix}a-c&a+b-c-d\\c&d+c\end{pmatrix},\\
S^{-1}\begin{pmatrix}a&b\\c&d\end{pmatrix}S &=&
\begin{pmatrix}d&-c\\-b&a\end{pmatrix},\\
D_{1,-1}^{-1}\begin{pmatrix}a&b\\c&d\end{pmatrix}D_{1,-1} &=&
\begin{pmatrix}a&-b\\-c&d\end{pmatrix}.
\end{eqnarray*}

The algorithm in the proof of part (c)(i) of Theorem \ref{t7.10}
is ascribed to Legendre in \cite[7.1.6]{Tr13}.  
Theorem \ref{t7.10} treats the types IV and V largely together.

\begin{theorem}\label{t7.10}
Let $f(t)=t^2-rt+s\in\Z[t]$ be irreducible with roots
$\lambda_{1/2}$ and discriminant $D$ as above.

\begin{enumerate}
\item[(a)] 
Each matrix $B=\begin{pmatrix}a&b\\c&d\end{pmatrix}
\in M_{2\times 2}(\Z)$ with characteristic polynomial $f$
satisfies $bc\neq 0$. If $D<0$ then $bc<0$.
\item[(b)] 
The set
\begin{eqnarray*}
M(r,s)&:=& \{\begin{pmatrix}a&b\\c&d\end{pmatrix}\in 
M_{2\times 2}(\Z)\,|\, a+d=r,\ ad-bc=s,\\
&& \hspace*{2cm}0<|c|\leq |b|,\ a-d\in (-|c|,|c|],\\ 
&&\hspace*{2cm}\textup{in the case }|c|=|b|\quad a-d\in[0,|c|]\}
\end{eqnarray*}
is finite and not empty.
\item[(c)] 
\begin{enumerate}
\item[(i)] 
Each $\SL_2(\Z)$-conjugacy class of matrices in 
$M_{2\times 2}(\Z)$ with characteristic polynomial $f$
has representatives in $M(r,s)$.
\item[(ii)] 
In the case of type IV these representatives are called
{\sf semi-normal forms}.
\item[(iii)] 
In the case of type V there is exactly one such 
representative. It is called {\sf normal form}.
\end{enumerate}
\end{enumerate}
\end{theorem}

{\bf Proof:}
(a) $4D=(a-d)^2+4bc$ is not a square because $f$ is 
irreducible. Therefore $bc\neq 0$. If $D<0$ then $bc<0$.

(b) The case $D<0$: $|c|\leq |b|$ and $a-d\in (-|c|,|c|]$ imply
\begin{eqnarray*}
c^2&\leq& |b||c|=-bc=\bigl(\frac{a-d}{2}\bigr)^2-D
\leq \frac{1}{4}c^2-D,\quad\textup{so}\\
c^2&\leq& \frac{4}{3}|D|,\quad\textup{so }|c|\leq 
\sqrt{\frac{4}{3}|D|},\\
|a-\frac{r}{2}|&=&|d-\frac{r}{2}|=|\frac{a-d}{2}|
\leq \frac{|c|}{2}\leq \sqrt{\frac{1}{3}|D|}.
\end{eqnarray*}
$|c|\in\N$, so $|c|\geq 1$, so 
\begin{eqnarray*}
|b|&\leq& \frac{1}{4}c^2-D\leq \frac{4}{3}|D|.
\end{eqnarray*}
Therefore the set $M(r,s)$ is finite.

The case $D>0$: First we show $bc>0$. Suppose $bc<0$. Then
\begin{eqnarray*}
c^2&\leq& |b||c|=-bc=\bigl(\frac{a-d}{2}\bigr)^2-D
\leq \frac{1}{4}c^2-D,\quad\textup{so}\\
\frac{3}{4}c^2&\leq& -D<0,
\end{eqnarray*}
a contradiction. So $bc>0$. Then
\begin{eqnarray*}
c^2&\leq& bc=D-\bigl(\frac{a-d}{2}\bigr)^2\leq D,
\quad\textup{so }|c|\leq \sqrt{D},\\
|a-\frac{r}{2}|&=&|d-\frac{r}{2}|=|\frac{a-d}{2}|
\leq \frac{|c|}{2}\leq \frac{1}{2}\sqrt{|D|},\\
|b|&\leq& D\quad\textup{because }|c|\geq 1.
\end{eqnarray*}
Therefore the set $M(r,s)$ is finite. 

In both cases $D<0$ and $D>0$ the set $M(r,s)$ is not empty 
because of part (c)(i) and because one has always the 
$\SL_2(\Z)$-conjugacy class of the matrix
$\begin{pmatrix}0&-s\\1&r\end{pmatrix}$.

(c) (i) In each step of the following algorithm for constructing
a representative in $M(r,s)$, a given matrix in 
$M_{2\times 2}(\Z)$ with characteristic polynomial 
$f(t)=t^2-rt+s$ is called 
$B=\begin{pmatrix}a&b\\c&d\end{pmatrix}$.

{\bf Step 1:} Go from $B$ to $\www{B}=T^{-k}BT^k$ for the
unique $k\in\Z$ with $a-d-2kc\in(-|c|,|c|]$. Then
$\www{c}=c$, $\www{a}-\www{d}\in (-|c|,|c|]$.
If then $|\www{c}|\leq |\www{b}|$ go to Step 3, else go to
Step 2.

{\bf Step 2:} Here $|b|<|c|$. Go from $B$ to 
$\www{B}=S^{-1}BS$. Then $|\www{c}|=|b|<|c|$. Go to Step 1.

At some point one goes from Step 1 to Step 3 because 
$|b|,|c|\in\N$ and each application of Step 2 diminishes 
$|c|$ strictly.

{\bf Step 3:} Now $|c|\leq |b|$ and $a-d\in (-|c|,|c|]$.
Stop if $|c|<|b|$ or if $|c|=|b|$ and $a-d\in[0,|c|]$.
If $|c|=|b|$ and $a-d\in (-|c|,0)$ go from $B$ to
$\www{B}=S^{-1}BS$. Then $|\www{c}|=|b|=|c|=|\www{b}|$ and
$\www{a}-\www{d}=d-a\in (0,|\www{c}|)$. 

(ii) Nothing has to be proved.

(iii) Now $D<0$. Consider one $\SL_2(\Z)$-conjugacy class of
matrices in $M_{2\times 2}$ with characteristic polynomial
$f$. Either all matrices 
$\begin{pmatrix}a&b\\c&d\end{pmatrix}$ in it satisfy $c>0$ 
(1st case) or all matrices in it satisfy $c<0$ (2nd case).
For each matrix $B=\begin{pmatrix}a&b\\c&d\end{pmatrix}$
in it define
\begin{eqnarray*}
z_B&:=&\left\{\begin{array}{ll}
\frac{1}{c}(-d+\lambda_1)=\frac{1}{c}(\frac{a-d}{2}+\sqrt{D}) &
\textup{ in the 1st case,}\\
\frac{1}{c}(-d+\lambda_2)=\frac{1}{c}(\frac{a-d}{2}-\sqrt{D}) &
\textup{ in the 2nd case.}
\end{array}\right.
\end{eqnarray*}
Then $z_B$ is in the upper half plane $\H$, and 
$\begin{pmatrix}z_B\\1\end{pmatrix}$ is an eigenvector of $B$
with eigenvalue $\lambda_1$ in the 1st case and eigenvalue
$\lambda_2$ in the 2nd case. 
For $A=\begin{pmatrix}\alpha&\beta\\ \gamma&\delta\end{pmatrix}
\in \SL_2(\Z)$ 
$$z_{ABA^{-1}}=\frac{\alpha z_B+\beta}{\gamma z_B+\delta}.$$
It is well known (e.g. \cite[VII \S 1]{Se73}) that the
$\PSL_2(\Z)$ orbit of any $z\in\H$ intersects each of the two
fundamental domains
\begin{eqnarray*}
F_1&:=& \{z\in\Z\,|\, \Re z\in (-\frac{1}{2},
\frac{1}{2}],\ |z|\geq 1,\ \textup{if }|z|=1\textup{ then }
\Re z\in[0,\frac{1}{2}]\},\\
F_2&:=& \{z\in\Z\,|\, \Re z\in {}[-\frac{1}{2},
\frac{1}{2}),\ |z|\geq 1,\ \textup{if }|z|=1\textup{ then }
\Re z\in[-\frac{1}{2},0]\},
\end{eqnarray*}
of the action of $\PSL_2(\Z)$ on $\H$ in precisely one point.
With
\begin{eqnarray*}
\Re z_B&=&\frac{a-d}{2c}\quad\textup{and}\\
|z_B|^2&=& |\frac{(a-d)/2\pm\sqrt{D}}{c}|^2
=\frac{1}{c^2}\Bigl(\bigl(\frac{a-d}{2}\bigr)^2-D\Bigr)
=\frac{|b|}{|c|}
\end{eqnarray*}
one sees easily
\begin{eqnarray*}
\begin{pmatrix}a&b\\c&d\end{pmatrix}\in M(r,s)&\iff&
\left\{\begin{array}{ll}
z_B\in F_1&\textup{ in the 1st case,}\\
z_B\in F_2&\textup{ in the 2nd case.}
\end{array}\right.
\end{eqnarray*}
Therefore the $\SL_2(\Z)$-conjugacy class above contains exactly
one matrix in $M(r,s)$.\hfill$\Box$

\begin{examples}\label{t7.11}
(i) The quadratic number field $A=\Q[\sqrt{-5}]$ and its
maximal order $\Lambda_1=\Z[\sqrt{-5}]=\langle 1,\sqrt{-5}
\rangle_\Z$ are treated in many algebra books because
they form a small example with class number 2. 

Therefore there are two $\GL_2(\Z)$-conjugacy classes of matrices
with characteristic polynomial $p_1(t)=t^2+5$. We will recover
this with Theorem \ref{t7.10} (b) and (c)(iii). 
Here $r=0$, $s=5$, $D=\frac{r^2}{4}-s=-5<0$. 
One sees easily
that the set $M(0,5)$ in Theorem \ref{t7.10} (b) is
\begin{eqnarray*}
M(0,5)=\{\begin{pmatrix}0&-5\\1&0\end{pmatrix}, 
\begin{pmatrix}0&5\\-1&0\end{pmatrix},
\begin{pmatrix}1&-3\\2&-1\end{pmatrix},
\begin{pmatrix}1&3\\-2&-1\end{pmatrix}\}.
\end{eqnarray*}
Therefore there are four $\SL_2(\Z)$-conjugacy classes
and two $\GL_2(\Z)$-conjugacy classes with the representatives
$\begin{pmatrix}0&-5\\1&0\end{pmatrix}$ and 
$\begin{pmatrix}1&-3\\2&-1\end{pmatrix}$. 
Theorem \ref{t7.9} (a)(i) provides full lattices 
$\langle 1,\sqrt{-5}\rangle_\Z=\Lambda_1$ and 
$\langle 2,-1+\sqrt{-5}\rangle_\Z=:L_0$ and $\Z$-bases
of them which give rise to these matrices,
\begin{eqnarray*}
\sqrt{-5}(1,\sqrt{-5})&=&(1,\sqrt{-5})
\begin{pmatrix}0&-5\\ 1&0\end{pmatrix},\\
\sqrt{-5}(2,-1+\sqrt{-5})&=& (2,-1+\sqrt{-5})
\begin{pmatrix}1&-3\\ 2&-1\end{pmatrix}.
\end{eqnarray*}
The corresponding group $G([\Lambda_1]_\varepsilon)$ of 
$\varepsilon$-classes of full lattices is
$G([\Lambda_1]_\varepsilon)=\{[\Lambda_1]_\varepsilon,
[L_0]_\varepsilon\}$ 
with unit element $[\Lambda_1]_\varepsilon$. 

(ii) Now we are interested in matrices in $M_{2\times 2}(\Z)$
with characteristic polynomial $p_2(t)=t^2+20$. The algebra
$A=\Q[\sqrt{-20}]=\Q[\sqrt{-5}]$ is the same as in part (i).
The $\GL_2(\Z)$-conjugacy classes of these matrices 
correspond by Theorem \ref{t6.3}
to the $\varepsilon$-classes of full lattices $L$
with $\OO(L)\supset \Z[\sqrt{-20}]=\langle 1,2\sqrt{-5}]\rangle_\Z
=:\Lambda_2$. Only the order $\Lambda_1$ is above $\Lambda_2$.
Therefore
\begin{eqnarray*}
\{[L]_\varepsilon\,|\, L\in\LL(A),\OO(L)\supset \Lambda_2\}
=G([\Lambda_1]_\varepsilon)\,\cup\, 
G([\Lambda_2]_\varepsilon).
\end{eqnarray*}
Theorem \ref{t5.12} (d) gives a surjective group homomorphism
$$G([\Lambda_2]_\varepsilon)\to G([\Lambda_1]_\varepsilon),\quad
[L]_\varepsilon\mapsto [\Lambda_1]_\varepsilon [L]_\varepsilon$$
and allows to determine the size of the finite commutative group
$G([\Lambda_2]_\varepsilon)$. 

But applying Theorem \ref{t7.10} (b) and (c)(iii)
is faster and gives more details. Here $r=0$, $s=20$,
$D=\frac{r^2}{4}-s=-20<0$. The proof of Theorem \ref{t7.10} (b)
gives for matrices $\begin{pmatrix}a&b\\c&d\end{pmatrix}
\in M(0,20)$ the (in)equalities
\begin{eqnarray*}
0<|c|\leq \Bigl\lfloor\sqrt{\frac{4}{3}|D|}\Bigr\rfloor =5,\\
d=-a,\ 2a\in(-|c|,|c|],\textup{ so }a\in[-2,2],\\
\textup{if }|b|=|c|\textup{ then }a\in[0,2],\\
|c|\leq |b|\leq \Bigl\lfloor \frac{4}{3}|D|\Bigr\rfloor =26.
\end{eqnarray*}

The case $a=2$: $-bc=24$, $|c|\geq 2a=4$, 
$\begin{pmatrix}2&-6\\4&-2\end{pmatrix},\ 
\begin{pmatrix}2&6\\-4&-2\end{pmatrix}$.

The case $a=1$: $-bc=21$, $|c|\geq 2a=2$, 
$\begin{pmatrix}1&-7\\3&-1\end{pmatrix},\ 
\begin{pmatrix}1&7\\-3&-1\end{pmatrix}$.

The case $a=0$: $-bc=20$, 
$$\begin{pmatrix}0&-20\\ 1&0 \end{pmatrix},  
\begin{pmatrix}0&20\\ -1&0 \end{pmatrix}, 
\begin{pmatrix}0&-10\\ 2&0 \end{pmatrix}, 
\begin{pmatrix}0&10\\ -2&0 \end{pmatrix}, 
\begin{pmatrix}0&-5\\ 4&0 \end{pmatrix}, 
\begin{pmatrix}0&5\\ -4&0 \end{pmatrix}.$$

The case $a=-1$: $-bc=21$, $|c|>2|a|=2,
\begin{pmatrix}-1&-7\\3&1\end{pmatrix},\ 
\begin{pmatrix}-1&7\\-3&1\end{pmatrix}$.

The case $a=-2$: $-bc=24$, $|c|>2|a|=4$, $|c|\leq |b|$,
impossible.

$M(0,20)$ contains 12 matrices, so there are 12 $\SL_2(\Z)$
conjugacy classes and 6 $\GL_2(\Z)$-conjugacy classes of matrices
with characteristic polynomial $p_2(t)=t^2+20$, with
representatives those 6 matrices 
$B=\begin{pmatrix}a&b\\c&d\end{pmatrix}$ of the 12 matrices
above which satisfy $c>0$. Theorem \ref{t7.9} (a)(i) provides
for such a matrix $B$ the full lattice $L$ with $\Z$-basis
$(c,-a+\sqrt{-20})$ with 
$$\sqrt{-20}(c,-a+\sqrt{-20})=(c,-a+\sqrt{-20})B.$$
\begin{eqnarray*}
\begin{pmatrix}0&-10\\2&0\end{pmatrix}:&&
\langle 2,\sqrt{-20}\rangle_\Z =2\langle 1,\sqrt{-5}\rangle_\Z
=2\Lambda_1,\\
\begin{pmatrix}2&-6\\4&-2\end{pmatrix}:&&
\langle 4,-2+\sqrt{-20}\rangle_\Z 
=2\langle 2,-1+\sqrt{-5}\rangle_\Z =2 L_0,\ \OO(L_0)=\Lambda_1,\\
\begin{pmatrix}0&-20\\1&0\end{pmatrix}:&&
\langle 1,\sqrt{-20}\rangle_\Z =\langle 1,2\sqrt{-5}\rangle_\Z
=\Lambda_2,\\
\begin{pmatrix}0&-5\\4&0\end{pmatrix}:&&
\langle 4,\sqrt{-20}\rangle_\Z =2\langle 2,\sqrt{-5}\rangle_\Z
=: 2L_2,\ \OO(L_2)=\Lambda_2,\\
\begin{pmatrix}1&-7\\3&-1\end{pmatrix}:&&
\langle 3,-1+\sqrt{-20}\rangle_\Z =:L_1,\ \OO(L_1)=\Lambda_2,\\
\begin{pmatrix}-1&-7\\3&1\end{pmatrix}:&&
\langle 3,1+\sqrt{-20}\rangle_\Z =:L_3,\ \OO(L_3)=\Lambda_2.
\end{eqnarray*}
Therefore 
$$G([\Lambda_2]_\varepsilon)=\{[\Lambda_2]_\varepsilon,
[L_2]_\varepsilon,[L_1]_\varepsilon,[L_3]_\varepsilon\}.$$
It turns out that this group is cyclic of order four
with generators $[L_1]_\varepsilon$ and $[L_3]_\varepsilon$
and that therefore the surjective group homomorphism
$$G([\Lambda_2]_\varepsilon\to G([\Lambda_1]_\varepsilon,\quad
[L]_\varepsilon\mapsto [\Lambda_1]_\varepsilon [L]_\varepsilon,$$
maps $[\Lambda_2]_\varepsilon$ and $[L_2]_\varepsilon$ to
$[\Lambda_1]_\varepsilon$ and $[L_1]_\varepsilon$ and
$[L_3]_\varepsilon$ to $[L_0]_\varepsilon$. This can all 
be read off from the following multiplication table for
$\Lambda_1,L_0,\Lambda_2,L_2,L_1,L_3$. Its calculation
uses the norms of $1+\sqrt{-5}$ and $2+\sqrt{-5}$, which are
$$(1+\sqrt{-5})(1-\sqrt{-5})=6,\quad
(2+\sqrt{-5})(2-\sqrt{-5})=9.$$
\begin{table}[H]
\begin{eqnarray*}
\begin{array}{l|ccccccc}
 & \Lambda_1 & L_0 & \Lambda_2 & L_2 & L_1 & L_3 \\ \hline 
\Lambda_1 & \Lambda_1 & L_0 & \Lambda_1 & \Lambda_1 & 
\frac{1+\sqrt{-5}}{2}L_0 & \frac{1-\sqrt{-5}}{2}L_0 \\
L_0 & & 2\Lambda_1 & L_0 & L_0 & (1+\sqrt{-5})\Lambda_1 & 
(1-\sqrt{-5})\Lambda_1 \\
\Lambda_2 & & & \Lambda_2 & L_2 & L_1 & L_3 \\
L_2 & & & & \Lambda_2 & \frac{2-\sqrt{-5}}{3}L_3 & 
\frac{2+\sqrt{-5}}{3}L_1 \\
L_1 & & & & & (2-\sqrt{-5})L_2 & 3\Lambda_2 \\
L_3 & & & & & & (2+\sqrt{-5})L_2
\end{array}
\end{eqnarray*}
\caption[Table 7.1]{Multiplication table of  
$\Lambda_1,L_0,\Lambda_2,L_2,L_1,L_3$}
\label{tab7.1}
\end{table}
\end{examples}

\subsection{Conway's topograph}
\label{s7.4}

In the case of type IV, Theorem \ref{t7.10} (c)(ii) leads to
a finite family of semi-normal forms for an 
$\SL_2(\Z)$-conjugacy class of matrices in $M_{2\times 2}(\Z)$.
The purpose of this subsection is to make this family and two
other families of semi-normal forms very transparent.
This is possible thanks to an idea of Conway.
He proposed a beautiful geometric way to
see binary quadratic forms \cite[ch. 1]{Co97}. 
Behind it is an infinite binary tree
embedded in the plane, the {\it topograph}, which serves as an
atlas for primitive vectors and $\Z$-bases of a rank 2 
$\Z$-lattice. First we recall this general construction.

\begin{definition}\label{t7.12}
\begin{enumerate}
\item[(a)] 
An {\it infinite binary tree} is an infinite (undirected)
graph which is a tree and where at each vertex three edges meet.
\item[(b)] 
A {\it topograph} $(D,G)$ is a contractible oriented
$C^\infty$ surface $D$ (e.g. $\R^2$ or $\H$ or the unit disk,
it does not matter) with an embedded infinite binary tree
$G\subset D$ (the edges may be embedded as $C^\infty$ curves
which meet transversely at the vertices), such that two points
which are close to one edge, but on different sides of that edge
are not in the same connected component of $D-G$.
\item[(c)] 
Let $(D,G)$ be a topograph. The connected components
$\rho$ of $D-G$ are called {\it regions}. The set of all regions
is called $R_1$. Two regions which share an edge are called
{\it neighbors}. The set of all ordered pairs $(\rho_1,\rho_2)$
of neighbors is called $R_2$, so $R_2\subset R_1\times R_1$.
In the figures below a pair $(\rho_1,\rho_2)$ of neighbors
is indicated by a bullet point near the common edge of 
$\rho_1$ and $\rho_2$ in the region $\rho_1$. 
Later the pair $(\rho_1,\rho_2)$ will be identified with 
the common edge as an oriented edge, where the orientation will be 
indicated by an arrow such that looking along the arrow,
$\rho_1$ is on the left side. So $R_2$ will be identified
with the set of oriented edges.

Two actions 
$$\sigma:R_2\to R_2\quad\textup{and}\quad \tau:R_2\to R_2$$
are defined as follows,
\begin{eqnarray*}
\sigma:(\rho_1,\rho_2)&\mapsto&(\rho_2,\rho_1),\\
\tau:(\rho_1,\rho_2)&\mapsto& (\rho_1,\rho_3),
\end{eqnarray*}
where $\rho_3$ is the neighbor of $\rho_1$ behind that edge
which is the next edge to the left (looking from $\rho_1$)
of the edge between $\rho_1$ and $\rho_2$. 
Figure \ref{fig:7.1} shows these actions on a small part of $R_2$.

\begin{figure}
\includegraphics[width=0.5\textwidth]{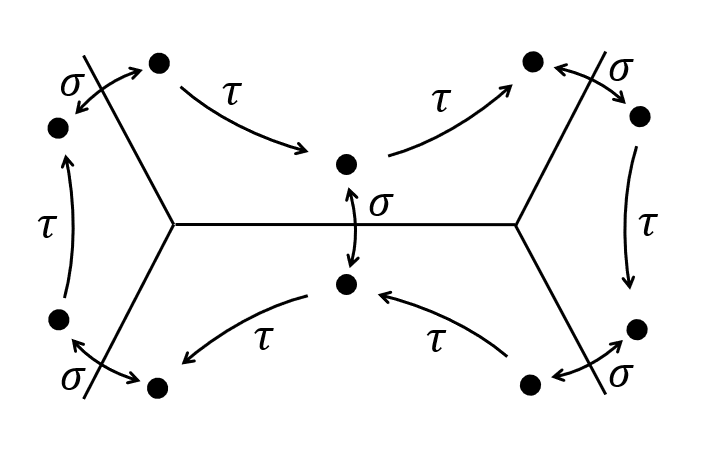}
\caption[Figure 7.1]{Action of $\sigma$ and $\tau$ on $R_2$}
\label{fig:7.1}
\end{figure}

These actions extend to an action on $R_2$ of the group
$\langle\sigma,\tau\rangle\subset \{\textup{bijections }
R_2\to R_2\}$ {\it from the right}.
\item[(d)] 
Let $H_\Z$ be a rank 2 $\Z$-lattice.
The set of $\Z$-bases of $H_\Z$ splits into two orientation
classes $o_1$ and $o_2$.

A {\it lax vector} is a set $\{\pm v\}$ where $v\in H_\Z$ is
a primitive vector, i.e. a vector whose coefficients with
respect to some $\Z$-basis are coprime.
Let $L_1$ be the set of all lax vectors. 

A {\it lax basis} is a set $\{\pm (v_1,v_2)\}$ where
$(v_1,v_2)\in o_1$. Let $L_2$ be the set of all lax bases.
\item[(e)] 
The class in $\PSL_2(\Z)$ of a matrix 
$A=\begin{pmatrix}\alpha&\beta\\ \gamma&\delta\end{pmatrix}
\in \SL_2(\Z)$ is denoted by $\oooo{A}$. 
Let $(H_\Z,o_1,L_1,L_2)$ be as above. The group $\PSL_2(\Z)$ acts
simply transitively from the right on $L_2$ by
\begin{eqnarray*}
\oooo{A}:\pm (v_1,v_2)\mapsto \pm (v_1,v_2)A
=\pm(\alpha v_1+\gamma v_2,
\beta v_1+\delta v_2).
\end{eqnarray*}
\end{enumerate}
\end{definition}

\begin{lemma}\label{t7.13}
Let $(D,G,R_1,R_2,\sigma,\tau)$ be as in Definition \ref{t7.12}
(b)+(c), and let $(H_\Z,o_1,L_1,L_2)$ be as in Definition
\ref{t7.12} (d).
Choose one pair $(\rho_1,\rho_2)\in L_2$ of neighboring regions
and one lax basis $\pm(v_1,v_2)$. These choices give rise to
two unique bijections
\begin{eqnarray*}
\psi_1:R_1\to L_1\quad\textup{and}\quad\psi_2:R_2\to L_2
\end{eqnarray*}
with the following properties, where $(\rho_3,\rho_4)\in R_2$ is arbitrary:
\begin{eqnarray}
\psi_1(\rho_1)&=&\pm v_1,\quad \psi_1(\rho_2)=\pm v_2,
\label{7.7}\\
\psi_2((\rho_3,\rho_4))&=&
\{\pm\psi_1(\rho_3),\pm\psi_1(\rho_4)\} \cap o_1,
\label{7.8}\\
\textup{especially }&&\psi_2((\rho_1,\rho_2))=\pm (v_1,v_2),
\label{7.9}\\
\psi_2((\rho_3,\rho_4)\sigma)&=&\psi_2((\rho_3,\rho_4))\cdot 
\oooo{S},\label{7.10}\\
\psi_2((\rho_3,\rho_4)\tau)&=&\psi_2((\rho_3,\rho_4))\cdot
\oooo{T}.\label{7.11}
\end{eqnarray}
In this way, the regions and the pairs of neighboring regions
in the topograph give an atlas for the lax vectors and for the
lax bases of $H_\Z$. Figure \ref{fig:7.2} shows this atlas in a 
small part of the topograph.
\end{lemma}

\begin{figure}
\includegraphics[width=1.0\textwidth]{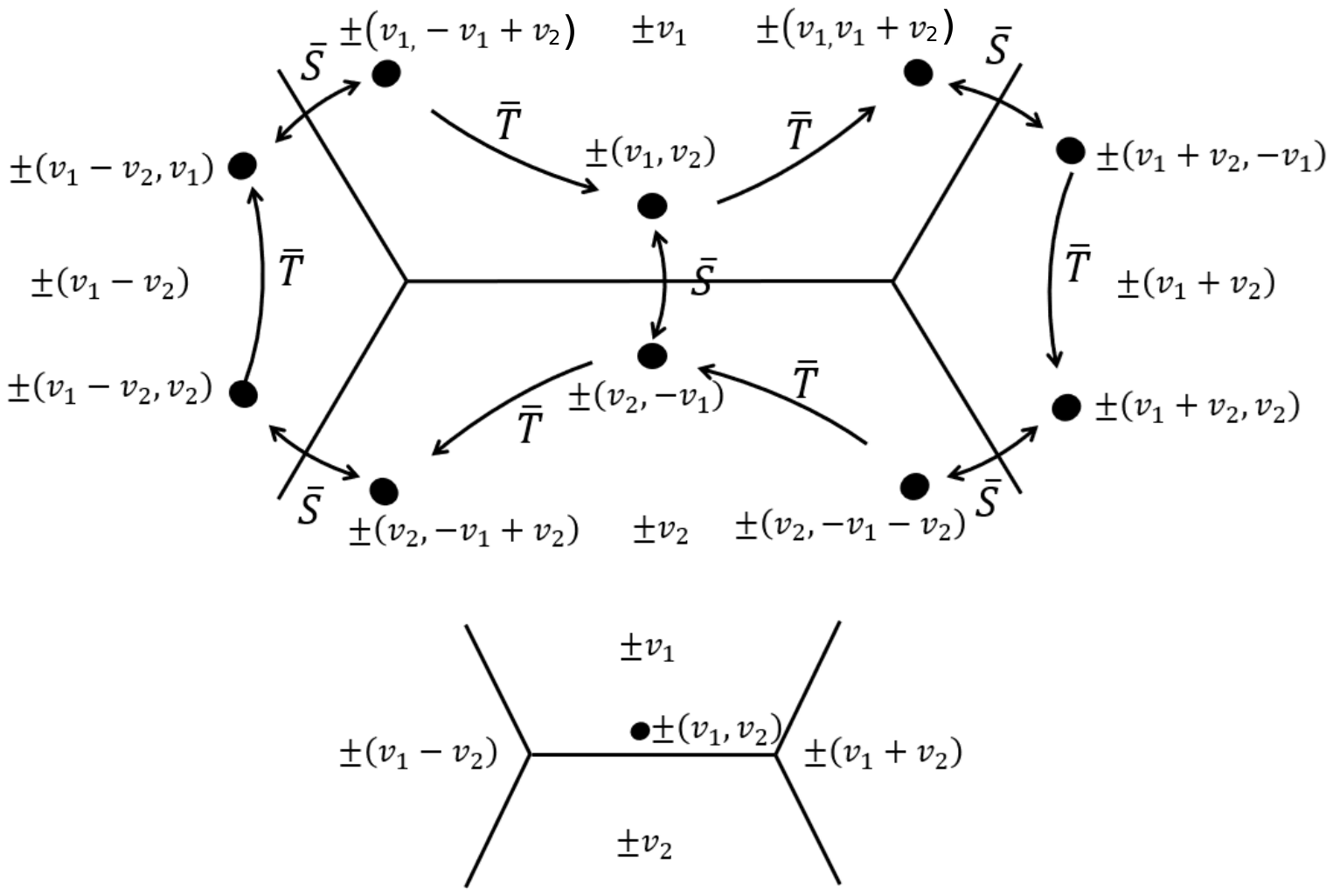}
\caption[Figure 7.2]{Action of $\PSL_2(\Z)$ on $L_2$ and the maps
$\psi_1:R_1\to L_1$ and $\psi_2:R_2\to L_2$}
\label{fig:7.2}
\end{figure}

{\bf Proof:} Because $G$ is an infinite binary tree, the 
subgroups $\langle \sigma\rangle$ and $\langle\tau\sigma\rangle$
of $\langle \sigma,\tau\rangle$ have order 2 respectively 3,
and $\langle \sigma,\tau\rangle$ is their free product.
It is well known that the group $\PSL_2(\Z)$ is the free product
of the subgroups $\langle \oooo{S}\rangle$ and
$\langle \oooo{T}\oooo{S}\rangle$ of order 2 respectively 3.
Therefore there is an isomorphism of groups
\begin{eqnarray*}
\langle\sigma,\tau\rangle&\to& \PSL_2(\Z)\quad\textup{with}\quad
\sigma\mapsto \oooo{S},\ \tau\mapsto\oooo{T}.
\end{eqnarray*}
The group $\langle \sigma,\tau\rangle$ acts simply transitively
on $R_2$. The group $\PSL_2(\Z)$ acts simply transitively on 
$L_2$. Therefore there is a unique bijection
$$\psi_2:R_2\to L_2$$
with \eqref{7.9}, \eqref{7.10} and \eqref{7.11}.

It remains to see the bijection $\psi_1:R_1\to L_1$ with
\eqref{7.7} and the compatibility \eqref{7.8}.
Fix one region $\rho_3$. The subgroup $\langle\tau\rangle$
acts simply transitively on the set of all pairs
$(\rho_3,\rho_4)$ of neighboring regions with first region
$\rho_3$. Because of \eqref{7.11} the first element of the
corresponding lax bases $\psi_2((\rho_3,\rho_4))$ is the same
for all these pairs. Define $\psi_1:R_1\to L_1$ by
\begin{eqnarray*}
\psi_1(\rho_3)&:=&(\textup{the common first lax vector of the
lax pairs }\psi_2((\rho_3,\rho_4))\\
&&\textup{for all regions }\rho_4\textup{ which are neighbors
of }\rho_3).
\end{eqnarray*}
Suppose $(\rho_3,\rho_4)\in R_2$. Then
\begin{eqnarray*}
\psi_1(\rho_4)&=& (\textup{the first element of the lax basis }
\psi_2((\rho_4,\rho_3)))\\
&\stackrel{\eqref{7.10}}{=}& (\textup{the first element of the
lax basis }\psi_2((\rho_3,\rho_4))\oooo{S})\\
&=& (\textup{the second element of the lax basis }
\psi_2((\rho_3,\rho_4))).
\end{eqnarray*}
This shows \eqref{7.8} and \eqref{7.7}.
$\psi_1$ is a bijection because $\psi_2$ is a bijection.
\hfill$\Box$

\begin{remarks}\label{t7.14}
\begin{enumerate}
\item[(i)] 
Consider $H_\Z=\Z^2=M_{2\times 1}(\Z)$ with orientation
class $o_1$ of $\Z$-bases such that 
$(\begin{pmatrix}1\\0\end{pmatrix},
\begin{pmatrix}0\\1\end{pmatrix})\in o_1$. The map
\begin{eqnarray*}
\varphi_2: \PSL_2(\Z)\to L_2,\quad \oooo{A}=
\oooo{\begin{pmatrix}\alpha&\beta\\ \gamma&\delta\end{pmatrix}}
\mapsto \pm (\begin{pmatrix}\alpha\\ \gamma\end{pmatrix},
\begin{pmatrix}\beta\\ \delta\end{pmatrix}),
\end{eqnarray*}
is a bijection which is compatible with the right multiplication
of $\PSL_2(\Z)$ on $\PSL_2(\Z)$ and the action from the right
in Definition \ref{t7.12} (e) of $\PSL_2(\Z)$ on $L_2$.
\item[(ii)] 
In the case of $H_\Z=\Z^2$ all the data in Definition 
\ref{t7.12} and Lemma \ref{t7.13} can be made explicit as 
follows. Choose $D$ as the upper half plane $\H$.
The group $\PSL_2(\Z)$ acts on $\H$ properly discontinuously
with fundamental domain $F_1$ in the proof of Theorem \ref{t7.10}
(c)(iii). This fundamental domain $F_1$ and all its images 
$\oooo{A}F_1$ for $A\in \SL_2(\Z)$ are degenerate hyperbolic
triangles with one degenerate vertex in $\Q\cup\{\infty\}
\subset\R\cup\{\infty\}=\partial \H$. The map
\begin{eqnarray*}
\chi_2:\{\oooo{A}F_1\,|\, \oooo{A}\in \PSL_2(\Z)\}\to L_2,\quad
\oooo{A}F_1\mapsto \varphi_2(\oooo{A}),
\end{eqnarray*}
is a bijection which is compatible with the actions of
$\PSL_2(\Z)$ on the right. Here $\oooo{A}_2\in \PSL_2(\Z)$
maps $\oooo{A}_1F_1$ to $\oooo{A}_1\oooo{A}_2F_1$. 

The lax bases $\varphi_2(\oooo{A}_1F_1)$ and 
$\varphi_2(\oooo{A}_3F_1)$ have the same first lax vector
if and only if $\oooo{A}_1F_1$ and $\oooo{A}_3F_1$ have the
same degenerate vertex in $\Q\cup\{\infty\}$. 

The binary tree $G\subset D=\H$ and the regions in $D-G$ 
are defined simultaneously as follows. 
For each $q\in\Q\cup\{\infty\}$ the set
\begin{eqnarray*}
\rho(q)&:=&\{\textup{the interior of the union of all }
\oooo{A}F_1\\
&&\textup{with degenerate vertex }q\}\subset\H
\end{eqnarray*}
is a region, and the binary tree $G$ is
$$G=\H-\cup_{q\in\Q\cup\{\infty\}}\rho(q).$$
\cite[page 247]{We17} and \cite[4.1, page 78]{Ha21} show 
pictures where $D$ is the unit disk
instead of the upper half plane.
\end{enumerate}
\end{remarks}

Now we apply the general constructions in Definition \ref{t7.12}
and Lemma \ref{t7.13} to give a {\it value atlas} for an  
abstract binary quadratic form $(H_\Z,o_1,q)$. 
Definition \ref{t7.15} defines the value atlas.
Lemma \ref{t7.16} states some basic properties.
Recall from Definition \ref{t7.4} 
that in an abstract binary quadratic form,
$H_\Z$ is a rank 2 $\Z$-lattice, $o_1$ is one of the two
orientation classes of $\Z$-bases of $H_\Z$, and 
$q:H_\Z\to\Z$ is a map such that for any $\Z$-basis 
$\uuuu{v}=(v_1,v_2)\in o_1$ the map
$q_{\uuuu{v}}:\Z^2\to\Z$ with $q_{\uuuu{v}}(\uuuu{x})=
q(x_1v_1+x_2v_2)$ is a binary quadratic form 
$[a,h,b]_{quad}$ for some $a,h,b\in\Z$.
Running through all $\Z$-bases in $o_1$, one obtains a proper
equivalence class of binary quadratic forms.
Observe $q_{-\uuuu{v}}=q_{\uuuu{v}}$ and $q(-v_1)=q(v_1)$.

\begin{definition}\label{t7.15}
Let $(H_\Z,o_1,q)$ be an abstract binary quadratic form.
Let $L_1$ and $L_2$ be the set of lax vectors respectively
lax bases of $H_\Z$. Then $(H_\Z,o_1,L_1,L_2)$ is as in 
Definition \ref{t7.12} (d). 
Let $(D,G,R_1,R_2,\sigma,\tau)$ be as in Definition \ref{t7.12}
(b)+(c). Choose bijections $\psi_1:R_1\to L_1$ and
$\psi_2:R_2\to L_2$ with \eqref{7.8}, \eqref{7.10}
and \eqref{7.11}. Define the maps
\begin{eqnarray*}
\omega_1:=q\circ \psi_1:R_1&\to&\Z,\\
\omega_2:R_2&\to&\Z,\quad (\rho_1,\rho_2)\mapsto h\\
&&\textup{where }q_{\psi_2((\rho_1,\rho_2))}=[a,h,b]_{quad}.
\end{eqnarray*}
$(D,G,R_1,R_2,\omega_1,\omega_2)$ is the {\it value atlas}
of $(H_\Z,o_1,q)$. 
In (a part of) the topograph one encodes (part of) the maps
$\omega_1$ and $\omega_2$ 
as follows. One writes the number $\omega_1(\rho)\in\Z$ 
into each region $\rho$. The set $R_2$ is identified with the
set of oriented edges in the binary tree. Here a pair
$(\rho_1,\rho_2)\in R_2$ is mapped to the edge between
$\rho_1$ and $\rho_2$ with an orientation which is indicated
by an arrow on the edge such that $\rho_1$ is to the left.

Observe $h:=\omega_2((\rho_1,\rho_2))
=-\omega_2((\rho_2,\rho_1))$.
If $h=0$ write $0$ near the edge and write no arrow on the edge.
If $h\neq 0$ either write $h$ near the edge and write the
arrow from $(\rho_1,\rho_2)$ on the edge or write $-h$ near
the edge and write the arrow from $(\rho_2,\rho_1)$ on the
edge. If $h\neq 0$ more often we choose the arrow such that 
the label of the edge is $|h|$, but not always.
The two choices are equivalent and can be recovered from one
another. See Figure \ref{fig:7.3}. There $\omega_1(\rho_1)=a$,
$\omega_1(\rho_2)=b$, $\omega_2((\rho_1,\rho_2))=h$.
\end{definition}

\begin{figure}
\includegraphics[width=0.7\textwidth]{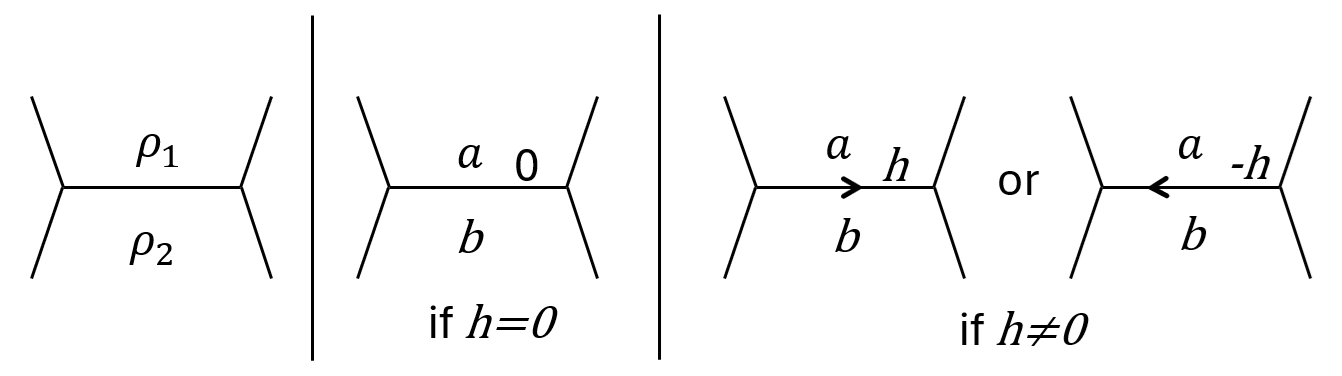}
\caption[Figure 7.3]{Encoding the value atlas by labeling
regions and oriented edges}
\label{fig:7.3}
\end{figure}

\begin{lemma}\label{t7.16}
Consider the data in Definition \ref{t7.15}.

\begin{enumerate}
\item[(a)] 
In each of the three pictures in Figure \ref{fig:7.3}
(one for $h=0$, two for $h\neq 0$), one recovers from
$(\rho_1,\rho_2)$ and Figure \ref{fig:7.3}
first the triple $(a,h,b)=(\omega_1(\rho_1),
\omega_2((\rho_1,\rho_2)),\omega_1(\rho_2))$ of values,
then the binary quadratic form $q_{\psi_2((\rho_1,\rho_2))}
=[a,h,b]_{quad}$, then the abstract binary quadratic form
$(H_\Z,o_1,q)$, and finally the maps $\omega_1:R_1\to\Z$
and $\omega_2:R_2\to\Z$.
\item[(b)] 
But these maps can also be recovered more directly from
$(\rho_1,\rho_2)$ and the triple $(a,h,b)=(\omega_1(\rho_1),
\omega_2((\rho_1,\rho_2)),\omega_1(\rho_2))$ of values, 
as follows.
\begin{enumerate}
\item[(i)] 
Suppose $\psi_2((\rho_1,\rho_2))=\pm (v_1,v_2)=\pm\uuuu{v}$ 
with $\uuuu{v}=(v_1,v_2)\in o_1$, so especially 
$\psi_1(\rho_1)=\pm v_1$, $\psi_1(\rho_2)=\pm v_2$, 
and suppose $\tau((\rho_1,\rho_2))=(\rho_1,\rho_3)$,
$\tau^{-1}((\rho_1,\rho_2))=(\rho_1,\rho_4)$, so
$\psi_1(\rho_3)=\pm (v_1+v_2)$, $\psi_1(\rho_4)=\pm (v_1-v_2)$.
Then
\begin{eqnarray}
\omega_1(\rho_3)=q(\pm(v_1+v_2))=q_{\uuuu{v}}(1,1)
=a+b+h,\label{7.12}\\
\omega_1(\rho_4)=q(\pm(v_1-v_2))=q_{\uuuu{v}}(1,-1)
=a+b-h,\label{7.13}\\
\textup{so }\omega_1(\rho_3)+\omega_1(\rho_4)=
2(\omega_1(\rho_1)+\omega_1(\rho_2)).\label{7.14}
\end{eqnarray}
See Figure \ref{fig:7.4}.

\begin{figure}
\includegraphics[width=0.9\textwidth]{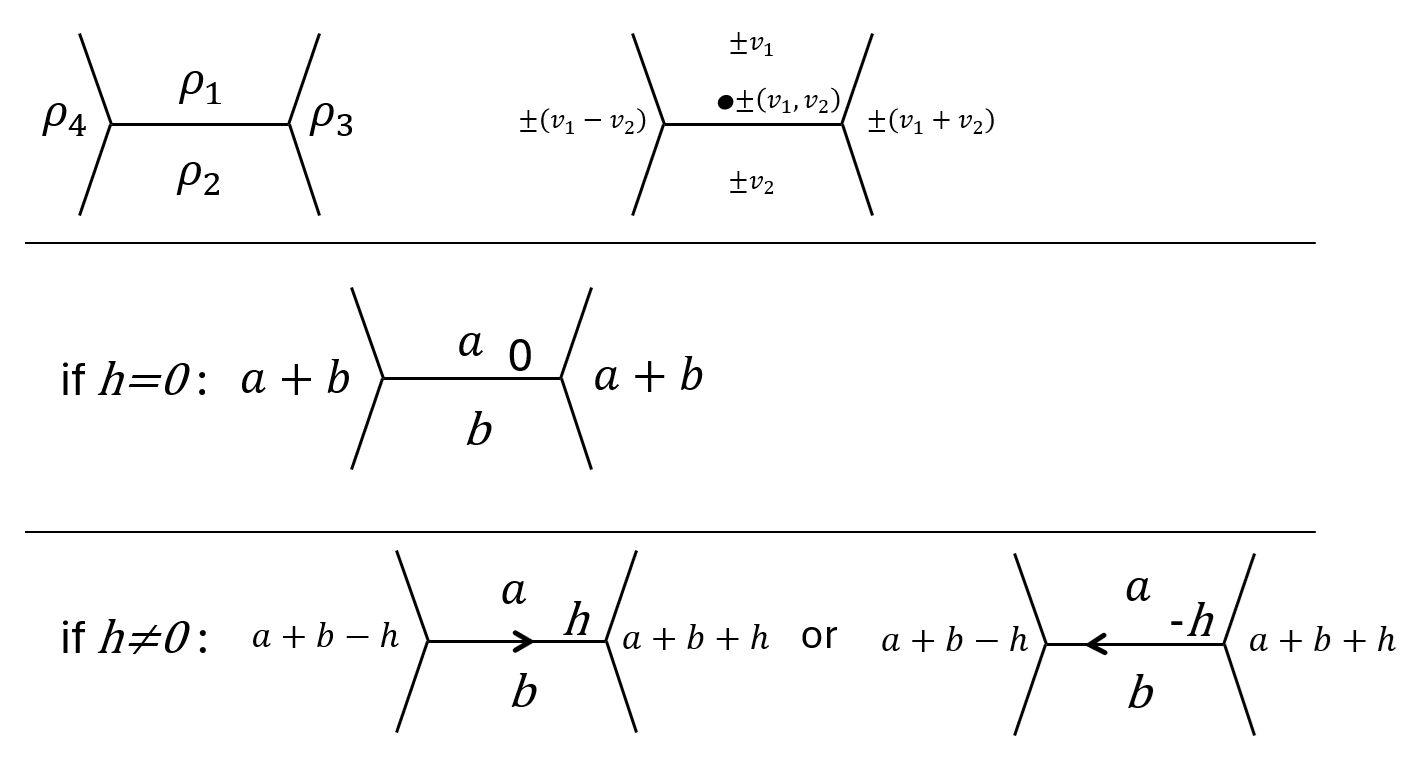}
\caption[Figure 7.4]{From the values 
$(\omega_1(\rho_1)$,$\omega_2((\rho_1,\rho_2))$,
$\omega_1(\rho_2))$
to the values $\omega_1(\rho_3)$ and $\omega_1(\rho_4)$}
\label{fig:7.4}
\end{figure}
\item[(ii)] 
Suppose three regions $\rho_1,\rho_2$ and $\rho_3$ meet
at one vertex with $\tau((\rho_1,\rho_2))=(\rho_1,\rho_3)$,
and $\omega_1(\rho_1)=a$, $\omega_1(\rho_2)=b$, 
$\omega_1(\rho_3)=c$. Then
\begin{eqnarray}
\left.\begin{array}{l}
\omega_2((\rho_1,\rho_2))=c-a-b,\\ 
\omega_2((\rho_2,\rho_3))=a-b-c,\\
\omega_2((\rho_3,\rho_1))=b-a-c.\end{array}\right\}
\label{7.15}
\end{eqnarray}
See Figure \ref{t7.5}.

\begin{figure}
\includegraphics[width=0.6\textwidth]{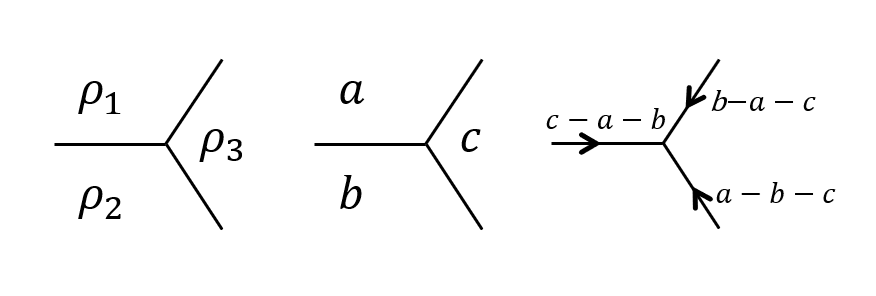}
\caption[Figure 7.5]{From the values of three neighboring
regions to the values of three oriented edges}
\label{fig:7.5}
\end{figure}
\item[(iii)] Suppose $(\rho_1,\rho_2)\in R_2$ with 
$\omega_1(\rho_1)=a$ and $\omega_2((\rho_1,\rho_2))=h$. Then
\begin{eqnarray}
\omega_2((\rho_1,\rho_2)\tau^k)=h+k\cdot 2a\quad\textup{for }
k\in\Z.\label{7.16}
\end{eqnarray}
See Figure \ref{t7.6}.

\begin{figure}
\includegraphics[width=0.5\textwidth]{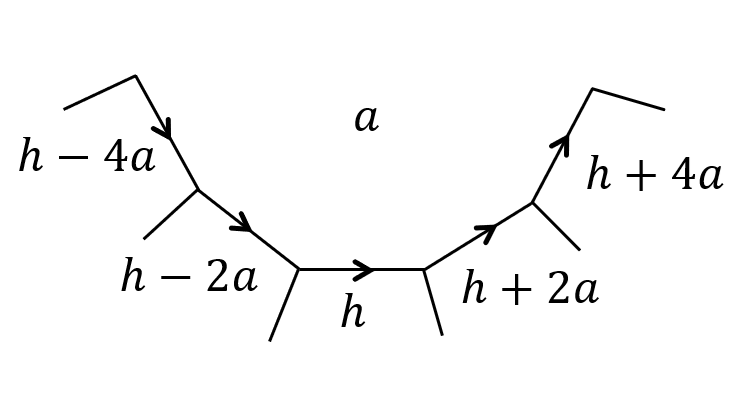}
\caption[Figure 7.6]{The values of all (consistently) oriented
edges of one region form an arithmetic progression}
\label{fig:7.6}
\end{figure}

\end{enumerate}
\end{enumerate}
\end{lemma}

{\bf Proof:}
(a) Clear.

(b) (i) Observe $q_{\uuuu{v}}=[a,h,b]_{quad}$. 

(ii) This follows from part (i).

(iii) Suppose $\tau((\rho_1,\rho_2))=(\rho_1,\rho_3)$.
Write $b:=\omega_1(\rho_2)$. By \eqref{7.12} and \eqref{7.15}
$$\omega_2((\rho_1,\rho_3))=a+(a+b+h)-b=h+2a.$$
Inductively one obtains \eqref{7.16}.\hfill$\Box$

\bigskip
At this point Conway \cite[ch. 1]{Co97}, 
Weissman \cite[ch. 10 and ch. 11]{We17}
and Hatcher \cite[section 5.1]{Ha21} apply Lemma \ref{t7.16}
to binary quadratic forms of different types and discuss
the specific behaviour of distribution of values of the
regions and the oriented edges. 
Here we restrict to type IV, i.e. to indefinite binary quadratic
forms $[a,h,b]_{quad}$ with discriminant 
$D=\frac{h^2}{4}-ab>0$ and $4D\in\N$ not a square.
By Lemma \ref{t7.2} (b) they correspond to matrices of type IV.
The following theorem is proved in 
\cite[ch. 1]{Co97}, 
\cite[Corollary 11.8, Proposition 11.9, Theorem 11.13]{We17} 
and \cite[section 4.2 and Theorem 5.2]{Ha21}.

\begin{theorem}\label{t7.17}
Consider the data in Definition \ref{t7.15}.
Suppose that the abstract binary form $(H_\Z,o_1,q)$
is indefinite with determinant $D>0$ and $4D\in\N$ not a square.

\begin{enumerate}
\item[(a)] 
The binary tree $G$ contains an infinite path of edges
such that the values of all regions on one side are positive
and the values of all regions on the other side are negative.
The path of edges is called {\sf Conway river}.
Let all edges with labels $\neq 0$ be oriented such that the
labels are positive. Then all edges on the negative side
have arrows pointing to the river, and all edges on the 
positive side have arrows pointing away from the river.
\item[(b)] 
The atlas of values on the regions and oriented 
edges is periodic
in the following sense. There is a nontrivial orientation 
preserving automorphism of the pair $(D,G)$ which maps
each region to a region with the same value and each 
oriented edge to an oriented edge with the same value.
It preserves the Conway river. Its restriction to the Conway
river is a translation of the river. 
\end{enumerate}
\end{theorem}

A semi-normal form of the abstract binary quadratic form
$(H_\Z,o_1,q)$ in Theorem \ref{t7.17} is a binary 
quadratic form $[a,h,b]_{quad}$ in the corresponding proper
equivalence class such that the triple $(a,h,b)$ has
certain properties. 

Because of the periodicity in Theorem \ref{t7.17} (b),
each semi-normal form is equal to $q_{\uuuu{v}}$ for 
infinitely many $\Z$-bases $\uuuu{v}\in o_1$.
A lax basis $\pm\uuuu{v}$ corresponds to a pair $(\rho_1,\rho_2)$
and to an oriented edge.

For three families of semi-normal forms we will discuss
the positions of the oriented edges which generate the 
semi-normal forms. The three families of semi-normal forms
are defined in Definition \ref{t7.18} and discussed in
Lemma \ref{t7.19} and the Remarks \ref{t7.20}.

\begin{definition}\label{t7.18}
\begin{enumerate}
\item[(a)] 
(Compare Theorem \ref{t7.10} (b) and 
\cite[\S 13 (2) and (3)]{Zag81}.
A binary quadratic form $[a,h,b]_{quad}$ is a semi-normal form
of type A if
\begin{eqnarray}\label{7.17}
|a|\leq |b|\quad\textup{and}\quad h\in [-|a|,|a|].
\end{eqnarray}
\item[(b)] 
(Compare \cite[\S 13 (6)]{Zag81}) A binary quadratic form
$[a,h,b]_{quad}$ is a semi-normal form of type B if
\begin{eqnarray}\label{7.18}
a>0,\ b>0\quad\textup{and}\quad h>a+b.
\end{eqnarray}
\item[(c)] (Compare \cite[page 291]{We17} and \cite[page 113]{Ha21}) 
A binary quadratic form
$[a,h,b]_{quad}$ is a semi-normal form of type C if
\begin{eqnarray}\label{7.19}
a>0,\ b<0,\ a+b+|h|>0,\ a+b-|h|<0.
\end{eqnarray}
\end{enumerate}
\end{definition}

\begin{lemma}\label{t7.19}
Consider the data in Theorem \ref{t7.17}, so
$(H_\Z,o_1,q)$ is an abstract binary quadratic form
with discriminant $D>0$ and $4D\in\N$ not a square.
Consider the semi-normal forms of the types A, B and C
in the corresponding proper equivalence class of binary
quadratic forms. Consider the oriented edges in the binary
tree which generate these semi-normal forms. 

\begin{enumerate}
\item[(a)] 
An oriented edge $(\rho_1,\rho_2)$ generates a semi-normal
form of type A if and only if it is on the Conway river with
$|\omega_1(\rho_1)|\leq |\omega_1(\rho_2)|$ and satisfies
one of the following two conditions:
\begin{enumerate}
\item[(i)]
Its label $\omega_2((\rho_1,\rho_2))$ is 0.
\item[(ii)] 
Its label $h:=\omega_2((\rho_1,\rho_2))$ is not 0.
Consider only the Conway river and orient all its edges
with labels $\neq 0$ such that these labels are positive
(on the given edge this may be the same or the other 
orientation, depending on the sign of $h$).
Then (at least) one of the two vertices of the edge is a source
or a well with respect to the two neighboring edges on the
Conway river. The value $|h|$ of the given edge is smaller 
or equal than the value $|\www{h}|$ of the other edge 
at this vertex.
\end{enumerate}
See Figure \ref{fig:7.7}.
There $0<h\leq \www{h}$, and the horizontal edges are on the
Conway river.

\begin{figure}
\includegraphics[width=0.9\textwidth]{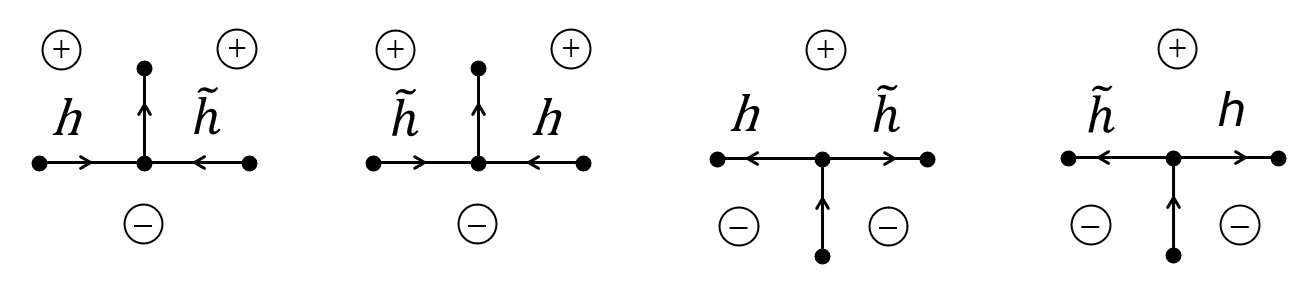}
\caption[Figure 7.7]{Lemma \ref{t7.19} (a)(ii): The oriented
edge with label $h$ generates a binary quadratic form
in semi-normal form of type A}
\label{fig:7.7}
\end{figure}

\item[(b)] An oriented edge $(\rho_1,\rho_2)$ generates a semi-normal
form of type B if and only if it is on the positive side
of the Conway river with one vertex on the Conway river
and pointing away from the Conway river. 
See Figure \ref{fig:7.8}.
There the two horizontal edges are on the Conway river.

\begin{figure}
\includegraphics[width=0.3\textwidth]{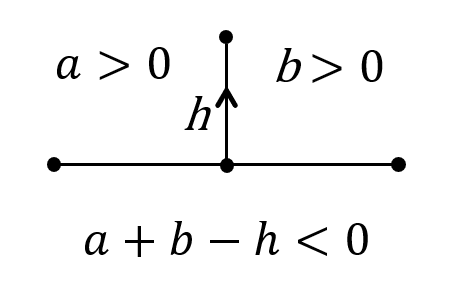}
\caption[Figure 7.8]{Lemma \ref{t7.19} (b): The oriented
edge with label $h$ generates a binary quadratic form
in semi-normal form of type B}
\label{fig:7.8}
\end{figure}

\item[(c)] An oriented edge $(\rho_1,\rho_2)$ generates a 
semi-normal form of type C if and only if it is on the Conway
river with $\omega_1(\rho_1)=a>0$, $\omega_1(\rho_2)=b<0$ and
the two neighboring edges which are not on the Conway river
are on different sides of the Conway river.
See Figure \ref{fig:7.9}.
There all horizontal edges are on the Conway river.

\begin{figure}
\includegraphics[width=0.8\textwidth]{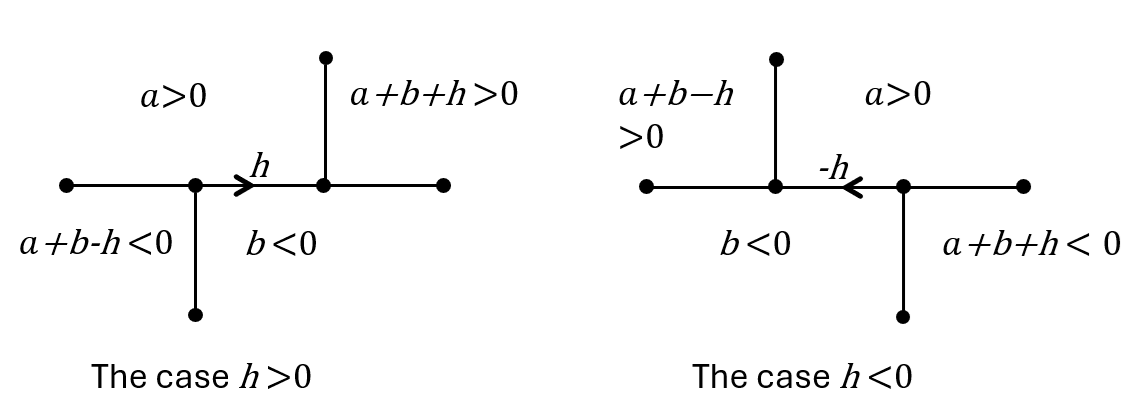}
\caption[Figure 7.9]{Lemma \ref{t7.19} (c): The oriented
edge with label $h$ generates a binary quadratic form
in semi-normal form of type C}
\label{fig:7.9}
\end{figure}

\end{enumerate}
\end{lemma}

{\bf Proof:}
The proofs of the parts (b) and (c) are rather obvious.
We leave them and the proof of part (a) to the reader.
\hfill$\Box$

\begin{remarks}\label{t7.20}
\begin{enumerate}
\item[(i)] 
Fix $r,s\in\Z$. The following set $\www{M}(r,s)$ contains
the set $M(r,s)$ in Theorem \ref{t7.10} (b). It might be slightly
bigger, but it is still finite.
\begin{eqnarray*}
\www{M}(r,s)&:=&\{\begin{pmatrix}a&b\\c&d\end{pmatrix}
\in M_{2\times 2}(\Z)\,|\, a+d=r,\ ad-bc=s, \\
&&\hspace*{2cm} 0<|c|\leq |b|,\ a-d\in [-|c|,|c|]\}.
\end{eqnarray*}
The matrices $\begin{pmatrix}a&b\\c&d\end{pmatrix}\in 
M_{2\times 2}(\Z)$ of type IV with fixed characteristic polynomial
$f(t)=t^2-rt+s$ which correspond by Lemma 
\ref{t7.2} (b) to binary quadratic forms $[c,d-a,-b]_{quad}$
in semi-normal form of type A are obviously precisely the
matrices in the set $\www{M}(r,s)$. If one wants to restrict to
$M(r,s)$, one has to refine part (a)(ii) of Lemma \ref{t7.19} suitably.
\item[(ii)] 
Also in \cite[\S 13 (2) and (3)]{Zag81} binary quadratic
forms in semi-normal form of type A are considered.
But there the binary quadratic forms in semi-normal form of type
B are preferred because of the following.
Consider two oriented edges $(\rho_1,\rho_2)$ and 
$(\rho_3,\rho_4)$ on the positive side of the Conway river 
which each have one vertex on the Conway river and which point
away from the river. So they generate binary quadratic forms
in semi-normal form of type B.
Suppose that walking along the Conway river with the positive
side to the right, $(\rho_3,\rho_4)$ is the next edge
on the positive side after the edge $(\rho_1,\rho_2)$
(in between there may be edges on the negative side).
Then $(\rho_3,\rho_4)=(\rho_1,\rho_2)\tau^{-k}\sigma$ for
some $k\in\N$. There is even a closed formula for $k$ in terms
of the value triple $(a,h,b)$ of the binary quadratic form
which is generated by $(\rho_1,\rho_2)$, see
\cite[\S 13 (5)]{Zag81}.
\item[(iii)] 
Binary quadratic forms of type IV in semi-normal form 
of type C are considered in \cite[page 291]{We17}
and \cite[page 113]{Ha21}. They are called {\it reduced forms}
in \cite[page 113]{Ha21}. The corresponding edges on the Conway
river are called {\it riverbends} in \cite[page 291]{We17}.
They exist because there are edges on both
sides of the Conway river.
\item[(iv)] 
Because of the periodicity in Theorem \ref{t7.17} (b) 
there are only finitely many semi-normal forms of each type.
\item[(v)] 
Let $[c,d-a,-b]_{quad}$ be a primitive (so 
$\gcd(c,d-a,b)=1$ by Theorem \ref{t7.9} (a)(iii)) binary
quadratic form in semi-normal form of type B.
By part (ii) and the periodicity in Theorem \ref{t7.17} (b)
it is mapped by a suitable smallest product 
$(T^{-k_1}S)(T^{-k_2}S)...(T^{-k_l}S)\in \SL_2(\Z)$ with
$l,k_1,...,k_l\in\N$ to itself. This product commutes with
the matrix $\begin{pmatrix}a&b\\c&d\end{pmatrix}$ and induces
a unit of norm 1 in the corresponding order 
$\Lambda=\Z[\lambda_1]$ in the real quadratic number field 
$A=\Q[\lambda_1]$. This unit and $-1$
generate the group of units of norm 1 in $\Lambda$.
\item[(vi)] 
Consider an arbitrary binary quadratic form
$[a,h,b]_{quad}$. If its $\GL_2(\Z)$ equivalence class splits
into two proper equivalence classes, then $[-a,h,-b]_{quad}$
is one representative of the other proper equivalence class.

{\bf Claim:} {\it The  value atlas from $[-a,h,-b]_{quad}$ 
is obtained from the value atlas from $[a,h,b]_{quad}$ 
by a reflection  $r:D\to D$ (along an arbitrary line in $D$) 
and the new maps
$\www{\omega}_1:=-\omega_1\circ r:r(R_1)\to\Z$ and
$\www{\omega}_2:=\omega_2\circ (r\times r):
(r\times r)(R_2)\to\Z$. }

We leave the proof of the claim to the reader.
\item[(vii)] Consider a primitive binary quadratic form $[a,h,b]_{quad}$
with discriminant $D>0$ and $4D\in\N$ not a square.
The reflection in part (vi) can be chosen to be along the
Conway river.

If the $\GL_2(\Z)$ equivalence class and the proper equivalence
class of $[a,h,b]_{quad}$ coincide, the procedure in part (vi)
together with a translation along the Conway river leads
to the old value atlas. 
Similarly to part (v), the composition of the reflection
along the Conway river and the translation along the 
Conway river give rise to a unit of norm $-1$ in the
corresponding order $\Lambda=\Z[\lambda_1]$ in the real quadratic 
number field $A=\Q[\lambda_1]$, which together with $-1$ generates the group of all units of $\Lambda$. Compare 
Theorem \ref{t7.9} (c).
\end{enumerate}
\end{remarks}

\begin{examples}\label{t7.21}
Consider the polynomial $f(t)=t^2-rt+s=t^2-7$, so with
$r=0$, $s=-7$, $D=\frac{r^2}{4}-s=7>0$. It will turn out
that there is only one $\GL_2(\Z)$-conjugacy class of matrices
in $M_{2\times 2}(\Z)$ with characteristic polynomial $f$,
and that it splits into two $\SL_2(\Z)$-conjugacy classes.

We are interested in matrices in semi-normal forms which
correspond by Lemma \ref{t7.2} (b) to binary quadratic forms
in semi-normal forms of types A, B and C (in Definition 
\ref{t7.18}). For that we will consider the value atlasses for
both proper equivalence classes in Conway's topograph.

But first we look at the semi-normal forms in Theorem \ref{t7.10}
(c)(ii). One sees easily
\begin{eqnarray*}
M(0,7)=\{\begin{pmatrix}0&7\\1&0\end{pmatrix},\ 
\begin{pmatrix}0&-7\\-1&0\end{pmatrix},\ 
\begin{pmatrix}1&3\\2&-1\end{pmatrix},\ 
\begin{pmatrix}1&-3\\-2&-1\end{pmatrix}\}.
\end{eqnarray*}
The possibly slightly bigger set $\www{M}(0,7)$ in Remark
\ref{t7.20} (i) is indeed slightly bigger, it is
\begin{eqnarray*}
\www{M}(0,7)=M(0,7)\,\cup\, 
\{\begin{pmatrix}-1&3\\2&1\end{pmatrix},\ 
\begin{pmatrix}-1&-3\\-2&1\end{pmatrix}\}.
\end{eqnarray*}
Thus there are six matrices in semi-normal form of type A.
But Theorem \ref{t7.10} does not say in how many 
$\GL_2(\Z)$ and $\SL_2(\Z)$-conjugacy classes they lie.
To see this, the topograph and the value atlasses are better
suited.

The companion matrix of the polynomial $f(t)=t^2-7$ is
$\begin{pmatrix}0&7\\1&0\end{pmatrix}$. It corresponds by
Lemma \ref{t7.2} (b) to the binary quadratic form
$[1,0,-7]_{quad}$. The corresponding edge in the topograph
lies on the Conway river. The value atlas near the Conway river
is given in Figure \ref{fig:7.10}.
The horizontal edges are on the Conway river. The periodicity
is expressed by the two small circles. Glueing infinitely many
copies of the graph in Figure \ref{fig:7.10}
at the small circles gives the Conway river and a neighborhood
of it. 

With Lemma \ref{t7.19} one reads off from Figure \ref{fig:7.10}
the binary quadratic forms in semi-normal forms of types A, B 
and C
which are in the proper equivalence class of $[1,0,-7]_{quad}$.
They are listed in Table \ref{tab7.2}. Below them the matrices
in the $\SL_2(\Z)$-conjugacy class of
$\begin{pmatrix}0&7\\1&0\end{pmatrix}$ which correspond to them
by Lemma \ref{t7.2} (b) are listed.

\begin{figure}
\includegraphics[width=0.8\textwidth]{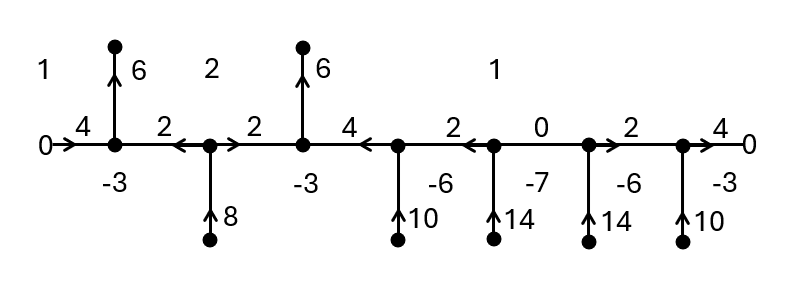}
\caption[Figure 7.10]{Part of the value atlas for 
$[1,0,-7]_{quad}$}
\label{fig:7.10}
\end{figure}

\begin{table}[H]
\begin{eqnarray*}
\begin{array}{lllll}
\textup{Type A} & [2,-2,-3]_{quad} & [2,2,-3]_{quad} & 
[1,0,-7]_{quad} & \\
 & \begin{pmatrix}-1&3\\2&1\end{pmatrix}  
 & \begin{pmatrix}1&3\\2&-1\end{pmatrix}  
 & \begin{pmatrix}0&7\\1&0\end{pmatrix} & \\
\textup{Type B} & [1,6,2]_{quad} & [2,6,1]_{quad} & & \\
 & \begin{pmatrix}3&-2\\1&-3\end{pmatrix} 
 & \begin{pmatrix}3&-1\\2&-3\end{pmatrix} & & \\
\textup{Type C} & [1,4,-3]_{quad} & [2,-2,-3]_{quad} 
 & [2,2,-3]_{quad} & [1,-4,-3]_{quad} \\
 & \begin{pmatrix}2&3\\1&-2\end{pmatrix}  
 & \begin{pmatrix}-1&3\\2&1\end{pmatrix}  
 & \begin{pmatrix}1&3\\2&-1\end{pmatrix}  
 & \begin{pmatrix}-2&3\\1&2\end{pmatrix}  
\end{array}
\end{eqnarray*}
\caption[Table 7.2]{Semi-normal forms of types A, B and C
in the proper equivalence class of $[1,0,-7]_{quad}$}
\label{tab7.2}
\end{table}

Figure \ref{fig:7.11} is obtained from Figure \ref{fig:7.10}
by the procedure in the claim in Remark \ref{t7.20} (vi).
It belongs to the other proper equivalence class within the
$\GL_2(\Z)$ equivalence class of $[1,0,-7]_{quad}$, namely to
the proper equivalence class of $[-1,0,7]_{quad}$.
Table \ref{tab7.3} is analogous to Table \ref{tab7.2}.
It lists the binary quadratic forms in semi-normal forms of
types A, B and C in the proper equivalence class of 
$[-1,0,7]_{quad}$ and the corresponding matrices in
the $\SL_2(\Z)$-conjugacy class of 
$\begin{pmatrix}0&-7\\-1&0\end{pmatrix}$.

\begin{figure}
\includegraphics[width=0.8\textwidth]{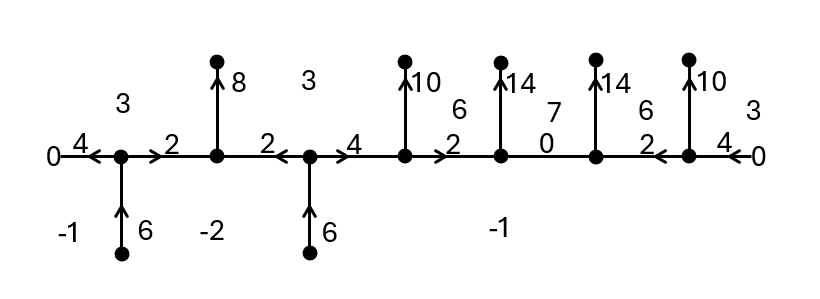}
\caption[Figure 7.11]{Part of the value atlas for
$[-1,0,7]_{quad}$}
\label{fig:7.11}
\end{figure}

\begin{table}[H]
\begin{eqnarray*}
\begin{array}{llllll}
\textup{Type A} & [-2,-2,3]_{q.} & [-2,2,3]_{q.} & 
[-1,0,7]_{q.} & & \\
 & \begin{pmatrix}-1&-3\\-2&1\end{pmatrix}  
 & \begin{pmatrix}1&-3\\-2&-1\end{pmatrix}  
 & \begin{pmatrix}0&-7\\-1&0\end{pmatrix} & \\
\textup{Type B} & [3,8,3]_{q.} & [3,10,6]_{q.} & 
{[6,14,7]_{q.}} & [7,14,6]_{q.} & [6,10,3]_{q.} \\
 & \begin{pmatrix}4&-3\\3&-4\end{pmatrix} 
 & \begin{pmatrix}5&-6\\3&-5\end{pmatrix} 
 & \begin{pmatrix}7&-7\\6&-7\end{pmatrix} 
 & \begin{pmatrix}7&-6\\7&-7\end{pmatrix} 
 & \begin{pmatrix}5&-3\\6&-5\end{pmatrix}  \\
\textup{Type C} & [3,4,-1]_{q.} & [3,2,-2]_{q.} 
 & [3,-2,-2]_{q.} & [3,4,-1]_{q.} & \\
 & \begin{pmatrix}2&1\\3&-2\end{pmatrix}  
 & \begin{pmatrix}1&2\\3&-1\end{pmatrix}  
 & \begin{pmatrix}-1&2\\3&1\end{pmatrix}  
 & \begin{pmatrix}2&1\\3&-2\end{pmatrix}  
\end{array}
\end{eqnarray*}
\caption[Table 7.3]{Semi-normal forms of types A, B and C
in the proper equivalence class of $[-1,0,7]_{quad}$}
\label{tab7.3}
\end{table}

One can make three observations.
\begin{enumerate}
\item[(1)] 
All matrices in $\www{M}(0,7)$ are in the two tables.
Therefore there is only one $\GL_2(\Z)$-conjugacy class
of matrices with characteristic polynomial $f(t)=t^2-7$.
\item[(2)]
There is a 1:1 correspondence between the semi-normal forms
of type A in both tables (and both figures). The same
holds for type C, but not for type B.
\item[(3)]
The semi-normal forms of type A and type C overlap.
\end{enumerate}
\end{examples}

\section{An irreducible rank 3 case}\label{s8}
\setcounter{equation}{0}
\setcounter{table}{0}
\setcounter{figure}{0}

\noindent 
Dade, Taussky and Zassenhaus report in \cite[1.4]{DTZ62}
that they calculated with computer for some algebraic number
fields $A$ and orders $\Lambda_4\subset A$ the finite
subsemigroup 
$\{[L]_w\,|\, L\in\LL(A),\OO(L)\supset \Lambda_4\}$
of the semigroup $W(\LL(A))=W(\EE(A))$, and also
the division map on this semigroup.
In \cite[1.5]{DTZ62} they document one example.
There $A=\Q[\alpha]$ with $\alpha$ a zero of the 
irreducible polynomial
$$t^3+4t^2+8t+16$$
and
$$\Lambda_4=\Z[\alpha]\subset A.$$
The semigroup $\{[L]_w\,|\, L\in\LL(A),\OO(L)\supset \Lambda_4\}$
turns out to have five elements, four of them idempotent.
\cite[1.5]{DTZ62} gives full lattices $I_1,...,I_5$ which
represent the five $w$-classes and presents the multiplication
table $([I_i]_w\cdot [I_j]_w)_{i,j\in\{1,...,5\}}$ and the
division table $([I_i]_w:[I_j]_w)_{i,j\in\{1,...,5\}}$.

We will study the finer subsemigroup 
$\{[L]_\varepsilon\,|\, L\in\LL(A),\OO(L)\supset \Lambda_4\}$
of $\EE(A)$. It will turn out to have six elements. While 
working this out in this section, we will see several results 
in the sections \ref{s4} and \ref{s5} in action.

First, we cite (and use without own proof) some facts which
\cite{LMFDB23} states on the algebraic number field
$A:=\Q[\alpha]$ and its maximal order $\Lambda_{max}$
of algebraic integers:
\begin{eqnarray*}
A&=&\Q[\gamma]\quad\textup{where } \gamma
\textup{ is a zero of }t^3-t^2+t+1,\\
\Lambda_{max}&=& \Z[\gamma]
=\langle 1,\gamma,\gamma^2\rangle_\Z,\\
\Lambda_{max}^{*}&=& \{\pm\gamma^l\,|\, l\in\Z\},\\
|G([\Lambda_{max}]_\varepsilon)|&=&
(\textup{class number of }A)=1.
\end{eqnarray*}
Define 
$$\beta:=\gamma-1.$$
Then
\begin{eqnarray*}
0&=&\beta^3+2\beta^2+2\beta+2,\\
\alpha&=&2\beta,\\
\Lambda_1&:=&\Lambda_{max}=\Z[\beta]
=\langle 1,\beta,\beta^2\rangle_\Z,\\
\Lambda_{max}^{*}&=& \{\pm (\beta+1)^l\,|\, l\in\Z\},\\
\Lambda_4&=&\Z[2\beta]=\langle 1,2\beta,4\beta^2\rangle_\Z.\
\end{eqnarray*}

\begin{lemma}\label{t8.1}
The only orders $\Lambda\supset\Lambda_4$ are
\begin{eqnarray*}
\Lambda_1&=& \langle 1,\beta,\beta^2\rangle_\Z,\\
\Lambda_2&=& \langle 1,2\beta,\beta^2\rangle_\Z,\\
\Lambda_3&=& \langle 1,2\beta,2\beta^2\rangle_\Z,\\
\Lambda_4&=& \langle 1,2\beta,4\beta^2\rangle_\Z.
\end{eqnarray*}
\end{lemma}

{\bf Proof:} Let $\Lambda\supset \Lambda_4$ be an order.

First case, $\Lambda\cap \{\beta+l\beta^2\,|\, l\in\Z\}
=\emptyset$: Then $\Lambda=\langle 1,2\beta\rangle_\Z
+\Lambda\cap \Z\beta^2$ and 
$\Lambda\cap\Z\beta^2\supset \Z4\beta^2$.
Only $\Lambda_2,\Lambda_3$ and $\Lambda_4$ are possible.

Second case, for some $l\in\Z$ we have $\beta+l\beta^2\in\Lambda$:
Because $4\beta^2\in\Lambda_4\subset\Lambda$ we can suppose
$l\in\{0,1,2,3\}$. We will show in each case $\Lambda=\Lambda_1$.

$l=0$: $\beta^2\in\Lambda$, $\Lambda=\Lambda_1$.

$l=1$: $(\beta+\beta^2)^2=\beta^2+2\beta^3+\beta^4
=-2\beta-\beta^2=-\beta-(\beta+\beta^2)$, so 
$\beta\in\Lambda$, so $\Lambda=\Lambda_1$.

$l=2$: $(\beta+2\beta^2)^2=\beta^2+4\beta^3+4\beta^4
=\beta^2+8$, so $\beta^2\in\Lambda$, so $\beta\in\Lambda$,
so $\Lambda=\Lambda_1$.

$l=3$: $\beta+3\beta^2=-(\beta+\beta^2)+2\beta+4\beta^2$,
so $\beta+\beta^2\in\Lambda$, so $\Lambda=\Lambda_1$.
\hfill$\Box$ 

\bigskip
For $i\in\{1,2,3,4\}$ denote by 
$$\tau_i:= |\{[L]_w\,|\, \OO(L)=\Lambda_i\}|$$
the number of $w$-equivalence classes of full lattices with
order $\Lambda_i$. $\Lambda_1$ and $\Lambda_4$ are cyclic
orders, so $\tau_1=\tau_4=1$ by Theorem \ref{t4.12}.
$\Lambda_2$ satisfies the hypothesis in Theorem \ref{t4.9},
namely $\Lambda_1=\Lambda_2+\beta\Lambda_2$, so
$\tau_2=\tau_1=1$. We will determine $\tau_3$ with Theorem 
\ref{t5.7}. In fact, we will apply Theorem \ref{t5.7}
to each $\Lambda_i$ and recover $\tau_1=\tau_2=\tau_4=1$.
The following table lists for $i\in\{1,2,3,4\}$ the
ingredients of Theorem \ref{t5.7}:
The generators $m_{i,1}$ and $m_{i,2}\in\Lambda_i$ with
$\Lambda_i=\langle 1,m_{i,1},m_{i,2}\rangle_\Z$,
the numbers $a_i,b_i,c_i,d_i\in\Z$, $\mu_i\in\N$,
$t_i\in \Z_{\geq 0}$ and $\tau_i\in\N$ with
\begin{eqnarray*}
m_{i,1}^2&=& b_im_{i,1}+a_im_{i,2}-a_ic_i,\\
m_{i,1}m_{i,2}&=& a_id_i,\\
m_{i,2}^2&=& d_im_{i,1}+c_im_{i,2}-b_id_i,\\
\mu_i&=& \gcd(a_i,b_i,c_i,d_i),\\
t_i&=& |\{p\in\P\,|\, p\textup{ divides }\mu_i\}|,\\
\tau_i&=& 2^{t_i}.
\end{eqnarray*}

\begin{table}[h]
\begin{eqnarray*}
\begin{array}{c|c|c|c|c|c|c|c|c|c}
i & m_{i,1} & m_{i,2} & a_i & b_i & c_i & d_i & \mu_i & 
t_i & \tau_i \\ \hline 
1 & \beta+2 & \beta^2+2 & 1 & 4 & 6 & 2 & 1 & 0 & 1 \\
2 & 2\beta+4 & \beta^2+2 & 4 & 8 & 6 & 1 & 1 & 0 & 1 \\
3 & 2\beta+4 & 2\beta^2+4 & 2 & 8 & 12 & 4 & 2 & 1 & 2 \\
4 & 2\beta+4 & 4\beta^2+8 & 1 & 8 & 24 & 16 & 1 & 0 & 1 
\end{array}
\end{eqnarray*}
\caption{The data in Theorem \ref{t5.7} for $\Lambda_i$,
$i\in\{1,2,3,4\}$}
\label{table3.1}
\end{table}

So $\tau_3=2$. One element in the $w$-class of non-invertible
full lattices $L$ with $\OO(L)=\Lambda_3$ can be found with
Dedekind's observation Lemma \ref{t4.6}, namely
\begin{eqnarray}\label{8.1}
L_3:=\langle 1,\beta,2\beta^2\rangle_\Z.
\end{eqnarray}
$L_3$ satisfies $1\in L_3\subsetneqq \Lambda_1$ and
$L_3^2=\Lambda_1$, so $L_3$ is not invertible by
Lemma \ref{t4.6}. One sees easily $\OO(L_3)=\Lambda_3$, so
\begin{eqnarray*}
\{[L]_w\,|\, L\in\LL(A),\OO(L)=\Lambda_3\}
=\{[\Lambda_3]_w,[L_3]_w\}.
\end{eqnarray*}
We have recovered the five $w$-equivalence classes in
\cite[1.5]{DTZ62} of full lattices $L$ with 
$\OO(L)\supset\Lambda_4$. 

But our representatives $\Lambda_1,\Lambda_2,\Lambda_3,L_3$
and $\Lambda_4$ differ from the representatives 
$I_5,I_3,I_4,I_2$ and $I_1$ in \cite[1.5]{DTZ62}.
The following list gives both sets of representatives
and their relations.
\begin{eqnarray*}
\Lambda_1=\langle 1,\beta,\beta^2\rangle_\Z &,& 
I_5=\langle 4,4\beta,4\beta^2\rangle =4\cdot \Lambda_1,\\
\Lambda_2=\langle 1,2\beta,\beta^2\rangle_\Z &,& 
I_3=\langle 4,2\beta,4\beta^2\rangle 
=(4\beta^2+2\beta+4)\cdot \Lambda_2,\\
\Lambda_3=\langle 1,2\beta,2\beta^2\rangle_\Z &,& 
I_4=\langle 4,2\beta+2,4\beta^2\rangle 
=2(\beta+1)\cdot \Lambda_3,\\
L_3      =\langle 1,\beta,2\beta^2\rangle_\Z &,& 
I_2=\langle 2,2\beta,4\beta^2\rangle =2\cdot L_3,\\
\Lambda_4=\langle 1,2\beta,4\beta^2\rangle_\Z &,& 
I_1=\Lambda_4.
\end{eqnarray*}

Contrary to \cite[1.5]{DTZ62} we also want to see the
$\varepsilon$-classes $[L]_\varepsilon$ with
$\OO(L)\supset\Lambda_4$. We know already
$|G([\Lambda_1])|=(\textup{class number of }A)=1$ from
\cite{LMFDB23}. We want to determine the numbers
$|G([\Lambda_i]_\varepsilon)|$ for $i\in\{2,3,4\}$.
Define for $i\in\{2,3,4\}$ the conductor
$$C_i:=\Lambda_i:\Lambda_1$$
of the pair $(\Lambda_1,\Lambda_i)$. 
Then Theorem \ref{t5.12} (e)
(respectively in the case of an algebraic number field as here
already \cite[Ch. I, (12.12) Theorem]{Ne99}) gives for
$i\in\{2,3,4\}$
\begin{eqnarray*}
|G([\Lambda_i]_\varepsilon)| &=& 
\frac{|G([\Lambda_1]_\varepsilon)|}
{[\Lambda_1^{*}:\Lambda_i^{*}]}\cdot 
\frac{|(\Lambda_1/C_i)^{*}|}{|(\Lambda_i/C_i)^{*}|}.
\end{eqnarray*}
By \cite{LMFDB23}
\begin{eqnarray*}
\Lambda_1^{*}&=& \{\pm (\beta+1)^l\,|\, l\in\Z\}.
\end{eqnarray*}
Observe
\begin{eqnarray*}
\beta+1 & \notin& \Lambda_2,\Lambda_3,\Lambda_4,\\
(\beta+1)^2&=& \beta^2+2\beta+1\in\Lambda_2,
\notin\Lambda_3,\notin\Lambda_4,\\
(\beta+1)^3&\notin& \Lambda_2,\Lambda_3,\Lambda_4,\\
(\beta+1)^4&=& -2\beta-3 \in\Lambda_2,\Lambda_3,\Lambda_4,
\end{eqnarray*}
so
\begin{eqnarray*}
\Lambda_2^{*}=\{\pm (\beta+1)^{2l}\,|\, l\in\Z\},&&
[\Lambda_1^{*}:\Lambda_2^{*}]=2,\\
\Lambda_3^{*}=\Lambda_4^{*}=
\{\pm(\beta+1)^{4l}\,|\, l\in\Z\},&&
[\Lambda_1^{*}:\Lambda_3^{*}]=[\Lambda_1^{*}:\Lambda_4^{*}]=4.
\end{eqnarray*}
In the following isomorphisms $\Lambda_1/C_i\cong ...$ 
or $\Lambda_i/C_i\cong ...$, on the right hand side
the multiplication of the quotient ring is not made precise.
But on the right hand side $[\beta]$, $[2\beta]$ and 
$[\beta^2]$ are nilpotent, so all summands except
$\Z_m\cdot {[1]}$ are nilpotent, 
and that is sufficient to determine the groups 
$(\Lambda_1/C_i)^{*}$ and $(\Lambda_i/C_i)^{*}$
of units. Here recall $\Z_m=\Z/m\Z$ for $m\in\N$ 
from the Notations \ref{t1.2} (i). Observe
\begin{eqnarray*}
C_2&=& \langle 2,2\beta,\beta^2\rangle_\Z,\\
\frac{\Lambda_2}{C_2}&=& 
\frac{\langle 1,2\beta,\beta^2\rangle_\Z}
{\langle 2,2\beta,\beta^2\rangle_\Z}
\cong \Z_2\cdot [1],\\
\Bigl(\frac{\Lambda_2}{C_2}\Bigr)^{*}&\cong& \{[1]\},\quad 
|\Bigl(\frac{\Lambda_2}{C_2}\Bigr)^{*}|=1,\\
\frac{\Lambda_1}{C_2}&=& 
\frac{\langle 1,\beta,\beta^2\rangle_\Z}
{\langle 2,2\beta,\beta^2\rangle_\Z} \cong
\Z_2\cdot [1]+\Z_2\cdot [\beta],\\
\Bigl(\frac{\Lambda_1}{C_2}\Bigr)^{*}
&\cong& [1]+\Z_2\cdot [\beta],\quad 
|\Bigl(\frac{\Lambda_1}{C_2}\Bigr)^{*}|=2.
\end{eqnarray*}

\begin{eqnarray*}
C_3&=& \langle 2,2\beta,2\beta^2\rangle_\Z,\\
\frac{\Lambda_3}{C_3}&=& 
\frac{\langle 1,2\beta,2\beta^2\rangle_\Z}
{\langle 2,2\beta,2\beta^2\rangle_\Z}
\cong \Z_2\cdot [1],\\
\Bigl(\frac{\Lambda_3}{C_3}\Bigr)^{*}&\cong& \{[1]\},\quad
|\Bigl(\frac{\Lambda_3}{C_3}\Bigr)^{*}|=1,\\
\frac{\Lambda_1}{C_3}&=& 
\frac{\langle 1,\beta,\beta^2\rangle_\Z}
{\langle 2,2\beta,2\beta^2\rangle_\Z} \cong
\Z_2\cdot [1]+\Z_2\cdot [\beta]+\Z_2\cdot [\beta^2]\\
\Bigl(\frac{\Lambda_1}{C_3}\Bigr)^{*}
&\cong& [1]+\Z_2\cdot [\beta]+\Z_2\cdot [\beta^2],\quad
|\Bigl(\frac{\Lambda_1}{C_3}\Bigr)^{*}|=4.
\end{eqnarray*}

\begin{eqnarray*}
C_4&=& \langle 4,4\beta,4\beta^2\rangle_\Z,\\
\frac{\Lambda_4}{C_4}&=& 
\frac{\langle 1,2\beta,4\beta^2\rangle_\Z}
{\langle 4,4\beta,4\beta^2\rangle_\Z}
\cong \Z_4\cdot [1]+\Z_2\cdot[2 \beta],\\
\Bigl(\frac{\Lambda_4}{C_4}\Bigr)^{*}&\cong& 
\{[1],[3]\}+\Z_2\cdot[2\beta],\quad
|\Bigl(\frac{\Lambda_4}{C_4}\Bigr)^{*}|=4,\\
\frac{\Lambda_1}{C_4}&=& 
\frac{\langle 1,\beta,\beta^2\rangle_\Z}
{\langle 4,4\beta,4\beta^2\rangle_\Z} \cong
\Z_4\cdot [1]+\Z_4\cdot [\beta]+\Z_4\cdot [\beta^2]\\
\Bigl(\frac{\Lambda_1}{C_4}\Bigr)^{*}
&\cong& \{[1],[3]\}+\Z_4\cdot [\beta]+\Z_4\cdot [\beta^2],\quad
|\Bigl(\frac{\Lambda_1}{C_4}\Bigr)^{*}|=32,
\end{eqnarray*}
so
\begin{eqnarray*}
|G([\Lambda_2]_\varepsilon)|&=& 
\frac{1}{2}\cdot \frac{2}{1}=1,\\
|G([\Lambda_3]_\varepsilon)|&=& 
\frac{1}{4}\cdot \frac{4}{1}=1,\\
|G([\Lambda_4]_\varepsilon)|&=& 
\frac{1}{4}\cdot \frac{32}{4}=2.
\end{eqnarray*}

We want to find an invertible full lattice $L_4$ with
$G([\Lambda_4]_\varepsilon)=\{[\Lambda_4]_\varepsilon,
[L_4]_\varepsilon\}$. By Theorem \ref{t5.12} (c)+(d) 
we can choose $L_4=C_4+a\Lambda_1$ where $a\in\Lambda_1$
with 
\begin{eqnarray*}
a+C_4&\in& (\Lambda_1/C_4)^{*} - 
\Bigl(\textup{the subgroup generated by }(\Lambda_4/C_4)^{*}\\
&&\hspace*{3cm}\textup{ and the image of }\Lambda_1^{*}\Bigr)\\
&=& (\Lambda_1/C_4)^{*} - 
(\Lambda_4/C_4)^{*}\cdot \{(\beta+1)^l\,|\, l\in\{0,1,2,3\}\}\\
&=& (32\textup{ elements}) - (16\textup{ elements}). 
\end{eqnarray*}
$a:=1+2\beta^2$ works and gives
\begin{eqnarray}
L_4&=& C_4+a\Lambda_4 \nonumber\\
&=& \langle 4,4\beta,4\beta^2\rangle_\Z
+ \langle 1+2\beta^2,(1+2\beta^2)2\beta,
(1+2\beta^2)4\beta^2\rangle_\Z \nonumber\\
&=& \langle 2,2\beta,2\beta^2+1\rangle_\Z.\label{8.2}
\end{eqnarray}
So the set $\{[L]_\varepsilon\,|\, L\in \LL(A),\OO(L)\supset
\Lambda_4\}$ consists of the following classes,
\begin{eqnarray*}
\begin{array}{c|c|c|c|l}
[\Lambda_1]_\varepsilon & [\Lambda_2]_\varepsilon & 
[\Lambda_3]_\varepsilon & [\Lambda_4]_\varepsilon, 
[L_4]_\varepsilon & \textup{: invertible} \\
 & & [L_3]_\varepsilon & & \textup{: not invertible}
\end{array}
\end{eqnarray*}

We are interested in the multiplication table for the six
$\varepsilon$-classes. Surprisingly, the set $\{\Lambda_1,\Lambda_2,
\Lambda_3,L_3,\Lambda_4,L_4\}$ of representatives itself
is multiplication invariant. Table \ref{table3.2} 
is their multiplication table. Part of it is easy to
see because the orders contain one another,
$\Lambda_1\supset\Lambda_2\supset\Lambda_2\supset\Lambda_4$,
so $\Lambda_i\Lambda_j=\Lambda_i$ if $i<j$. 
Only some of the products involving $L_3$ or $L_4$ need calculations.

\begin{table}
\begin{eqnarray*}
\begin{array}{c|cccccc}
\cdot & \Lambda_1 & \Lambda_2 & 
{} \Lambda_3 & L_3 & 
{} \Lambda_4 & L_4 \\ \hline
{} \Lambda_1 & \Lambda_1 & 
{} \Lambda_1 & \Lambda_1 & 
{} \Lambda_1 & \Lambda_1 & 
{} \Lambda_1 \\
{} \Lambda_2 & &  
{} \Lambda_2 & \Lambda_2 & 
{} \Lambda_1 & \Lambda_2 & 
{} \Lambda_2 \\
{} \Lambda_3 & & & \Lambda_3 & 
{} L_3 & \Lambda_3 & 
{} \Lambda_3 \\
{} L_3 & & & & 
{} \Lambda_1 & L_3 & 
{} L_3 \\
{} \Lambda_4 & & & & & \Lambda_4 & 
{} L_4 \\
{} L_4 & & & & & & \Lambda_4 
\end{array}
\end{eqnarray*}
\caption{The multiplication table for the six 
full lattices}
\label{table3.2}
\end{table}

Table \ref{table3.3} is the division table. 
\cite[Lemma 5.3 (b)]{HL26} and 
\cite[Lemma 8.1 (a)]{HL26}
give the major part of its entries. 

\begin{table}
\begin{eqnarray*}
\begin{array}{c|cccccc}
 : & [\Lambda_1]_\varepsilon & [\Lambda_2]_\varepsilon & 
{} [\Lambda_3]_\varepsilon & [L_3]_\varepsilon & 
{} [\Lambda_4]_\varepsilon & [L_4]_\varepsilon \\ \hline
{} [\Lambda_1]_\varepsilon & [\Lambda_1]_\varepsilon & 
{} [\Lambda_1]_\varepsilon & [\Lambda_1]_\varepsilon & 
{} [\Lambda_1]_\varepsilon & [\Lambda_1]_\varepsilon & 
{} [\Lambda_1]_\varepsilon \\
{} [\Lambda_2]_\varepsilon & [\Lambda_1]_\varepsilon &  
{} [\Lambda_2]_\varepsilon & [\Lambda_2]_\varepsilon & 
{} [\Lambda_1]_\varepsilon & [\Lambda_2]_\varepsilon & 
{} [\Lambda_2]_\varepsilon \\
{} [\Lambda_3]_\varepsilon &  [\Lambda_1]_\varepsilon & 
{} [\Lambda_1]_\varepsilon & [\Lambda_3]_\varepsilon & 
{} [\Lambda_1]_\varepsilon & [\Lambda_3]_\varepsilon & 
{} [\Lambda_3]_\varepsilon \\
{} [L_3]_\varepsilon &  [\Lambda_1]_\varepsilon & 
{} [\Lambda_2]_\varepsilon &  [L_3]_\varepsilon & 
{} [\Lambda_3]_\varepsilon & [L_3]_\varepsilon & 
{} [L_3]_\varepsilon \\
{} [\Lambda_4]_\varepsilon &  [\Lambda_1]_\varepsilon & 
{} [\Lambda_2]_\varepsilon &  [L_3]_\varepsilon & 
{} [\Lambda_3]_\varepsilon & [\Lambda_4]_\varepsilon & 
{} [L_4]_\varepsilon \\
{} [L_4]_\varepsilon & [\Lambda_1]_\varepsilon & 
{} [\Lambda_2]_\varepsilon &  [L_3]_\varepsilon & 
{} [\Lambda_3]_\varepsilon &  [L_4]_\varepsilon &  
{} [\Lambda_4]_\varepsilon 
\end{array}
\end{eqnarray*}
\caption{The division table for the six $\varepsilon$-classes}
\label{table3.3}
\end{table}
For example, it is remarkable that 
$\Lambda_4:\Lambda_3=2L_3\sim_\varepsilon L_3$
is not invertible. 

Finally, the matrices
$$M_L\in M_{3\times 3}(\Z)\quad\textup{ with }\quad
2\beta\cdot\BB_L=\BB_L\cdot M_L$$
for $L\in\{\Lambda_1,\Lambda_2,\Lambda_3,L_3,\Lambda_4,L_4\}$
and $\BB_L$ the $\Z$-basis of $L$ in Lemma \ref{t8.1}
(for $\Lambda_1,\Lambda_2,\Lambda_3,\Lambda_4$),
in \eqref{8.1} (for $L_3$) and in \eqref{8.2} (for $L_4$)
are as follows. They are representatives of the six
$\GL_3(\Z)$-conjugacy classes of integer matrices with
characteristic polynomial $t^3+4t^2+8t+16$.
\begin{eqnarray*}
M_{\Lambda_1}&=& 
\begin{pmatrix}0&0&-4\\2&0&-4\\0&2&-4\end{pmatrix},\quad
M_{\Lambda_2}=
\begin{pmatrix}0&0&-4\\1&0&-2\\0&4&-4\end{pmatrix},\\
M_{\Lambda_3}&=&
\begin{pmatrix}0&0&-8\\1&0&-4\\0&2&-4\end{pmatrix},\quad
M_{L_3}=
\begin{pmatrix}0&0&-8\\2&0&-8\\0&1&-4\end{pmatrix},\\
M_{\Lambda_4}&=&
\begin{pmatrix}0&0&-16\\1&0&-8\\0&1&-4\end{pmatrix},\quad
M_{L_4}=
\begin{pmatrix}0&-1&-2\\2&0&-3\\0&2&-4\end{pmatrix}.
\end{eqnarray*}

\section{The separable cases with 1-dimensional summands}
\label{s9}
\setcounter{equation}{0}
\setcounter{table}{0}
\setcounter{figure}{0}

\subsection{Observations on the case of arbitrary rank}
\label{s9.1}

\noindent 
The algebra $A=\Q e_1+...+\Q e_n$ with 
$e_ie_j=\delta_{ij}e_i$ and $n\in\Z_{\geq 2}$ is separable,
but not irreducible, so it is not an algebraic number field.
We will apply our tools and study the full lattices, orders
and semigroups. First we give a few general statements.
Then we treat the case $n=2$ systematically,
Then we give quite some details for the case $n=3$. Finally we 
apply them to one example, the conjugacy classes of $3\times 3$
matrices with eigenvalues $2,0,-2$.

The maximal order $\Lambda_{max}$ in $A$ is by the proof of 
Theorem \ref{t4.8}
$$\Lambda_{max}=\Z e_1+...+\Z e_n.$$
Each $\Lambda_{max}$-ideal $L\in\LL(A)$ splits as
\begin{eqnarray*}
&&L=\bigoplus_{i=1}^n L\cap \Q e_i
=\bigoplus_{i=1}^n \Z \beta_{ii} e_i\quad\textup{for some }
\beta_{11},...,\beta_{nn}\in\Q_{>0},\\
\textup{so}&&L=\Bigl(\sum_{i=1}^n\beta_{ii}e_i\Bigr)\cdot 
\Lambda_{max}\in [\Lambda_{max}]_\varepsilon,\\
\textup{so}&&G([\Lambda_{max}]_\varepsilon)= 
\{[\Lambda_{max}]_\varepsilon\},
\end{eqnarray*}
so the class number of $A$ is 
$|G([\Lambda_{max}]_\varepsilon)|=1$.

Each full lattice $L\in\LL(A)$ has a unique $\Z$-basis of the shape
\begin{eqnarray}\label{9.1}
\uuuu{e}\cdot 
\begin{pmatrix}\beta_{11} & \beta_{12} & \cdots & \beta_{1n}\\
0 & \beta_{22} & \cdots & \beta_{2n} \\
\vdots & \ddots & \ddots & \vdots \\
0 & \cdots & 0 & \beta_{nn}\end{pmatrix}
\quad\textup{with }\beta_{ii}\in\Q_{>0}\textup{ for }1\leq i\leq n\\
\textup{and }\beta_{ij}\in (-\frac{1}{2}\beta_{ii},
\frac{1}{2}\beta_{ii}]\cap \Q\textup{ for }i<j.\nonumber
\end{eqnarray}
Multiplication of $L$ with a unit 
$u=\sum_{i=1}^n\varepsilon_i\beta_{ii}^{-1}e_i\in A^{*}$ 
with suitable signs $\varepsilon_1,...,\varepsilon_n\in\{\pm 1\}$
leads to a full lattice $uL$ which has a $\Z$-basis of the shape
\begin{eqnarray}\label{9.2}
\uuuu{e}\cdot 
\begin{pmatrix} 1 & \delta_{12} & \cdots & \delta_{1n}\\
0 & 1 & \cdots & \vdots \\
\vdots & \ddots & \ddots & \vdots \\
0 & \cdots & 0 & 1 \end{pmatrix}
\quad\textup{with }\delta_{ij}\in(-\frac{1}{2},\frac{1}{2}]
\cap\Q\textup{ for }i<j<n\\
\textup{and }\delta_{in}\in [0,\frac{1}{2}]\cap\Q 
\textup{ for }i<n.\nonumber
\end{eqnarray}
If $\delta_{1n},...,\delta_{n-1,n}\in(0,\frac{1}{2})$ the 
$\Z$-basis
is unique. But if some $\delta_{in}\in\{0;\frac{1}{2}\}$
and $n\geq 3$, there are several possibilities, so in these
boundary cases \eqref{9.2} is only a semi-normal form.

An order has a unique $\Z$-basis of the shape
\begin{eqnarray}\label{9.3}
\uuuu{e}\cdot 
\begin{pmatrix}\alpha_{11} & \alpha_{12} & \cdots & 
\alpha_{1,n-1} & 1 \\
0 & \alpha_{22} & \cdots & \alpha_{2,n-1} & 1 \\
\vdots & \ddots & \ddots & \vdots & \vdots \\
\vdots & \cdots & 0 & \alpha_{n-1,n-1} & \vdots \\
0 & \cdots & \cdots & 0 & 1 \end{pmatrix}
\quad\textup{with }\alpha_{ii}\in\N\textup{ for }1\leq i<n\\
\textup{and }\alpha_{ij}\in (-\frac{1}{2}\alpha_{ii},
\frac{1}{2}\alpha_{ii}]\cap \Z\textup{ for }i<j<n.\nonumber
\end{eqnarray}
This is almost the $\Z$-basis in \eqref{9.1}. 
Only the last element
has been replaced by $1_A=e_1+...+e_n$. For $n\geq 3$ the condition
that $\Lambda$ is multiplication invariant gives additional
constraints on the entries $\alpha_{ij}$ in \eqref{9.3}.

The group of units $\Lambda^{*}$ of an order $\Lambda$ 
is quite different from the case of algebraic number fields. 
It is finite, namely
\begin{eqnarray*}
\Lambda_{max}^{*}&=& \{\sum_{i=1}^n\varepsilon_ie_i\,|\, 
\varepsilon_{1},...,\varepsilon_n\in\{\pm 1\}\},
\quad\textup{so }|\Lambda_{max}^{*}|=2^n,
\end{eqnarray*}
and for any order $\Lambda$
\begin{eqnarray*}
\Lambda_{max}^{*}\supset\Lambda^{*}\supset \{\pm 1_A\},
\quad\textup{so }|\Lambda^{*}|\in\{2,4,...,2^n\}.
\end{eqnarray*}

For an order $\Lambda$ the conductor $C:=\Lambda:\Lambda_{max}$
is by Theorem \ref{t5.12} (a) 
the maximal $\Lambda_{max}$-ideal in $\Lambda$, so it is 
\begin{eqnarray*}
C=\bigoplus_{i=1}^n \Lambda\cap\Z e_i
=\bigoplus_{i=1}^n\Z\gamma_i e_i
\end{eqnarray*}
for unique $\gamma_1,...,\gamma_n\in\N$.
The finite algebra $\Lambda_{max}/C$ is
\begin{eqnarray*}
\Lambda_{max}/C \cong \bigoplus_{i=1}^n\Z_{\gamma_i}\quad 
\textup{with }[e_i]\cdot [e_j]=0\textup{ for }i\neq j
\end{eqnarray*}
with group of units
\begin{eqnarray*}
\Bigl(\Lambda_{max}/C\Bigr)^{*}
&=& \prod_{i=1}^n \Z_{\gamma_i}^{*},\quad
|\Bigl(\Lambda_{max}/C\bigr)^{*}|
=\prod_{i=1}^n\varphi(\gamma_i).
\end{eqnarray*}
The subalgebra $\Lambda/C$ and its group of units
$\bigl(\Lambda/C\bigr)^{*}\subset (\Lambda_{max}/C)^{*}$
are usually not difficult to determine. Therefore all
ingredients of the formula from Theorem \ref{t5.12} (e)
\begin{eqnarray}\nonumber
|G([\Lambda]_\varepsilon)|&=& |G([\Lambda_{max}]_\varepsilon)|
\cdot \frac{|(\Lambda_{max}/C)^{*}|}{(\Lambda/C)^{*}|}
\cdot \frac{1}{[\Lambda_{max}^{*}:\Lambda^{*}]}\\
&\stackrel{\textup{here}}{=}&
1\cdot \frac{\prod_{i=1}^n\varphi(\gamma_i)}{|(\Lambda/C)^{*}|}
\cdot \frac{|\Lambda^{*}|}{2^n}\label{9.4}
\end{eqnarray}
can be determined rather easily.

Unfortunately, except for the case $n=3$ 
by Theorem \ref{t5.7}  from \cite{Fa65-2}, 
there is no formula available for the number
of $w$-classes in the set $\{[L]_\varepsilon\,|\, L\in\LL(A), 
\OO(L)=\Lambda\}$ for an arbitrary order $\Lambda$.

\subsection{The rank 2 case}\label{s9.2}

\noindent
Now we come to the case $n=2$.

\begin{theorem}\label{t9.1}
Let $A=\Q e_1+\Q e_2$ with $e_ie_j=\delta_{ij}e_i$ for
$i,j\in\{1,2\}$. 

\begin{enumerate}
\item[(a)] 
Each order in $A$ has a unique $\Z$-basis 
$\uuuu{e}\begin{pmatrix}\alpha & 1 \\ 0 & 1\end{pmatrix}$
for some $\alpha\in\N$. Vice versa for each $\alpha\in\N$
$\uuuu{e}\begin{pmatrix}\alpha&1\\0&1\end{pmatrix}$
is a $\Z$-basis of an order. 
So there is a 1:1 correspondence between the set of orders
and the set $\N$.
\item[(b)] 
Each $\varepsilon$-class of full lattices contains a unique
full lattice $L_{\delta}$ with a $\Z$-basis
\begin{eqnarray}\label{9.5}
\uuuu{e}\begin{pmatrix}1&\delta\\0&1\end{pmatrix}
\quad\textup{with }\delta\in[0,\frac{1}{2}]\cap\Q.
\end{eqnarray}
The $\Z$-basis is unique. So there is a 1:1 correspondence
between the set $\EE(A)$ and the set $[0,\frac{1}{2}]\cap\Q$.
Here an order with invariant $\alpha\in\N$ in part (a)
corresponds to $\frac{1}{\alpha}$ if $\alpha\geq 2$ and
to $0$ if $\alpha=1$
(the case with $\alpha=1$ is the maximal order).
\item[(c)] 
The semigroup structure in $\EE(A)$ corresponds to the 
following semigroup structure on the set
$[0,\frac{1}{2}]\cap\Q$. Write 
$\delta_i=\frac{\beta_i}{\alpha_i}\in [0,\frac{1}{2}]\cap\Q$
with $\beta_i\in\Z$, $\alpha_i\in\N$, $\gcd(\beta_i,\alpha_i)=1$
for $i\in\{1,2,3\}$. Then
$[L_{\delta_1}]_\varepsilon\cdot [L_{\delta_2}]_\varepsilon
={}[L_{\delta_3}]_\varepsilon$ with
\begin{eqnarray*}
\alpha_3&=&\gcd(\alpha_1,\alpha_2),\\
\beta_3&\equiv& \varepsilon \beta_1\beta_2\mmod \alpha_3
\textup{ for some }\varepsilon\in\{\pm 1\}.
\end{eqnarray*}
Here the sign $\varepsilon$ is chosen so that 
$\varepsilon\beta_1\beta_2\equiv (\textup{an element of }
[0,\frac{1}{2}\alpha_3]\cap\Z)\mmod\alpha_3$.
\item[(d)] 
Each full lattice is invertible. For an order $\Lambda$
with $\Z$-basis 
$\uuuu{e}\begin{pmatrix}\alpha&1\\0&1\end{pmatrix}$
the group of $\varepsilon$-classes of (automatically invertible)
exact $\Lambda$-ideals is 
\begin{eqnarray*}
G([\Lambda]_\varepsilon)&=& \{[L_\delta]_\varepsilon\,|\, 
\delta=\frac{\beta}{\alpha}\textup{ with }
\beta\in[0,\frac{1}{2}\alpha]\cap\Z,\gcd(\beta,\alpha)=1\}\\
&\cong& \left\{\begin{array}{ll}
\Z_{\alpha}^{*}/\{\pm 1\}&\textup{ if }\alpha\geq 3,\\
\{1\}&\textup{ if }\alpha\in\{1,2\}.\end{array}\right.
\end{eqnarray*}
Therefore 
\begin{eqnarray*}
|G([\Lambda]_\varepsilon)|&=&
\left\{\begin{array}{ll}
\frac{\varphi(\alpha)}{2}&\textup{ if }\alpha\geq 3,\\
1 &\textup{ if }\alpha\in\{1,2\}.\end{array}\right.
\end{eqnarray*}
\end{enumerate}
\end{theorem}

{\bf Proof:} 
(a) A special case of \eqref{9.3}.

(b) A special case of \eqref{9.2}.

(c) The full lattice $L_{\delta_1}\cdot L_{\delta_2}$
is generated over $\Z$ by the following elements,
\begin{eqnarray*}
\uuuu{e}\begin{pmatrix} 1 & \frac{\beta_1}{\alpha_1} & 
\frac{\beta_2}{\alpha_2} & 
\frac{\beta_1\beta_2}{\alpha_1\alpha_2} \\
0 & 0 & 0 & 1\end{pmatrix},\\
\textup{so also by }\quad
\uuuu{e}\begin{pmatrix} \frac{1}{\alpha_1} & 
\frac{1}{\alpha_2} & 
\frac{\beta_1\beta_2}{\alpha_1\alpha_2} \\
0 & 0 & 1\end{pmatrix},\\
\textup{so also by }\quad
\uuuu{e}\begin{pmatrix} \frac{1}{\lcm(\alpha_1,\alpha_2)} & 
\frac{\beta_1\beta_2}{\alpha_1\alpha_2} \\
0 & 1\end{pmatrix}.
\end{eqnarray*}
The full lattices 
$(\lcm(\alpha_1,\alpha_2)e_1+e_2)L_{\delta_1}L_{\delta_2}$
and $(\lcm(\alpha_1,\alpha_2)e_1-e_2)L_{\delta_1}L_{\delta_2}$ 
are generated by
\begin{eqnarray*}
\uuuu{e}\begin{pmatrix}1&
\frac{\beta_1\beta_2}{\gcd(\alpha_1,\alpha_2)}\\
0&1\end{pmatrix}\textup{ respectively by }
\uuuu{e}\begin{pmatrix}1&
-\frac{\beta_1\beta_2}{\gcd(\alpha_1,\alpha_2)}\\
0&1\end{pmatrix}
\end{eqnarray*}
Both are $\varepsilon$-equivalent to $L_{\delta_1}L_{\delta_2}$.

(d) Consider a full lattice $L_\delta$ as in (b) and an order
$\Lambda$ as in (a). Then $\Lambda\cdot L_\delta$ 
is generated over $\Z$ by the following elements,
\begin{eqnarray*}
\uuuu{e}\begin{pmatrix}\alpha & \alpha\delta & 1 & \delta\\
0 & 0 & 0 & 1 \end{pmatrix},
\end{eqnarray*}
so it is $L_\delta$ if and only if $\alpha\delta\in\Z$. Also
\begin{eqnarray*}
\OO(L_\delta)=\Lambda&\iff& 
\left\{\begin{array}{ll}
\alpha=1 &\textup{ if }\delta=0,\\
\alpha=\www{\alpha}&\textup{ if }
\delta=\frac{\beta}{\www{\alpha}}\in (0,\frac{1}{2}]\cap\Q
\textup{ with }\gcd(\beta,\www{\alpha})=1
\end{array}\right.
\end{eqnarray*}
is clear now. The semigroup structure in part (c) restricts to
the group $G([\Lambda]_\varepsilon)$, and this is isomorphic
to $\Z_\alpha^{*}/\{\pm 1\}$ if $\alpha\geq 3$
and to $\{1\}$ if $\alpha\in\{1;2\}$. \hfill$\Box$

\begin{remarks}\label{t9.2}
As we know already $|G([\Lambda]_\varepsilon)|=
\frac{\varphi(\alpha)}{2}$ if $\alpha\geq 3$ and
$=1$ if $\alpha\in\{1,2\}$, it is not necessary, but it is still
instructive to determine the ingredients of formula \eqref{9.4}
in the case $n=2$. 

Consider an order $\Lambda\subsetneqq\Lambda_{max}$ with 
$\Z$-basis $\uuuu{e}\begin{pmatrix}\alpha&1\\0&1\end{pmatrix}$
with $\alpha\in\Z_{\geq 3}$. Then
\begin{eqnarray*}
\Lambda_{max}^{*}&=&\{\pm e_1\pm e_2\},\quad
|\Lambda_{max}^{*}|=4,\\
\Lambda^{*}&=&\{\pm (e_1+e_2)\},\quad
|\Lambda^{*}|=2,\\
C&=&\Lambda:\Lambda_{max}=\Z\alpha e_1+\Z \alpha e_2,\\
\Lambda_{max}/C&\cong & \Z_{\alpha}\oplus \Z_{\alpha},\\
(\Lambda_{max}/C)^{*}&\cong & \Z_{\alpha}^{*}\times
\Z_{\alpha}^{*},\quad 
|(\Lambda_{max}/C)^{*}|=\varphi(\alpha)^2,\\
\Lambda/C&=& \Z_{\alpha}[e_1+e_2]\hookrightarrow
\Lambda_{max}/C,\\
(\Lambda/C)^{*}&\cong& \Z_{\alpha}^{*}[e_1+e_2]
\hookrightarrow (\Lambda_{max}/C)^{*},\quad
|(\Lambda/C)^{*}|=\varphi(\alpha),\\
|G([\Lambda]_\varepsilon)|&=& 1\cdot 
\frac{\varphi(\alpha)^2}{\varphi(\alpha)}\cdot \frac{2}{4}
=\frac{\varphi(\alpha)}{2}.
\end{eqnarray*}

Consider the order $\Lambda_2$ with 
$\Z$-basis $\uuuu{e}\begin{pmatrix}2&1\\0&1\end{pmatrix}$. Then
\begin{eqnarray*}
\Lambda_2^{*}&=&\{\pm e_1\pm e_2\},\quad
|\Lambda^{*}|=4,\\
C&=&\Lambda_2:\Lambda_{max}=\Z 2 e_1+\Z 2 e_2,\\
\Lambda_{max}/C&\cong & \Z_2\oplus \Z_2,\\
(\Lambda_{max}/C)^{*}&\cong & \Z_2^{*}\times
\Z_2^{*},\quad 
|(\Lambda_{max}/C)^{*}|=1,\\
|(\Lambda/C)^{*}|&=&1,\\
|G([\Lambda]_\varepsilon)|&=& 1\cdot 
\frac{1}{1}\cdot \frac{4}{4}=1.
\end{eqnarray*}
\end{remarks}

\begin{example}\label{t9.3}
Choose $\lambda\in\N$. By Theorem \ref{t6.3}   
the conjugacy classes of
integer $2\times 2$ matrices with eigenvalues $\lambda$ and $0$
correspond to the $\varepsilon$-classes of full lattices
$L$ with $\OO(L)\supset \Lambda$ where $\Lambda$ is the
order $\Lambda=\Z[\lambda e_1]$, so the order with $\Z$-basis
$\uuuu{e}\begin{pmatrix}\lambda&1\\0&1\end{pmatrix}$.
By Theorem \ref{t9.1} (d)
\begin{eqnarray*}
\{[L]_\varepsilon\,|\, \OO(L)\supset\Lambda\}
=\bigcup_{\alpha\in\N:\, \alpha|\lambda}
\{[L_\delta]_\varepsilon\,|\, \delta=\frac{\beta}{\alpha}
\textup{ with }\beta\in[0,\frac{1}{2}\alpha]
\cap\Z_{\alpha}^{*}\}.
\end{eqnarray*}
The $\Z$-basis 
$\uuuu{e}\begin{pmatrix}1&\delta\\0&1\end{pmatrix}$
of a full lattice $L_\delta$ 
with $\delta=\frac{\beta}{\alpha}\in[0,\frac{1}{2}]$,
$\alpha\in\N$, $\alpha|\lambda$, $\beta\in\Z_{\geq 0}$ 
and $\gcd(\beta,\alpha)=1$ gives
rise to the matrix $M_\delta$ with
\begin{eqnarray*}
\lambda e_1\cdot 
\uuuu{e}\begin{pmatrix}1&\delta\\0&1\end{pmatrix}
=\uuuu{e}\begin{pmatrix}\lambda&\lambda\delta\\0&0\end{pmatrix}
=\uuuu{e}\begin{pmatrix}1&\delta\\0&1\end{pmatrix}\cdot 
\begin{pmatrix}\lambda&\lambda\delta\\0&0\end{pmatrix}
=\uuuu{e}\begin{pmatrix}1&\delta\\0&1\end{pmatrix}\cdot M_\delta
\end{eqnarray*}
so
$$M_\delta=\begin{pmatrix}\lambda&\lambda\delta\\
0&0\end{pmatrix}
=\begin{pmatrix}\lambda&\lambda\frac{\beta}{\alpha}\\
0&0\end{pmatrix}.$$
We obtain the following theorem. The normal form in it 
is also observed in \cite[16.]{AO82}. 
\end{example}

\begin{theorem}\label{t9.4}
Each conjugacy class of $2\times 2$ matrices with eigenvalues
$\lambda\in\N$ and 0 has a unique representative 
\begin{eqnarray*}
\begin{pmatrix}\lambda&\mu\\0&0\end{pmatrix}\quad
\textup{with }\mu\in[0,\frac{1}{2}\lambda]\cap\Z.
\end{eqnarray*}
This matrix comes from an $\varepsilon$-class of full lattices
with order with $\Z$-basis
$\uuuu{e}\begin{pmatrix}\alpha&1\\0&1\end{pmatrix}$
where $\alpha=\lambda/\gcd(\lambda,\mu)$.
\end{theorem}

\subsection{The rank 3 case}\label{s9.3}

\noindent
Now we come to the case $n=3$. Now 
$$A=\Q e_1+\Q e_2+\Q e_3.$$
First we discuss the orders in $A$. 
Here $\uuuu{e}=(e_1,e_2,e_3)$.

\begin{lemma}\label{t9.5}
\begin{enumerate}
\item[(a)] 
For each triple 
\begin{eqnarray}\label{9.6}
(\alpha_1,\alpha_2,\alpha_3)\in\N^2\times\Z&&
\textup{with }
\alpha_3\in(-\frac{1}{2}\alpha_1,\frac{1}{2}\alpha_1]\\
&&\textup{and }
\alpha_1|(\alpha_3(\alpha_2-\alpha_3))\nonumber
\end{eqnarray}
the full lattice with $\Z$-basis
\begin{eqnarray}\label{9.7}
\uuuu{e}\begin{pmatrix}\alpha_1&\alpha_3&1\\
0&\alpha_2&1\\ 0&0&1\end{pmatrix}
\end{eqnarray}
is an order. This order is called 
$\Lambda_{\alpha_1\alpha_2\alpha_3}$. There are no other orders.
For each order, the $\Z$-basis in \eqref{9.7} with 
\eqref{9.6} is unique.
The maximal order is $\Lambda_{max}=\Lambda_{110}$. 
\item[(b)] 
An order $\Lambda_{\beta_1\beta_2\beta_3}$ contains
an order $\Lambda_{\alpha_1\alpha_2\alpha_3}$ if and only if 
\begin{eqnarray}\label{9.8}
\beta_1|\alpha_1,\quad \beta_2|\alpha_2,\quad 
\frac{\alpha_2}{\beta_2}\beta_3\equiv \alpha_3\mmod \beta_1.
\end{eqnarray}
Especially, the set of orders which contain a given order 
$\Lambda_{\alpha_1\alpha_2\alpha_3}$ is finite and can be 
determined easily.
\item[(c)] 
The size $|\Lambda_{\alpha_1\alpha_2\alpha_3}^{*}|$
of the set $\Lambda_{\alpha_1\alpha_2\alpha_3}^{*}
\subset\Lambda_{max}^{*}=\{\pm e_1\pm e_2\pm e_3\}$
of units in $\Lambda_{\alpha_1\alpha_2\alpha_3}$ is 
$2$ or $4$ or $8$. In most cases it is 2, then
$\Lambda_{\alpha_1\alpha_2\alpha_3}^{*}=\{\pm 1_A\}$.
\begin{eqnarray*}
&&|\Lambda_{\alpha_1\alpha_2\alpha_3}^{*}|=8
\iff (\alpha_1,\alpha_2,\alpha_3)\in \\
&&\hspace*{1cm}\{(1,1,0),(1,2,0),(2,1,0),
(2,1,1),(2,2,0)\},\\
&&|\Lambda_{\alpha_1\alpha_2\alpha_3}^{*}|=4\iff 
(\alpha_1,\alpha_2,\alpha_3)\in \\
&&\hspace*{1cm}\{(1,\alpha_2,\alpha_3)\,|\, \alpha_2\geq 3\}
\cup\{(2,\alpha_2,\alpha_3)\,|\, \alpha_2\geq 3\}
\cup\{(2,2,1)\}\\
&& \hspace*{1cm}\cup\,\{(\alpha_1,1,0)\,|\, \alpha_1\geq 3\}
\cup\{(\alpha_1,2,0)\,|\, \alpha_1\geq 3\},\\
&& \hspace*{1cm}\cup\,\{(\alpha_1,1,1)\,|\, \alpha_1\geq 3\}
\cup\{(\alpha_1,2,2)\,|\, \alpha_1\geq 4\}
\cup\{(3,2,-1)\}\\
&&|\Lambda_{\alpha_1\alpha_2\alpha_3}^{*}|=2\textup{ else.}
\end{eqnarray*}
\item[(d)] 
The conductor $C:=\Lambda_{\alpha_1\alpha_2\alpha_3}:
\Lambda_{max}$ of an order $\Lambda_{\alpha_1\alpha_2\alpha_3}$
is the maximal $\Lambda_{max}$-ideal in
$\Lambda_{\alpha_1\alpha_2\alpha_3}$. It is 
\begin{eqnarray}\label{9.9}
C&=&\sum_{i=1}^3\Lambda_{\alpha_1\alpha_2\alpha_3}\cap \Z e_i
=\sum_{i=1}^3\Z\beta_ie_i\\
\textup{with }\beta_1&=&\alpha_1,\ 
\beta_2=\frac{\alpha_1\alpha_2}{\gcd(\alpha_1,\alpha_3)},\ 
\beta_3=\frac{\alpha_1\alpha_2}
{\gcd(\alpha_1,\alpha_2-\alpha_3)}.\label{9.10}
\end{eqnarray}
Then
\begin{eqnarray}\nonumber
\lcm(\beta_1,\beta_2,\beta_3)&=&\lcm(\beta_1,\beta_2)
=\lcm(\beta_1,\beta_3)=\lcm(\beta_2,\beta_3)\\
&=&\frac{\alpha_1\alpha_2}{\gcd(\alpha_1,\alpha_2,\alpha_3)}
\label{9.11}
\end{eqnarray}
and 
\begin{eqnarray}\label{9.12} 
\Bigl(\Lambda_{max}/C\Bigr)^{*}&\cong& 
\Z_{\beta_1}^{*}\times \Z_{\beta_2}^{*}
\times \Z_{\beta_3}^{*},\\
|\Bigl(\Lambda_{max}/C\Bigr)^{*}|
&=&\prod_{i=1}^3\varphi(\beta_i),\label{9.13}\\
\Bigl(\Lambda_{max}/C\Bigr)^{*}
&\supset&
\Bigl(\Lambda_{\alpha_1\alpha_2\alpha_3}/C\Bigr)^{*}
\supset
\Z_{\lcm(\beta_1,\beta_2,\beta_3)}^{*}\cdot [1_A].\label{9.14}
\end{eqnarray}

(e) \cite[Corollary 4.1]{Fa65-2} which is cited as Theorem 
\ref{t5.7} applies to $A$ and its orders. Let 
$\tau_{\alpha_1\alpha_2\alpha_3}
=\tau(\Lambda_{\alpha_1\alpha_2\alpha_3})$ be the number of 
$w$-classes in the set 
$\{[L]_\varepsilon\,|\, L\in\LL(A),\OO(L)=
\Lambda_{\alpha_1\alpha_2\alpha_3}\}$. Then
\begin{eqnarray}\label{9.15}
\tau_{\alpha_1\alpha_2\alpha_3}&=&2^t,\quad 
t=|\{p\in\P\,|\, p|\mu\}|,\quad 
\mu=\gcd(a,b,c,d),\\
(a,b,c,d)&=&(\frac{\alpha_3(\alpha_3-\alpha_2)}{\alpha_1},
\alpha_2-2\alpha_3,\alpha_1,0).\nonumber
\end{eqnarray}
\end{enumerate}
\end{lemma}

{\bf Proof:}
(a) Any full lattice $L\subset \Lambda_{max}$ with $1\in L$
has a unique $\Z$-basis of the shape in \eqref{9.7} with
$\alpha_1,\alpha_2\in\N$, $\alpha_3\in (-\frac{1}{2}\alpha_1,
\frac{1}{2}\alpha_1]\cap\Z$. It is an order if it is
multiplication invariant. The only constraint is
\begin{eqnarray*}
&&(\alpha_3e_1+\alpha_2e_2)(\alpha_3e_1+\alpha_2e_2)\in L\\
&=& \alpha_3^2e_1+\alpha_2^2e_2
=\alpha_3(\alpha_3-\alpha_2)e_1
+\alpha_2(\alpha_3e_1+\alpha_2e_2),
\end{eqnarray*}
so $\alpha_1\,|\, (\alpha_3(\alpha_2-\alpha_3))$.

(b) \eqref{9.8} rewrites the conditions
$\beta_1e_1\in\Lambda_{\alpha_1\alpha_2\alpha_3}$ and
$\beta_3e_1+\beta_2e_2\in \Lambda_{\alpha_1\alpha_2\alpha_3}$.

(c) The characterization of $\Lambda$ with $|\Lambda^{*}|=8$
follows from 
\begin{eqnarray*}
\Lambda_{\alpha_1\alpha_2\alpha_3}^{*}
=\Lambda_{max}^{*}&\iff& 
2e_1\in\Lambda_{\alpha_1\alpha_2\alpha_3}\textup{ and }
2e_1+2e_2\in \Lambda_{\alpha_1\alpha_2\alpha_3}\\
&\iff& \alpha_1\in\{1;2\},\ \alpha_2\in\{1;2\},\\
&&\alpha_3e_1+\alpha_2e_2\neq e_1+2e_2\textup{ if }\alpha_1=2.
\end{eqnarray*}
Similarly, the characterization of $\Lambda$ with
$|\Lambda_{\alpha_1\alpha_2\alpha_3}^{*}|=4$ follows from
\begin{eqnarray*}
|\Lambda_{\alpha_1\alpha_2\alpha_3}|=4&\iff& 
\textup{either }2e_1\in\Lambda_{\alpha_1\alpha_2\alpha_3}\\
&&\textup{ or }2e_2\in \Lambda_{\alpha_1\alpha_2\alpha_3}\\
&&\textup{ or }2e_3\in \Lambda_{\alpha_1\alpha_2\alpha_3}\\
&\iff& \textup{either }\alpha_1\in\{1;2\}\\
&&\textup{ or }(\alpha_2,\alpha_3)\in\{(1,0),(2,0)\}\\
&&\textup{ or }
\Bigl((\alpha_2,\alpha_3)\in\{(1,1),2,2)\}\textup{ or }\\
&&\hspace*{2cm}
(\alpha_1,\alpha_2,\alpha_3)=(3,2,-1)\Bigr).
\end{eqnarray*}

(d) Recall that $C$ is the maximal $\Lambda_{max}$-ideal in 
$\Lambda_{\alpha_1\alpha_2\alpha_3}$.
It splits as in \eqref{9.9} because it is a 
$\Lambda_{max}$-ideal.
Obviously $\beta_1=\alpha_1$.
$\beta_2$ is the smallest multiple of $\alpha_2$ such that
$\Z\frac{\beta_2}{\alpha_2}\alpha_3 e_1$ contains 
$\Z\alpha_1e_1$, so 
$$\frac{\beta_2}{\alpha_2}
=\frac{\lcm(\alpha_1,\alpha_3)}{\alpha_3}
=\frac{\alpha_1}{\gcd(\alpha_1,\alpha_3)}.$$
Because of 
$$\alpha_3e_1+\alpha_2e_2=(\alpha_3-\alpha_2)e_1+
\alpha_2(e_1+e_2)$$
$\beta_3$ is the smallest multiple of $\alpha_2$ such that
$\Z\frac{\beta_3}{\alpha_2}(\alpha_3-\alpha_2)$ contains
$\Z\alpha_1e_1$, so 
$$\frac{\beta_3}{\alpha_2}
=\frac{\lcm(\alpha_1,\alpha_2-\alpha_3)}{\alpha_2-\alpha_3}
=\frac{\alpha_1}{\gcd(\alpha_1,\alpha_2-\alpha_3)}.$$

We leave the proof of \eqref{9.11} to the reader.
\eqref{9.12} and \eqref{9.13} are trivial. The ring 
$\Lambda_{max}/C\cong \Z_{\beta_1}\times\Z_{\beta_2}\times 
\Z_{\beta_3}$ contains the ring 
$\Z_{\lcm(\beta_1,\beta_2,\beta_3)}[1_A]$. This shows
\eqref{9.14}.
The possible input into 
$(\Lambda_{\alpha_1\alpha_2\alpha_3}/C)^{*}$ from 
$[\alpha_3e_1+\alpha_2e_2]$ is nontrivial. It has to 
be analyzed case by case.

(e) Compare Theorem \ref{t5.7}.
We start with $\www{\omega}_1=\alpha_3e_1+\alpha_2e_2$
and $\www{\omega}_2=\alpha_1e_1$. Because of
$\www{\omega}_1\www{\omega}_2=
\alpha_3\alpha_1e_1=\alpha_3\www{\omega}_2$ we set
\begin{eqnarray*}
\omega_1&:=& \www{\omega}_1-\alpha_3\cdot 1
=(\alpha_2-\alpha_3)e_2-\alpha_3e_3\\
\omega_2&:=& \www{\omega}_2=\alpha_1e_1
\end{eqnarray*}
and obtain
\begin{eqnarray*}
\omega_1\omega_2&=&0=ad\cdot 1,\\
\omega_1^2&=&(\alpha_2-2\alpha_3)\omega_1+
\frac{\alpha_3(\alpha_3-\alpha_2)}{\alpha_1}\omega_2
+\alpha_3(\alpha_2-\alpha_3)\cdot 1\\
&=& b\omega_1+a\omega_2-ac\cdot 1,\\
\omega_2^2&=& \alpha_1\omega_2 =d\omega_1+c\omega_2-bd\cdot 1,
\end{eqnarray*}
so $(a,b,c,d),\delta,t$ and $\tau_{\alpha_1\alpha_2\alpha_3}$ 
are as in \eqref{9.15}.
\hfill$\Box$

\bigskip
Next we find normal forms and the orders for all
$\varepsilon$-classes of full lattices in $A$. 
The conditions in \eqref{9.17} for a normal form in \eqref{9.16}
are much more involved than the conditions 
$\delta_2,\delta_3\in[0,\frac{1}{2}]$, $\delta_1\in[0;1]$ for a
semi-normal form in \eqref{9.16}.

\begin{lemma}\label{t9.6}
\begin{enumerate}
\item[(a)] 
Each $\varepsilon$-class of full lattices in $A$ contains
a unique full lattice with a $\Z$-basis
\begin{eqnarray}\label{9.16}
\uuuu{e}\cdot\begin{pmatrix}1&\delta_1&\delta_3\\
0&1&\delta_2\\0&0&1\end{pmatrix}
\textup{ with }\delta_1,\delta_2,\delta_3\in\Q
\end{eqnarray}
and
\begin{eqnarray}
\left.\begin{array}{lll}
\delta_2,\delta_3\in[0,\frac{1}{2}],&&\\
\delta_1\in[0,1]&\textup{if}&
(\delta_2,\delta_3)\in(0,\frac{1}{2})\times(0,\frac{1}{2}),\\
\delta_1\in[0,\frac{1}{2}]&\textup{if}&
(\delta_2,\delta_3)\in(0,\frac{1}{2})\times \{0,\frac{1}{2}\}
\cup \{0\}\times (0,\frac{1}{2})\\
&& \cup\, \{(0,0),(0,\frac{1}{2}),(\frac{1}{2},0)\},\\
\delta_1\in[0,2\delta_3]&\textup{if}& 
(\delta_2,\delta_3)\in\{\frac{1}{2}\}\times (0,\frac{1}{2}),\\
\delta_1\in[0,\frac{1}{2})&\textup{if}& 
(\delta_2,\delta_3)=(\frac{1}{2},\frac{1}{2}).
\end{array}\right\}\label{9.17}
\end{eqnarray}
The $\Z$-basis is unique. 
\item[(b)] 
Consider an order $\Lambda_{\alpha_1\alpha_2\alpha_3}$ and
a full lattice $L$ in $A$ with a $\Z$-basis as in \eqref{9.16}
(not necessarily also \eqref{9.17}). Then 
$\Lambda_{\alpha_1\alpha_2\alpha_3}\cdot L=L$ if and only if 
\begin{eqnarray}\label{9.18}
\alpha_1\delta_1,\ (\alpha_2-\alpha_3)\delta_1,\ 
\alpha_1\delta_3,\ \alpha_2\delta_2,
\ \alpha_3\delta_3-\alpha_2\delta_1\delta_2
\in\Z.
\end{eqnarray}
\item[(c)] 
Let $L$ be a full lattice in $A$ with a $\Z$-basis as in
\eqref{9.16} (not necessarily also \eqref{9.17}).
Then $\OO(L)=\Lambda_{\alpha_1\alpha_2\alpha_3}$ with
\begin{eqnarray}\label{9.19}
\left.\begin{array}{lll}
\alpha_1&=&\min(\beta_1\in\N\,|\, \beta_1\delta_1\in\Z,
\beta_1\delta_3\in\Z),\\
\alpha_2&=&\min(\beta_2\in\N\,|\, \textup{ a }\beta_3\in\Z
\textup{ exists with }\beta_2\delta_2\in\Z,\\
&&\hspace*{2cm} (\beta_2-\beta_3)\delta_1\in\Z,
\beta_3\delta_3-\beta_2\delta_1\delta_2\in\Z,\\
\alpha_3&=&\Bigl(\textup{the }\beta_3\in(-\frac{1}{2}\alpha_1,
\frac{1}{2}\alpha_1]\cap\Z\textup{ with }\\
&&(\alpha_2-\beta_3)\delta_1\in\Z,\beta_3\delta_3-
\alpha_2\delta_1\delta_2\in\Z\Bigr).
\end{array}\right\}
\end{eqnarray}
\end{enumerate}
\end{lemma}

{\bf Proof:} 
(a) From \eqref{9.2} it is clear that each
$\varepsilon$-class contains a full lattice $L$ with a 
$\Z$-basis as in \eqref{9.16} with 
$\delta_2,\delta_3\in[0,\frac{1}{2}]$ and $\delta_1\in[0,1]$.
It is also clear that in the interior cases 
$(\delta_2,\delta_3)\in(0,\frac{1}{2})\times (0,\frac{1}{2})$
this is unique and is a normal form.

But in the boundary cases
$(\delta_2,\delta_3)\in\{0,\frac{1}{2}\}\times [0,\frac{1}{2}]
\,\cup\, (0,\frac{1}{2})\times\{0,\frac{1}{2}\}$ 
also a full lattice $u\cdot L$ with a suitable unit 
$u\in\{\pm e_1\pm e_2\pm e_3\}$ has such a 
$\Z$-basis, with other coefficients.

{\bf The case} 
$(\delta_2,\delta_3)\in [0,\frac{1}{2}]\times\{0\}$:
If $\delta_1\in (\frac{1}{2},1]$, first we replace
$\delta_1e_1+e_2$ by $(\delta_1-1)e_1+e_2$, so 
$\delta_1$ by $\delta_1-1\in(-\frac{1}{2},0]$.
The full lattice $(-e_1+e_2+e_3)L$ has the $\Z$-basis
$\uuuu{e}\begin{pmatrix}1&1-\delta_1&0\\0&1&\delta_2\\0&0&1
\end{pmatrix}$.

{\bf The case} 
$(\delta_2,\delta_3)\in\{0\}\times (0,\frac{1}{2}]$: 
If $\delta_1\in(\frac{1}{2},1]$, first we replace 
$\delta_1e_1+e_2$ by $(\delta_1-1)e_1+e_2$, so 
$\delta_1$ by $\delta_1-1\in(-\frac{1}{2},0]$. The full lattice 
$(e_1-e_2+e_3)L$ has the $\Z$-basis 
$\uuuu{e}\begin{pmatrix}1&1-\delta_1&\delta_3\\0&1&0\\0&0&1
\end{pmatrix}$.

{\bf The case} 
$(\delta_2,\delta_3)\in (0,\frac{1}{2}]\times 
\{\frac{1}{2}\}$: If $\delta_1\in(\frac{1}{2},1]$,
first we replace $\delta_1e_1+e_2$ by $(\delta_1-1)e_1+e_2$
and $\frac{1}{2}e_1+\delta_2e_2+e_3$ by
$-\frac{1}{2}e_1+\delta_2e_2+e_3$.
The full lattice $(-e_1+e_2+e_3)L$ has the $\Z$-basis 
$\uuuu{e}\begin{pmatrix}1&1-\delta_1&\frac{1}{2}\\
0&1&\delta_2\\0&0&1 \end{pmatrix}$.

In the special case $(\delta_1,\delta_2,\delta_3)=
(\frac{1}{2},\frac{1}{2},\frac{1}{2})$ we replace
$\frac{1}{2}e_1+\frac{1}{2}e_2+e_3$ by
$(\frac{1}{2}-\frac{1}{2})e_1+(\frac{1}{2}-1)e_2+e_3$ and
$\frac{1}{2}e_1+e_2$ by $(\frac{1}{2}-1)e_1+e_2$.
The full lattice $(e_1-e_2+e_3)L$ has the $\Z$-basis 
$\uuuu{e}\begin{pmatrix}1&\frac{1}{2}&0\\0&1&\frac{1}{2}\\0&0&1
\end{pmatrix}$.

{\bf The case} 
$(\delta_2,\delta_3)\in\{\frac{1}{2}\}\times 
(0,\frac{1}{2})$: If $\delta_1\in (2\delta_3,1]$, first
we replace $\delta_3e_1+\frac{1}{2}e_2+e_3$ by 
$(\delta_3-\delta_1)e_1+(\frac{1}{2}-1)e_2+e_3$.
The full lattice $\www{L}:=(-e_1-e_2+e_3)L$ has the $\Z$-basis 
$$\uuuu{e}\begin{pmatrix}1&\delta_1&\delta_1-\delta_3\\
0&1&\frac{1}{2}\\0&0&1\end{pmatrix}=:
\uuuu{e}\begin{pmatrix}1&\www{\delta}_1&\www{\delta}_3\\
0&1&\frac{1}{2}\\0&0&1\end{pmatrix}.$$
Here $\www{\delta}_1\in [0,2\www{\delta}_3)$ because
$$\www{\delta}_1<2\www{\delta}_3\iff \delta_1<2\delta_1
-2\delta_3\iff \delta_1>2\delta_3.$$
If $\www{\delta}_3=\delta_1-\delta_3\in[0,\frac{1}{2}]$
the case is fine. Suppose 
$\www{\delta}_3\in (\frac{1}{2},1)$.
We replace $\www{\delta}_1e_1+e_2$ by $(\www{\delta}_1-1)e_1+e_2$
and $\www{\delta}_3e_1+\frac{1}{2}e_2+e_3$ by 
$(\www{\delta}_3-1)e_1+\frac{1}{2}e_2+e_3$. 
The full lattice $\widehat{L}
=(-e_1+e_2+e_3)\www{L}$ has the $\Z$-basis
$$\uuuu{e}\begin{pmatrix}1&1-\www{\delta}_1&1-\www{\delta}_3\\
0&1&\frac{1}{2}\\0&0&1\end{pmatrix}=:
\uuuu{e}\begin{pmatrix}1&\widehat{\delta}_1&\widehat{\delta}_3\\
0&1&\frac{1}{2}\\0&0&1\end{pmatrix}.$$
Now
$\widehat{\delta}_3\in (0,\frac{1}{2})$ and 
$\widehat{\delta}_1\in [0,2\widehat{\delta}_3)$ because
\begin{eqnarray*}
\widehat{\delta}_1&=& 1-\www{\delta}_1=1-\delta_1\geq 0,\\
\widehat{\delta}_1-2\widehat{\delta}_3&=&
(1-\www{\delta}_1)-2(1-\www{\delta}_3)
=-1-\www{\delta}_1+2\www{\delta}_3\\
&=& -1-\delta_1+2(\delta_1-\delta_3)
= -1+\delta_1-2\delta_3\leq -2\delta_3<0.
\end{eqnarray*}
We leave it to the reader to show that in the boundary cases
no further reductions are possible.

(b) The full lattice $\Lambda_{\alpha_1\alpha_2\alpha_3}\cdot L$ 
is generated over $\Z$ by the following tuple of elements of $A$,
$$\uuuu{e}\begin{pmatrix}\alpha_1 & \alpha_1\delta_1 & 
\alpha_1\delta_3 & \alpha_3 & \alpha_3\delta_1 & 
\alpha_3\delta_3 & 1 & \delta_1 & \delta_3 \\
0 & 0 & 0 & 0 & \alpha_2 & \alpha_2\delta_2 & 0 & 1 & \delta_2\\
0 & 0 & 0 & 0 & 0 & 0 & 0 & 0 & 1 \end{pmatrix}.$$
The elements from the columns $1,4,7,8,9$ are in $L$ and 
generate $L$. The conditions that the elements from the columns
$2,3,5,6$ are in $L$ are equivalent to \eqref{9.18}.
Here observe
\begin{eqnarray*}
\alpha_3\delta_1e_1+\alpha_2e_2 &=&
(\alpha_3-\alpha_2)\delta_1e_1 + \alpha_2(\delta_1e_1+e_2),\\
\alpha_3\delta_3e_1+\alpha_2\delta_2e_2 &=& 
(\alpha_3\delta_3-\alpha_2\delta_1\delta_2)e_1 
+ \alpha_2\delta_2(\delta_1e_1+e_2).
\end{eqnarray*}

(c) This follows from part (b).\hfill$\Box$

\begin{remarks}\label{t9.7}
\begin{enumerate}
\item[(i)] 
The parts (c) and (d) of Lemma \ref{t9.5} give the 
ingredients for formula \eqref{9.4} for the size of the
group $G([\Lambda]_\varepsilon)$ in the case $n=3$.
\item[(ii)] 
The parts (a) and (b) of Lemma \ref{t9.6} can be used to
determine for a given order 
$\Lambda_{\alpha_1\alpha_2\alpha_3}$ the unique representatives
in part (a) of the set $\{[L]_\varepsilon\,|\, 
L\in\LL(A),\OO(L)\supset\Lambda_{\alpha_1\alpha_2\alpha_3}\}$.
\item[(iii)] 
Consider three different integers $\lambda_1,\lambda_2$ and
$\lambda_3$ and the cyclic order 
$\Lambda=\Z[\lambda_1e_1+\lambda_2e_2+\lambda_3e_3]$. 
An $\varepsilon$-class $[L]_\varepsilon$ of a full lattice $L$
with $\OO(L)\supset \Lambda$ determines a conjugacy class of
integer $3\times 3$ matrices with eigenvalues $\lambda_1,
\lambda_2$ and $\lambda_3$. Lemma \ref{t9.6} (a) offers a unique
matrix $M$ in the conjugacy class via 
$(\lambda_1e_1+\lambda_2e_2+\lambda_3e_3)\cdot \uuuu{e}D
=\uuuu{e}D\cdot M$ where $\uuuu{e}D$ is the $\Z$-basis
in \eqref{9.16} with \eqref{9.17} for the full lattice in
Lemma \ref{t9.6} (a) in the $\varepsilon$-class 
$[L]_\varepsilon$. In the case $(\lambda_1,\lambda_2,\lambda_3)
=(-2,2,0)$ column 5 in Table \ref{table4.4} gives the matrices
in this normal form for the conjugacy classes of matrices.
\end{enumerate}
\end{remarks}

\subsection{Conjugacy classes of matrices with
eigenvalues $2,0,-2$} 

First we find heuristically six $3\times 3$ matrices 
with eigenvalues $2,0,-2$ and guess that they are in 
different conjugacy classes. It will turn out that this is true.
From different decompositions of the characteristic polynomial
$(t-2)t(t+2)$ one finds immediately five matrices
$M_1,M_2,M_3,M_4$ and $M_6$ which arise as 
companion matrices for the factors of the decompositions
of the characteristic polynomial.

Example \ref{t9.3} provides three conjugacy classes of
$2\times 2$ matrices with eigenvalues $2$ and $-2$,
with representatives 
$\begin{pmatrix}2&0\\0&-2\end{pmatrix}$, 
$\begin{pmatrix}2&1\\0&-2\end{pmatrix}$,
$\begin{pmatrix}2&2\\0&-2\end{pmatrix}$. 
The third one gives rise to the matrix $M_5$.
Observe that the matrices
$\begin{pmatrix}2&1\\0&-2\end{pmatrix}$ and 
$\begin{pmatrix}0&4\\1&0\end{pmatrix}$ are conjugate.
Table \ref{table4.1} lists the matrices $M_1,...,M_6$ and
$\Z$-bases $\uuuu{e}\cdot B_i$ ($i\in\{1,...,6\}$)
of six chosen full lattices with
$$(-2e_1+2e_2)\cdot \uuuu{e}B_i 
= \uuuu{e}B_i\cdot M_i.$$
The six full lattices turn out to be pairwise different orders
$\Lambda_{\alpha_1\alpha_2\alpha_3}$. Also these orders
are listed. As the orders are pairwise different, the 
conjugacy classes are pairwise different.

\begin{table}
\begin{eqnarray*}
\begin{array}{ccccc}
i & \textup{factors} & M_i & \uuuu{e}B_i & 
\Lambda_{\alpha_1\alpha_2\alpha_3}\\
\hline 
1 & \begin{matrix}t-2\\t\\t+2\end{matrix} 
&\begin{pmatrix}2&0&0\\0&0&0\\0&0&-2\end{pmatrix} & 
\uuuu{e}\begin{pmatrix}0&0&1\\1&0&0\\0&1&0\end{pmatrix} & 
\Lambda_{110}=\Lambda_{max}\\
2 & \begin{matrix}t(t-2)\\ t+2\end{matrix} 
& \begin{pmatrix}0&0&0\\1&2&0\\0&0&-2\end{pmatrix} & 
\uuuu{e}\begin{pmatrix}0&0&1\\1&2&0\\1&0&0\end{pmatrix} & 
\Lambda_{120} \\
3 & \begin{matrix}t(t+2)\\ t-2\end{matrix} 
& \begin{pmatrix}0&0&0\\1&-2&0\\0&0&2\end{pmatrix} & 
\uuuu{e}\begin{pmatrix}1&-2&0\\0&0&1\\1&0&0\end{pmatrix} & 
\Lambda_{210} \\
4 & \begin{matrix}(t-2)(t+2)\\ t\end{matrix} 
& \begin{pmatrix}0&4&0\\1&0&0\\0&0&0\end{pmatrix} & 
\uuuu{e}\begin{pmatrix}1&-2&0\\1&2&0\\0&0&1\end{pmatrix} & 
\Lambda_{411} \\
5 & \begin{matrix}(t-2)(t+2)\\ t\end{matrix} 
& \begin{pmatrix}2&2&0\\0&-2&0\\0&0&0\end{pmatrix} & 
\uuuu{e}\begin{pmatrix}0&1&0\\2&1&0\\0&0&1\end{pmatrix} & 
\Lambda_{211} \\
6 & (t-2)t(t+2) 
& \begin{pmatrix}0&0&0\\1&0&4\\0&1&0\end{pmatrix} & 
\uuuu{e}\begin{pmatrix}1&-2&4\\1&2&4\\1&0&0\end{pmatrix} & 
\Lambda_{8,2,-2} 
\end{array}
\end{eqnarray*}
\caption{Six matrices with eigenvalues $2,0,-2$ and six 
$\Z$-bases of full lattices (in fact, orders) which realize 
these matrices w.r.to the multiplication with $-2e_1+2e_2$}
\label{table4.1}
\end{table}

The following systematic analysis will show that there are
four more conjugacy classes. Consider the endomorphism
$$(\textup{multiplication with }-2e_1+2e_2):A\to A$$
and the order $\Z[-2e_1+2e_2]=\Lambda_{8,2,-2}$
which the element $-2e_1+2e_2$ generates.
By Theorem \ref{t6.3} 
the conjugacy classes of integer matrices are in 1:1 
correspondence with the elements of the set
$$\{[L]_\varepsilon\,|\, L\in\LL(A),\OO(L)\supset
\Lambda_{8,2,-2}\}.$$
Lemma \ref{t9.5} (b) allows to determine the orders
$\Lambda\supset \Lambda_{8,2,-2}$ easily.
They are eight orders. Figure \ref{fig4.1} lists them
and their inclusions (inclusions are indicated by
non-horizontal lines and transitivity of inclusion).
It lists also $L_1$ and $L_2$ which are full lattices
which will be explained later.

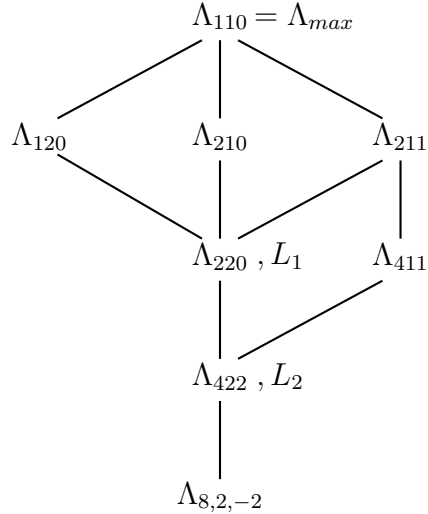
\begin{figure}
\begin{tikzpicture}[scale=0.8]
\node at (3,0) {$\Lambda_{8,2,-2}$};
\node at (3,2) {$\Lambda_{422}$};
\node at (4,2) {$,L_2$};
\node at (3,4) {$\Lambda_{220}$};
\node at (4,4) {$,L_1$};
\node at (3,6) {$\Lambda_{210}$};
\node at (3,8) {$\Lambda_{110}$};
\node at (4.4,8) {$=\Lambda_{max}$};
\node at (6,4) {$\Lambda_{411}$};
\node at (0,6) {$\Lambda_{120}$};
\node at (6,6) {$\Lambda_{211}$};
\draw [thick] (3,0.3)--(3,1.6);
\draw [thick] (3,2.3)--(3,3.6);
\draw [thick] (3,4.3)--(3,5.6);
\draw [thick] (3,6.3)--(3,7.6);
\draw [thick] (6,4.3)--(6,5.6);
\draw [thick] (3.3,2.3)--(5.7,3.6);
\draw [thick] (3.3,4.3)--(5.7,5.6);
\draw [thick] (0.3,6.3)--(2.7,7.6);
\draw [thick] (5.7,6.3)--(3.3,7.6);
\draw [thick] (0.3,5.7)--(2.7,4.3);
\end{tikzpicture}
\caption[Figure 4.1]{The orders which contain the order
$\Lambda_{8,2,-2}$, and two other full lattices $L_1$ and $L_2$}
\label{fig4.1}
\end{figure}

Lemma \ref{t9.5} (c) allows for each order $\Lambda$ of 
the eight orders to determine the size of the group
$\Lambda^{*}$ of units.
They are listed in column 2 of table \ref{table4.2}. 
All five orders $\Lambda$ in Lemma \ref{t9.5} (c) with
$|\Lambda^{*}|=8$ turn up.

Lemma \ref{t9.5} (d) allows to determine for each order
$\Lambda$ of the eight orders the conductor 
$$C:= \Lambda:\Lambda_{max}=\Z\beta_1e_1
+\Z\beta_2e_2+\Z\beta_3e_3$$
and the sizes of the groups of units 
$(\Lambda_{max}/C)^{*}$ and $(\Lambda/C)^{*}$.
With formula \eqref{9.4} one calculates the size of the 
group $G([\Lambda]_\varepsilon)$. 
Table \ref{table4.2}
gives in the columns 3, 4, 5 and 6 these data.

\begin{table}
\begin{eqnarray*}
\begin{array}{cccccc}
\Lambda & |\Lambda^{*}| & (\beta_1,\beta_2,\beta_3) & 
|(\Lambda_{max}/C)^{*}| & |(\Lambda/C)^{*}| & 
|G([\Lambda]_\varepsilon)| \\ \hline 
\Lambda_{110} & 8 & (1,1,1) & 1 & 1 & 
\frac{1}{1}\cdot \frac{8}{8}=1 \\
\Lambda_{120} & 8 & (1,2,2) & 1 & 1 & 
\frac{1}{1}\cdot \frac{8}{8}=1 \\
\Lambda_{210} & 8 & (2,1,2) & 1 & 1 & 
\frac{1}{1}\cdot \frac{8}{8}=1 \\
\Lambda_{211} & 8 & (2,2,1) & 1 & 1 & 
\frac{1}{1}\cdot \frac{8}{8}=1 \\
\Lambda_{220} & 8 & (2,2,2) & 1 & 1 & 
\frac{1}{1}\cdot \frac{8}{8}=1 \\
\Lambda_{411} & 4 & (4,4,1) & 4 & 2 & 
\frac{4}{2}\cdot \frac{4}{8}=1 \\
\Lambda_{422} & 4 & (4,4,2) & 4 & 2 & 
\frac{4}{2}\cdot \frac{4}{8}=1 \\
\Lambda_{8,2,-2} & 2 & (8,8,4) & 32 & 8 & 
\frac{32}{8}\cdot \frac{2}{8}=1 
\end{array}
\end{eqnarray*}
\caption{Data for the eight orders which are needed in 
formula \eqref{9.4} for $|G([\Lambda]_\varepsilon)|$}
\label{table4.2}
\end{table}

By Theorem \ref{t5.12} (d)
the map $G([\Lambda_{8,2,-2}]_\varepsilon)\to 
G([\Lambda]_\varepsilon)$ is surjective for each order
$\Lambda\supset \Lambda_{8,2,-2}$. Therefore 
$|G([\Lambda_{8,2,-2}]_\varepsilon)|=1$ implies
$|G([\Lambda]_\varepsilon)|=1$ for any order 
$\Lambda\supset\Lambda_{8,2,-2}$, so it would have been
sufficient to calculate the last row of table \ref{table4.2}.

By Theorem \ref{t5.5} (d)
$|G([\Lambda]_\varepsilon)|=1$ for each order 
$\Lambda\supset\Lambda_{8,2,-2}$ implies that each 
$w$-class in the semigroup 
$$\{[L]_\varepsilon\,|\, L\in\LL(A),\OO(L)\supset
\Lambda_{8,2,-2}\}$$ has only one element, so this semigroup
coincides with its quotient by $w$-equivalence.

Lemma \ref{t9.5} (e) allows to determine for each order
$\Lambda\supset\Lambda_{8,2,-2}$ the number
$\tau\in\N$ of $w$-classes in the set
$\{[L]_\varepsilon\,|\, L\in\LL(A),\OO(L)=\Lambda\}$.
Table \ref{table4.3} 
gives the data $(a,b,c,d),\mu,t$ and $\tau$ from
Lemma \ref{t9.5} (e) for each order 
$\Lambda\supset\Lambda_{8,2,-2}$. 

\begin{table}
\begin{eqnarray*}
\begin{array}{ccccc}
\textup{Order} & (a,b,c,d) & \mu & t & \tau=2^t \\ \hline
\Lambda_{110} & (0,1,1,0) & 1 & 0 & 1 \\
\Lambda_{120} & (0,2,1,0) & 1 & 0 & 1 \\
\Lambda_{210} & (0,1,2,0) & 1 & 0 & 1 \\
\Lambda_{211} & (0,-1,2,0) & 1 & 0 & 1 \\
\Lambda_{220} & (0,2,2,0) & 2 & 1 & 2 \\
\Lambda_{411} & (0,-1,4,0) & 1 & 0 & 1 \\
\Lambda_{422} & (0,-2,4,0) & 2 & 1 & 2 \\
\Lambda_{8,2,-2} & (1,6,8,0) & 1 & 0 & 1
\end{array}
\end{eqnarray*}
\caption{Data in Lemma \ref{t9.5} (e) for the eight orders
$\Lambda\supset\Lambda_{8,2,-2}$}
\label{table4.3}
\end{table}

In fact, the orders
\begin{eqnarray*}
\Lambda_{120}= \Z[e_1+2e_2],&&
\Lambda_{210}= \Z[2e_1+e_2]\\
\Lambda_{211}= \Z[-e_1+e_2],&&
\Lambda_{8,2,-2}= \Z[-2e_1+2e_2],
\end{eqnarray*}
are cyclic, so by Theorem \ref{t4.12} 
$\tau=1$ for each of them anyway.
The order $\Lambda_{411}$ satisfies 
$\Lambda_{max}=\Lambda_{411}+e_1\cdot \Lambda_{411}$,
so by Theorem \ref{t4.9} 
$\tau=1$ for $\Lambda_{411}$ (and for $\Lambda_{211}$
and $\Lambda_{max}$) anyway.
So it would have been sufficient to calculate the rows
for $\Lambda_{220}$ and $\Lambda_{422}$ in table
\ref{table4.3}. 

$\tau(\Lambda_{220})=\tau(\Lambda_{422})=2$ says that there
is one $\varepsilon$-class $[L_1]_\varepsilon$
of non-invertible full lattices with $\OO(L_1)=\Lambda_{220}$
and that there is one $\varepsilon$-class $[L_2]_\varepsilon$
of non-invertible full lattices with $\OO(L_2)=\Lambda_{422}$.
One way to find representatives $L_1$ and $L_2$ is to
determine the full lattices in normal form in 
Lemma \ref{t9.6} (a) which satisfy \eqref{9.18} for
$(\alpha_1,\alpha_2,\alpha_3)=(8,2,-2)$. 
One finds ten lattices, eight in the $\varepsilon$-classes
of the eight orders $\Lambda\supset\Lambda_{8,2,-2}$ and 
two representatives for the classes $[L_1]_\varepsilon$
and $[L_2]_\varepsilon$. Here we just give for two 
representatives $L_1$ and $L_2$ pretty $\Z$-bases:
\begin{eqnarray}\label{9.20} 
L_1:\quad \uuuu{e}
\begin{pmatrix}1&\frac{1}{2}&1\\0&1&1\\0&0&1\end{pmatrix},\quad
L_2:\quad \uuuu{e}
\begin{pmatrix}4&-1&1\\0&1&1\\0&0&1\end{pmatrix}.
\end{eqnarray}

Table \ref{table4.4} 
lists for each full lattice 
$L\in\{\Lambda\textup{ order }\,|\, \Lambda\supset
\Lambda_{8,2,-2}\}\ \cup\ \{L_1,L_2\}$ the following data:

\begin{enumerate}
\item[Column 2:] 
The $\Z$-basis $\uuuu{e}B$ in Lemma \ref{t9.5} (a)
of $L$ if $L$ is an order respectively the $\Z$-basis in
\eqref{9.20} if $L\in\{L_1,L_2\}$.
\item[Column 3:]
The matrix $M\in M_{3\times 3}(\Z)$ with 
$(-2e_1+2e_2)\cdot\uuuu{e}B=\uuuu{e}B\cdot M$.
\item[Column 4:]
The $\Z$-basis $\uuuu{e}D$ of the
$\varepsilon$-equivalent full lattice $\www{L}$
in the normal form in Lemma \ref{t9.6} (a).
\item[Column 5:]
The matrix $\www{M}\in M_{3\times 3}(\Z)$ with 
$(-2e_1+2e_2)\cdot\uuuu{e}D=\uuuu{e}D\cdot M$.
\end{enumerate}
The ten matrices $M$ in column 4 and the ten matrices
$\www{M}$ in column 5 are different representatives of the
ten conjugacy classes of integer $3\times 3$ matrices with
eigenvalues $2,0,-2$.

\begin{table}
\begin{eqnarray*}
\begin{array}{ccccc}
L & \uuuu{e}B & M & \uuuu{e}D & \www{M} \\ \hline
\Lambda_{110} & 
\uuuu{e}\begin{pmatrix}1&0&1\\0&1&1\\0&0&1\end{pmatrix} &
\begin{pmatrix}-2&0&-2\\0&2&2\\0&0&0\end{pmatrix} & 
\uuuu{e}\begin{pmatrix}1&0&0\\0&1&0\\0&0&1\end{pmatrix} & 
\begin{pmatrix}-2&0&0\\0&2&0\\0&0&0\end{pmatrix} \\
\Lambda_{120} & 
\uuuu{e}\begin{pmatrix}1&0&1\\0&2&1\\0&0&1\end{pmatrix} &
\begin{pmatrix}-2&0&-2\\0&2&1\\0&0&0\end{pmatrix} & 
\uuuu{e}\begin{pmatrix}1&0&0\\0&1&\frac{1}{2}\\0&0&1
\end{pmatrix} & 
\begin{pmatrix}-2&0&0\\0&2&1\\0&0&0\end{pmatrix} \\
\Lambda_{210} & 
\uuuu{e}\begin{pmatrix}2&0&1\\0&1&1\\0&0&1\end{pmatrix} &
\begin{pmatrix}-2&0&-1\\0&2&2\\0&0&0\end{pmatrix} & 
\uuuu{e}\begin{pmatrix}1&0&\frac{1}{2}\\0&1&0\\0&0&1
\end{pmatrix} & 
\begin{pmatrix}-2&0&-1\\0&2&0\\0&0&0\end{pmatrix} \\
\Lambda_{211} & 
\uuuu{e}\begin{pmatrix}2&1&1\\0&1&1\\0&0&1\end{pmatrix} &
\begin{pmatrix}-2&-2&-2\\0&2&2\\0&0&0\end{pmatrix} & 
\uuuu{e}\begin{pmatrix}1&\frac{1}{2}&0\\0&1&0\\0&0&1
\end{pmatrix} & 
\begin{pmatrix}-2&-2&0\\0&2&0\\0&0&0\end{pmatrix} \\
\Lambda_{220} & 
\uuuu{e}\begin{pmatrix}2&0&1\\0&2&1\\0&0&1\end{pmatrix} &
\begin{pmatrix}-2&0&-1\\0&2&1\\0&0&0\end{pmatrix} & 
\uuuu{e}\begin{pmatrix}1&0&\frac{1}{2}\\0&1&\frac{1}{2}\\
0&0&1\end{pmatrix} & 
\begin{pmatrix}-2&0&-1\\0&2&1\\0&0&0\end{pmatrix} \\
L_1 & 
\uuuu{e}\begin{pmatrix}1&\frac{1}{2}&1\\0&1&1\\0&0&1
\end{pmatrix} &
\begin{pmatrix}-2&-2&-3\\0&2&2\\0&0&0\end{pmatrix} & 
\uuuu{e}\begin{pmatrix}1&\frac{1}{2}&\frac{1}{2}\\0&1&0\\
0&0&1\end{pmatrix} & 
\begin{pmatrix}-2&-2&-1\\0&2&0\\0&0&0\end{pmatrix} \\
\Lambda_{411} & 
\uuuu{e}\begin{pmatrix}4&1&1\\0&1&1\\0&0&1\end{pmatrix} &
\begin{pmatrix}-2&-1&-1\\0&2&2\\0&0&0\end{pmatrix} & 
\uuuu{e}\begin{pmatrix}1&\frac{1}{4}&0\\0&1&0\\0&0&1
\end{pmatrix} & 
\begin{pmatrix}-2&-1&0\\0&2&0\\0&0&0\end{pmatrix} \\
\Lambda_{422} & 
\uuuu{e}\begin{pmatrix}4&2&1\\0&2&1\\0&0&1\end{pmatrix} &
\begin{pmatrix}-2&-2&-1\\0&2&1\\0&0&0\end{pmatrix} & 
\uuuu{e}\begin{pmatrix}1&\frac{1}{2}&\frac{1}{4}\\
0&1&\frac{1}{2}\\0&0&1\end{pmatrix} & 
\begin{pmatrix}-2&-2&-1\\0&2&1\\0&0&0\end{pmatrix} \\
L_2 & 
\uuuu{e}\begin{pmatrix}4&-1&1\\0&1&1\\0&0&1\end{pmatrix} &
\begin{pmatrix}-2&1&0\\0&2&2\\0&0&0\end{pmatrix} & 
\uuuu{e}\begin{pmatrix}1&\frac{1}{4}&\frac{1}{2}\\0&1&0\\
0&0&1\end{pmatrix} & 
\begin{pmatrix}-2&-1&-1\\0&2&0\\0&0&0\end{pmatrix} \\
\Lambda_{8,2,-2} & 
\uuuu{e}\begin{pmatrix}8&-2&1\\0&2&1\\0&0&1\end{pmatrix} &
\begin{pmatrix}-2&1&0\\0&2&1\\0&0&0\end{pmatrix} & 
\uuuu{e}\begin{pmatrix}1&\frac{1}{4}&\frac{3}{8}\\
0&1&\frac{1}{2}\\0&0&1\end{pmatrix} & 
\begin{pmatrix}-2&-1&-1\\0&2&1\\0&0&0\end{pmatrix} 
\end{array}
\end{eqnarray*}
\caption{Ten conjugacy classes of $3\times 3$ matrices: 
two families of ten full lattices (given by $\Z$-bases) 
and of ten matrices}
\label{table4.4}
\end{table}

Table \ref{table4.5} 
is the multiplication table for the eight orders
$\Lambda\supset\Lambda_{8,2,-2}$ 
and the full lattices $L_1$ and $L_2$.
In most cases the product of two full lattices is given,
though in most cases with a factor $L_1$ only the
$\varepsilon$-class of the product is given.
Only the products with a factor $L_1$ or $L_2$ have to
be calculated. The squares $L_1^2$ and $L_2^2$ are invertible
by Theorem \ref{t4.4}, so they are $\varepsilon$-equivalent
to two of the eight orders. The product of two orders is by 
Lemma \ref{t4.2} (c)
the minimal order which contains the two orders, so it can
be read off from figure \ref{fig4.1}.
We did not calculate the division table.

\begin{table}
\begin{eqnarray*}
\begin{array}{ccccccccccc}
\cdot & \Lambda_{110} & \Lambda_{120} & \Lambda_{210} & 
\Lambda_{211} & \Lambda_{220} & L_1 & 
\Lambda_{411} & \Lambda_{422} & L_2 & \Lambda_{8,2,-2} \\ \hline
\Lambda_{110} & \Lambda_{110} & \Lambda_{110} & \Lambda_{110} & 
\Lambda_{110} & \Lambda_{110} & [\Lambda_{110}]_\varepsilon & 
\Lambda_{110} & \Lambda_{110} & \Lambda_{110} & \Lambda_{110} \\
\Lambda_{120} & & \Lambda_{120} & \Lambda_{110} & 
\Lambda_{110} & \Lambda_{120} & [\Lambda_{110}]_\varepsilon & 
\Lambda_{110} & \Lambda_{120} & \Lambda_{110} & \Lambda_{120} \\
\Lambda_{210} & & & \Lambda_{210} & 
\Lambda_{110} & \Lambda_{210} & [\Lambda_{110}]_\varepsilon & 
\Lambda_{110} & \Lambda_{210} & \Lambda_{110} & \Lambda_{210} \\
\Lambda_{211} & & & &
\Lambda_{211} & \Lambda_{211} & [\Lambda_{110}]_\varepsilon & 
\Lambda_{211} & \Lambda_{211} & \Lambda_{211} & \Lambda_{211} \\
\Lambda_{220} & & & & 
& \Lambda_{220} & L_1 & 
\Lambda_{211} & \Lambda_{220} & \Lambda_{211} & \Lambda_{220} \\
L_1           & & & & & & [\Lambda_{110}]_\varepsilon & 
[\Lambda_{110}]_\varepsilon & L_1 & 
[\Lambda_{110}]_\varepsilon & L_1 \\
\Lambda_{411} & & & & & & & 
\Lambda_{411} & \Lambda_{411} & \Lambda_{211} & \Lambda_{411} \\
\Lambda_{422} & & & & & & & & \Lambda_{422} & L_2 & 
\Lambda_{422} \\
L_2           & & & & & & & & & \Lambda_{211} & L_2 \\
\Lambda_{8,2,-2} &  & & & & & & & & & \Lambda_{8,2,-2} 
\end{array}
\end{eqnarray*}
\caption{The multiplication table for the ten full lattices
respectively (when one factor is $L_1$) their $\varepsilon$-classes}
\label{table4.5}
\end{table}

\section{The cases with a single Jordan block}\label{s10}
\setcounter{equation}{0}
\setcounter{table}{0}
\setcounter{figure}{0}

\subsection{Observations on the case of arbitrary rank}
\label{s10.1}

\noindent 
Let $n\in\N$. An integer $n\times n$ matrix with a single
Jordan block has one integer eigenvalue. We can restrict 
to the eigenvalue 0, so the nilpotent case, by subtracting
the eigenvalue times the unit matrix from the given matrix.

The corresponding algebra is
\begin{eqnarray}\label{10.1}
A=\Q\cdot 1_A\oplus \Q \cdot a+ ...+\Q\cdot a^{n-1}
\quad\textup{with}\quad a^n=0.
\end{eqnarray}
It is irreducible. Its separable part is $F=\Q\cdot 1_A$,
its nilpotent part is $\NN=\Q\cdot a+ ... + \Q\cdot a^{n-1}$.
The conjugacy classes of integer $n\times n$ matrices 
correspond 1:1 to the $\varepsilon$-classes of
$\Lambda_a$-ideals 
$\{[L]_\varepsilon\,|\, L\in\LL(A),\OO(L)\supset\Lambda_a\}$
where $\Lambda_a$ is the cyclic order
$$\Lambda_a=\Z\cdot 1_A+\Z\cdot a + ...+ \Z\cdot a^{n-1}
=\Z[a].$$
Here a $\Z$-basis $\BB\in M_{1\times n}(L)$ 
of a $\Lambda_a$-ideal $L$ gives rise
to the matrix $M\in M_{n\times n}(\Z)$ with
$$a\cdot\BB=\BB\cdot M.$$

The fact that the separable part of $A$ is $F=\Q\cdot 1_A$ has
strong implications.

\begin{lemma}\label{t10.1}
For each order $\Lambda$ in $A$ the group of $\varepsilon$-classes
of invertible exact $\Lambda$-ideals is just
$G([\Lambda]_\varepsilon)=\{[\Lambda]_\varepsilon\}$.
Each $w$-class of full lattices contains only one $\varepsilon$-class,
so $W(\EE(A))=\EE(A)$. 
\end{lemma}

{\bf Proof:}
$F=\Q\cdot 1_A\subset A$ has only one order, 
the order $\Lambda_0:=\Z\cdot 1_A$. 
It satisfies $G([\Lambda_0]_\varepsilon)
=\{[\Lambda_0]_\varepsilon\}$. 
Under the projection 
$\pr_F:A\to F$ with respect to the decomposition
$A=F\oplus \NN$, each order $\Lambda$ is mapped to $\Lambda_0$.
By Theorem \ref{t5.10}
the induced group homomorphism 
$G([\Lambda]_\varepsilon)\to G([\Lambda_0]_\varepsilon)$ is
an isomorphism. Therefore 
$G([\Lambda]_\varepsilon)=\{[\Lambda]_\varepsilon\}$.
By Theorem \ref{t5.5} (d)
each $w$-class of full lattices contains only one
$\varepsilon$-class.\hfill$\Box$ 

\bigskip
Write $\uuuu{a}:=(1_A,a,a^2, ...,a^{n-1})$. Each full lattice
$L\in\LL(A)$ has a unique $\Z$-basis of the shape
\begin{eqnarray}\label{10.2}
\uuuu{a}\cdot 
\begin{pmatrix}\beta_{11} & 0 & \cdots & 0\\
\beta_{21} & \beta_{22} & \ddots & \vdots \\
\vdots & \ddots & \ddots & 0 \\
\beta_{n1} & \cdots & \beta_{n,n-1} & \beta_{nn}\end{pmatrix}
\quad\textup{with }\beta_{ii}\in\Q_{>0}
\textup{ for }1\leq i\leq n\\
\textup{and }\beta_{ij}\in (-\frac{1}{2}\beta_{ii},
\frac{1}{2}\beta_{ii}]\cap \Q\textup{ for }i>j.\nonumber
\end{eqnarray}
Multiplication of $L$ with the unit 
$u=\bigl(\sum_{i=1}^n\beta_{i1}a^{i-1}\bigr)^{-1}\in A^{*}$ 
leads to the full lattice $uL$ which has a $\Z$-basis of the shape
\begin{eqnarray}\label{10.3}
\uuuu{a}\cdot 
\begin{pmatrix} 1 & 0 & \cdots & 0 \\
0 & \gamma_{22} & \ddots & \vdots \\
\vdots & \vdots & \ddots & 0  \\
0 & \gamma_{n2} & \cdots & \gamma_{nn} 
\end{pmatrix}
\quad\textup{with }
\gamma_{ii}\in\Q_{>0}\textup{ for }2\leq i\leq n\\
\textup{and }\gamma_{ij}\in(-\frac{1}{2}\gamma_{ii},
\frac{1}{2}\gamma_{ii}]
\cap\Q\textup{ for }i>j\geq 2\nonumber
\end{eqnarray}
Though this is not unique. It is only a semi-normal form,
because one can add multiples of the other columns to the
first column and multiply $uL$ with a new unit similar
to the one above.

Each order $\Lambda$ in $\Lambda$ satisfies $\Z\cdot 1_A\subset 
\Lambda \subset \Z\cdot 1_A \oplus \NN$ because of the proof
of Lemma \ref{t10.1}. 
It has a unique $\Z$-basis of the shape
\begin{eqnarray}\label{10.4}
\uuuu{a}\cdot 
\begin{pmatrix} 1 & 0 & \cdots & 0 \\
0 & \alpha_{22} & \ddots & \vdots \\
\vdots & \vdots & \ddots & 0  \\
0 & \alpha_{n2} & \cdots & \alpha_{nn} 
\end{pmatrix}
\quad\textup{with }\alpha_{ii}\in\Q_{>0}\textup{ for }2\leq i<n\\
\textup{and }\alpha_{ij}\in (-\frac{1}{2}\alpha_{ii},
\frac{1}{2}\alpha_{ii}]\cap \Q\textup{ for }i>j\geq 2.\nonumber
\end{eqnarray}
For $n\geq 3$ the condition that $\Lambda$ is 
multiplication invariant gives additional constraints 
on the entries $\alpha_{ij}$ in \eqref{10.4}.

Theorem \ref{t10.2} is Theorem \ref{t4.4} in the given case.
We offer the following independent proof because it is short
and instructive. 

\begin{theorem}\label{t10.2}
Let $n\in\N$ and let $A$ be as in \eqref{10.1}.
For each full lattice $L\in\LL(A)$ the $(n-1)$-th power
$L^{n-1}$ is invertible.
\end{theorem}

{\bf Proof:}
We can restrict to a full lattice as in \eqref{10.3},
so with $1_A\in L$. Each power $L^k$ with 
$k\in\{1,2,...,n-1\}$ has also a $\Z$-basis as in \eqref{10.3},
which we denote as follows,
\begin{eqnarray}\label{10.5}
(b_0^{(k)},b_1^{(k)},...,b_{n-1}^{(k)}) =
\uuuu{a}\cdot 
\begin{pmatrix} 1 & 0 & \cdots & 0 \\
0 & \gamma_{22}^{(k)} & \ddots & \vdots \\
\vdots & \vdots & \ddots & 0  \\
0 & \gamma_{n2}^{(k)} & \cdots & \gamma_{nn}^{(k)} 
\end{pmatrix}\\
\quad\textup{with }
\gamma_{ii}^{(k)}\in\Q_{>0}\textup{ for }2\leq i\leq n
\quad\textup{and}\quad\gamma_{ij}^{(k)}\in\Q
\textup{ for }i>j\geq 2\nonumber
\end{eqnarray}
(we do not need 
$\gamma_{ij}^{(k)}\in(-\frac{1}{2}\gamma_{ii},
\frac{1}{2}\gamma_{ii}]
\cap\Q\textup{ for }i>j\geq 2$).
Then $b_0^{(k)}=1_A$ and for $i\in\{0,1,...,n-1\}$
\begin{eqnarray*}
b_i^{(k)}&\in& a^i\Q[a],\\
\langle b_i^{(k)},b_{i+1}^{(k)},...,b_{n-1}^{(k)}\rangle_\Z
&=& (a^i\Q[a])\cap L^k.
\end{eqnarray*}

{\bf Claim:} For $k\in\{1,...,n-1\}$
\begin{eqnarray}\label{10.6}
L^k&=& \langle 1_A,b_1^{(1)},b_2^{(2)},...,b_{k-1}^{(k-1)},
b_k^{(k)},b_{k+1}^{(k)},...,b_{n-1}^{(k)}\rangle_\Z.
\end{eqnarray}

{\bf Proof of the claim:}
We prove this by induction in $k$. It is obviously true for
$k=1$. Suppose it is true for some $k\in\{1,...,n-2\}$.
We want to prove it for $k+1$. 
$1_A\in L$ implies $L^k\subset L^{k+1}$. Therefore
\begin{eqnarray*}
L^{k+1}\supset \langle 1_A,b_1^{(1)},b_2^{(2)},...,
b_{k-1}^{(k-1)},b_k^{(k)},b_{k+1}^{(k+1)},...,b_{n-1}^{(k+1)}
\rangle_\Z =:\www{L},\\
(a^{k+1}\Q[a])\cap L^{k+1}=\langle b_{k+1}^{(k+1)},...,
b_{n-1}^{(k+1)}\rangle_\Z 
=(a^{k+1}\Q[a])\cap \www{L},
\end{eqnarray*}
\begin{eqnarray*}
\www{L}&=& \langle 1_A,b_1^{(1)},b_2^{(2)},...,b_{k-1}^{(k-1)},
b_k^{(k)}\rangle_\Z + (a^{k+1}\Q[a])\cap L^{k+1}\\
&\supset& \langle 1_A,b_1^{(1)},b_2^{(2)},...,b_{k-1}^{(k-1)},
b_k^{(k)}\rangle_\Z + (a^{k+1}\Q[a])\cap L^{k}\\
&=& \langle 1_A,b_1^{(1)},b_2^{(2)},...,b_{k-1}^{(k-1)},
b_k^{(k)}\rangle_\Z + \langle b_{k+1}^{(k)},...,b_{n-1}^{(k)}
\rangle_\Z 
\stackrel{\eqref{10.6}}{=}L^k.
\end{eqnarray*}
We want to show $L\cdot L^k\subset\www{L}$. 
We use the generators of $L^k$ in \eqref{10.6}
and the generators of $L$ in \eqref{10.6} with $k=1$.
We have
\begin{eqnarray*}
1_A\cdot L^{k}&=& L^k\subset \www{L},\\
b_i^{(1)}\cdot b_j^{(j)}&\in& L^{j+1}\subset L^k
\subset\www{L}\quad\textup{for }i\in\{1,...,n-1\},
j\in\{1,...,k-1\},\\
b_i^{(1)}\cdot b_j^{(k)}&\in& (a^{k+1}\Q[a])\cap L^{k+1}
=(a^{k+1}\Q[a])\cap\www{L}\subset\www{L}\\
&&\hspace*{2cm}\textup{for }i\in\{1,...,n-1\},
j\in\{k,...,n-1\}.
\end{eqnarray*}
Therefore $L\cdot L^k\subset \www{L}$, so $L^{k+1}=\www{L}$.
The induction step works. The claim is proved
\hfill$(\Box)$

\medskip
We will show that $L^{n-1}$ is multiplication invariant,
so $L^{n-1}\cdot L^{n-1}\subset L^{n-1}$. Then it is an order
because of $1_A\in L^{n-1}$, so it is invertible.
We use the generators of $L^{n-1}$ in \eqref{10.6} with
$k=n-1$. 
\begin{eqnarray*}
1_A\cdot L^{n-1}&=& L^{n-1},\\
b_i^{(i)}\cdot b_j^{(j)}&=& 0\textup{ for }
i,j\in\{1,...,n-1\}\textup{ with }i+j\geq n,\\
b_i^{(i)}\cdot b_j^{(j)}&\in& (a^{i+j}\Q[a])\cap L^{i+j}
\subset (a^{i+j}\Q[a])\cap L^{n-1}\subset L^{n-1}\\
&&\hspace*{2cm}\textup{for }i,j\in\{1,...,n-1\}\textup{ with }
i+j\leq n-1.
\end{eqnarray*}
Therefore $L^{n-1}\cdot L^{n-1}\subset L^{n-1}$, so $L^{n-1}$
is an order, so it is invertible.\hfill$\Box$

\subsection{The rank 2 case}\label{s10.2}
Now we come to the case $n=2$. 
Recall the Notations \ref{t1.2} (v) and (vi). 
The normal form in part (d) of Theorem \ref{t10.3} is also
observed in \cite[15.]{AO82}.

\begin{theorem}\label{t10.3}
Let $A=\Q\cdot 1_A + \Q\cdot a$ with $a^2=0$.

\begin{enumerate}
\item[(a)] 
The set of orders in $A$ is 
$$\{\textup{orders in }A\}=\{\Z\cdot 1_A + \Z\cdot \alpha a\,|\, 
\alpha\in\Q_{>0}\}.$$
\item[(b)] 
Each full lattice $L$ in $A$ has a unique $\Z$-basis
\begin{eqnarray*}
(1_A,a)\begin{pmatrix}\beta_{11}&0\\
\beta_{21}&\beta_{22}\end{pmatrix}\textup{ with }
\beta_{11},\beta_{22}\in\Q_{>0},
\beta_{21}\in (-\frac{1}{2}\beta_{22},\frac{1}{2}\beta_{22}]
\cap\Q.
\end{eqnarray*}
Its $\varepsilon$-class contains the order 
$\Z\cdot 1_A+\Z\cdot\frac{\beta_{22}}{\beta_{11}}a$
(and no other order). 
So $L$ is invertible and $\OO(L)=\Z\cdot 1_A+\Z\cdot 
\frac{\beta_{22}}{\beta_{11}}a$. The semigroup structure
on $\EE(A)=W(\EE(A))$ is
\begin{eqnarray*}
(\EE(A),\cdot)\cong (\{\textup{orders}\},\cdot)
\cong (\Q_{>0},{\gcd}_\Q)\cong 
(\{\Z\textup{-lattices}\neq0\textup{ in }\Q\},+).
\end{eqnarray*}
\item[(c)] 
Consider the order $\Lambda_a:=\Z\cdot 1_A+\Z\cdot a$. Then
the set of orders which contain $\Lambda_a$ is 
$$\{\textup{orders }\Lambda\supset\Lambda_a\}
=\{\Z\cdot 1_A+\Z\cdot \frac{1}{m}a\,|\, m\in\N\}$$
with semigroup structure isomorphic to
\begin{eqnarray*}
\bigl(\{\frac{1}{m}\,|\, m\in\N\},{\gcd}_\Q\bigr)
\stackrel{\cong}{\longrightarrow} 
\bigl(\N,\lcm),\quad\frac{1}{m}\mapsto m.
\end{eqnarray*}
\item[(d)] 
Each conjugacy class of nilpotent $2\times 2$ matrices
$\neq 0$ has a unique representative 
$\begin{pmatrix}0&m\\0&0\end{pmatrix}$ with $m\in\N$.
It comes from the $\varepsilon$-class of full lattices
which contains the order $\Z\cdot 1_A+\Z\cdot 
\frac{1}{m}a$.
\end{enumerate}
\end{theorem}

{\bf Proof:} 
(a) A special case of \eqref{10.4}

(b) The first statement is a special case of \eqref{10.2}.
For $L$ as in part (b) and 
$u:=(\beta_{11}\cdot 1_A + \beta_{21}\cdot a)^{-1}\in A^{*}$
$$u\cdot L=\Z\cdot 1_A + \Z\cdot \frac{\beta_{22}}{\beta_{11}}a.$$
The semigroup structure $(\{\textup{orders in }A\},\cdot)$
is given for $\alpha_1,\alpha_2\in\Q_{>0}$ by
\begin{eqnarray*}
&&(\Z\cdot 1_A+\Z\cdot \alpha_1a)
(\Z\cdot 1_A+\Z\cdot \alpha_2a)\\
&=&\Z\cdot 1_A + (\Z \alpha_1+\Z \alpha_2)a \\
&=&\Z\cdot 1_A +\Z\cdot {\gcd}_\Q(\alpha_1,\alpha_2)
a.
\end{eqnarray*}

(c) This follows from part (b).

(d) By Theorem \ref{t6.3} 
there is a 1:1 correspondence between the set of conjugacy
classes of nilpotent $2\times 2$ matrices $\neq 0$ and the set
of $\varepsilon$-classes $[L]_\varepsilon$ of full lattices
with $\OO(L)\supset \Lambda_a$. Choose in such an 
$\varepsilon$-class the order. It is $\Z\cdot 1_A +\Z \cdot 
\frac{1}{m}a$ for some $m\in\N$. With respect to the $\Z$-basis
$(\frac{1}{m}a,1_A)$ we obtain the nilpotent 
$2\times 2$ matrix $M$ with
$$a\cdot (\frac{1}{m}a,1_A) 
=(0,a)= (\frac{1}{m}a,1_A)\cdot M,\quad
\textup{so }M=\begin{pmatrix}0&m\\0&0\end{pmatrix}.
\hspace*{2cm}\Box$$

\subsection{The rank 3 case}\label{s10.3}
Now we come to the case $n=3$, so now
\begin{eqnarray}\label{10.7}
A=\Q\cdot 1_A+ \Q\cdot a +\Q\cdot a^2\textup{ with }a^3=0.
\end{eqnarray}
Now $\uuuu{a}=(1_A,a,a^2)$. 
By \eqref{10.2} each full lattice $L\in\LL(A)$ has a unique
$\Z$-basis of the shape
\begin{eqnarray}\label{10.8}
\uuuu{a}\begin{pmatrix}\beta_{11}&0&0\\
\beta_{21}&\beta_{22}&0\\
\beta_{31}&\beta_{32}&\beta_{33}\end{pmatrix}
\quad\textup{with }\beta_{11},\beta_{22},\beta_{33}\in\Q_{>0}\\
\textup{and }\beta_{ij}\in (-\frac{1}{2}\beta_{ii},
\frac{1}{2}\beta_{ii}]\cap\Q\textup{ for }i>j.\nonumber
\end{eqnarray}
Define the map
\begin{eqnarray*}
\Delta:\LL(A)\to\Q_{>0},\quad L\mapsto 
\frac{\beta_{22}^2}{\beta_{11}\beta_{33}}.
\end{eqnarray*}
The group $\Aut(A)$ of algebra automorphisms of $A$ is
\begin{eqnarray*}
\Aut(A)=\{f:A\to A&|& f\ \Q\textup{-linear},
f(1_A)=1_A,f(a)=f_1a+f_2a^2,\\ 
&&f(a^2)=f(a)^2=f_1^2a^2,f_1\in\Q-\{0\},f_2\in\Q\}.
\end{eqnarray*}

\begin{theorem}\label{t10.4} 
\begin{enumerate}
\item[(a)] 
The map $\Delta$ is constant on each $\varepsilon$-class,
so $\Delta:\EE(A)\to\Q_{>0}$ is well defined.
\item[(b)] 
It is also constant on each orbit in $\LL(A)$ of 
$\Aut(A)$.
\item[(c)] 
For each $L$ there are a unit $u\in A^{*}$ and an
automorphism $f\in\Aut(A)$ such that $f(u\cdot L)$ has the
$\Z$-basis
\begin{eqnarray}\label{10.9} 
\uuuu{a}\begin{pmatrix}1&0&0\\0&1&0\\0&0&\Delta(L)^{-1}
\end{pmatrix}.
\end{eqnarray}
\item[(d)] 
For $L\in\LL(A)$
\begin{eqnarray}\label{10.10}
\Delta(\OO(L))&=& \nu_\Q(\Delta(L))\cdot \delta_\Q(\Delta(L)),\\
\Delta(L^2)&=& \nu_\Q(\Delta(L)).\label{10.11}
\end{eqnarray}
\item[(e)] 
$L\in\LL(A)$ is invertible $\iff$ $\Delta(L)\in\N$.
\item[(f)] 
Consider an order $\Lambda$ in $A$. Then $\Delta(\Lambda)
\in\N$ by part (e). The set
$\{[L]_\varepsilon\,|\, L\in\LL(A),\OO(L)=\Lambda\}$
contains for each decomposition $\Delta(\Lambda)=n_1\cdot d_1$
with $n_1,d_1\in\N$ and $\gcd(n_1,d_1)=1$ a unique
$\varepsilon$-class $[L]_\varepsilon$ with 
$\Delta([L]_\varepsilon)=\frac{n_1}{d_1}$.
It does not contain other $\varepsilon$-classes. 
\end{enumerate}
\end{theorem}

{\bf Proof:}
(a) Multiplication of $L$ by a unit $u_0\cdot 1_A+u_1\cdot a
+ u_2\cdot a^2$ with $u_0\in\Q-\{0\}$, $u_1,u_2\in\Q$, changes
$\beta_{11},\beta_{22},\beta_{33}$ to 
$u_0\beta_{11},u_0\beta_{22},u_0\beta_{33}$, so 
$\Delta(u\cdot L)=\Delta(L)$.

(b) An automorphism $f\in\Aut(A)$ with $f(a)=f_1a+f_2a^2$,
$f_1\in\Q-\{0\}$, $f_2\in\Q$, changes $\beta_{11},\beta_{22},
\beta_{33}$ to $\beta_{11},f_1\beta_{22},f_1^2\beta_{33}$,
so $\Delta(f(L))=\Delta(L)$.

(c) Multiplication of $L$ with the unit
$\Bigl(\beta_{11}\cdot 1_A+\beta_{21}\cdot a+\beta_{31}\cdot a^2
\Bigr)^{-1}$ leads to the full lattice $u\cdot L$ with a
$\Z$-basis as in \eqref{10.3}. The automorphism
$f\in\Aut(A)$ with $f(\gamma_{22}a+\gamma_{32}a^2)=a$ leads to
the full lattice $f(u\cdot L)$ with a $\Z$-basis as in 
\eqref{10.9}. 

(d) It is sufficient to consider a full lattice $L\in\LL(A)$
with a $\Z$-basis as in \eqref{10.9}. Then $\OO(L)$ has the
$\Z$-basis 
$$\uuuu{a}\begin{pmatrix}1&0&0\\0&\delta_\Q(\Delta(L))&0\\
0&0&\Delta(L)^{-1}\end{pmatrix},$$
and $L^2$ has the $\Z$-basis
$$\uuuu{a}\begin{pmatrix}1&0&0\\0&1&0\\0&0&\nu_\Q(\Delta(L))^{-1}
\end{pmatrix},$$
so $\Delta(\OO(L))$ and $\Delta(L^2)$ are as claimed.

(e) Again it is sufficient to consider a full lattice with a 
$\Z$-basis as in \eqref{10.9}. The $\Z$-basis of $L^2$
in part (d) shows that $L^2$ is an order. 

$\Leftarrow$: If $\Delta(L)\in\N$
then $\delta_\Q(\Delta(L))=1$ and the $\Z$-basis of $L$ in 
\eqref{10.9} show that $L$ is an order, so invertible.

$\Rightarrow$: 
If $L$ is invertible then $\OO(L)=\OO(L^2)=L^2$ and 
$\Delta(\OO(L))=\Delta(L^2)$, so by part (d) 
$\delta_\Q(\Delta(L))=1$ and $\Delta(L)\in\N$.

(f) Let $\Lambda$ be an order. Consider a $\Z$-basis as in 
\eqref{10.4}. The automorphism $f\in\Aut(A)$ with
$f(\alpha_{22}a+\alpha_{32}a^2)=a$ maps $\Lambda$ to the 
order $f(\Lambda)$ with the $\Z$-basis
$$\uuuu{a}\begin{pmatrix}1&0&0\\0&1&0\\0&0&\Delta(L)^{-1}
\end{pmatrix}\quad\textup{where}\quad 
\Delta(L)=\frac{\alpha_{22}^2}{\alpha_{33}}\in\N.$$
We can suppose $f=\id$ and $\Lambda=f(\Lambda)$.
Now we proceed in three steps.

(i) Consider an $\varepsilon$-class $[L]_\varepsilon$ of full
lattices with $\OO([L]_\varepsilon)=\Lambda$ and an element
$L$ in it with a $\Z$-basis as in \eqref{10.3},
\begin{eqnarray*}
\uuuu{a}\begin{pmatrix}1&0&0\\0&\gamma_{22}&0\\
0&\gamma_{32}&\gamma_{33}\end{pmatrix}\quad
\textup{with }\gamma_{22},\gamma_{33}\in\Q_{>0}
\textup{ and }\gamma_{32}\in\Q.
\end{eqnarray*}
$\OO(L)=\Lambda$ implies $\gamma_{33}=\Delta(\Lambda)^{-1}$. 
With \eqref{10.10} this gives 
$$\gamma_{33}=\Delta(\Lambda)^{-1}
=\delta_\Q(\Delta(L))^{-1}\cdot \nu_\Q(\Delta(L))^{-1}.$$
Therefore 
\begin{eqnarray*}
\frac{\nu_\Q(\Delta(L))}{\delta_\Q(\Delta(L))}
&=&\Delta(L)=\frac{\gamma_{22}^2}{\gamma_{33}} 
=\gamma_{22}^2\Delta(\Lambda)
=\gamma_{22}^2\nu_\Q(\Delta(L))\cdot \delta_\Q(\Delta(L)),\\
\textup{so}&&\gamma_{22}=\delta_\Q(\Delta(L)) ^{-1}.
\end{eqnarray*}
$L$ contains $a\cdot 1_A=a$ and 
$\delta_\Q(\Delta(L))\bigl(\gamma_{22}a+\gamma_{32}a^2)=a
+\delta_\Q(\Delta(L))\cdot \gamma_{32}a^2$, so
$\delta_\Q(\Delta(L))\cdot \gamma_{32}a^2$.
Therefore $\delta_\Q(\Delta(L))\cdot \gamma_{32}\in 
\Z\cdot \gamma_{33}$, so 
$$\gamma_{32}\in\Z\cdot \frac{1}{\delta_\Q(\Delta(L))^2\cdot
\nu_\Q(\Delta(L))}
=\Z\cdot{\gcd}_\Q(\gamma_{22}^2,\gamma_{33}).$$

By Lemma \ref{t10.5}, which comes next, there is a unit
$u\in A^{*}$ such that $u\cdot L$ has the $\Z$-basis
$$\uuuu{a}\begin{pmatrix}1&0&0\\0&\gamma_{22}&0\\0&0&
\gamma_{33}\end{pmatrix}.$$

(ii) Vice versa, suppose that a decomposition 
$\Delta(\Lambda)=n_1\cdot d_1$ with $n_1,d_1\in\N$ and
$\gcd(n_1,d_1)=1$ is given. The full lattice $\www{L}$ with
$\Z$-basis
$$\uuuu{a}\begin{pmatrix}1&0&0\\0&d_1^{-1}&0\\0&0& 
d_1^{-1}n_1^{-1}\end{pmatrix}$$
satisfies $\OO(\www{L})=\Lambda$.

(iii) (i) and (ii) together give part (f).\hfill$\Box$

\bigskip
The semi-normal form in \eqref{10.3} can be improved to a normal
form in the case $n=3$.

\begin{lemma}\label{t10.5}
Let $A$ be as in \eqref{10.7}.
Each $\varepsilon$-class of full lattices in $A$ contains a 
unique full lattice $L$ with a $\Z$-basis
\begin{eqnarray}\label{10.12}
\uuuu{a}\begin{pmatrix}1&0&0&\\0&\gamma_{22}&0\\
0&\gamma_{32}&\gamma_{33}\end{pmatrix}\quad
\textup{with }\gamma_{22},\gamma_{33}\in\Q_{>0}\\
\textup{and }\gamma_{32}\in [0,{\gcd}_\Q(\gamma_{22}^2,
\gamma_{33}))\cap\Q.\nonumber
\end{eqnarray}
The $\Z$-basis is unique.
\end{lemma}

{\bf Proof:}
We start with a semi-normal form as in \eqref{10.3}, so almost
as in \eqref{10.12}, but we start with the condition 
$\gamma_{32}\in [0,\gamma_{33})\cap\Q$
instead of the sharper condition
$\gamma_{32}\in[0,{\gcd}_\Q(\gamma_{22}^2,\gamma_{33}))\cap\Q$.
For $r\in\Z$
$$\uuuu{a}\begin{pmatrix}1&0&0\\r\gamma_{22}&\gamma_{22}&0\\
r\gamma_{32}&\gamma_{32}&\gamma_{33}\end{pmatrix}$$
is a $\Z$-basis of $L$. Then a $\Z$-basis of 
$\bigl(1+r\gamma_{22}a+r\gamma_{32}a^2\bigr)^{-1}\cdot L$
is 
$$\uuuu{a}\begin{pmatrix}1&0&0\\0&\gamma_{22}&0\\
0&\gamma_{32}-r\gamma_{22}^2&\gamma_{33}\end{pmatrix}.$$
Therefore we can change $\gamma_{32}$ by adding multiples
of ${\gcd}_\Q(\gamma_{22}^2,\gamma_{33})$.
\hfill$\Box$

\begin{remarks}\label{t10.6}
\begin{enumerate}
\item[(i)] 
Consider an order $\Lambda$ in $A$. 
Then $\Delta(\Lambda)\in\N$ by Theorem \ref{t10.4} (e).
Theorem \ref{t10.4} (f) says that the number of 
$\varepsilon$-classes of exact $\Lambda$-ideals is
$\tau=2^t$ where $t=|\{p\in\P\,|\, p|\Delta(\Lambda)\}|$.
By Lemma \ref{t10.1} the $\varepsilon$-classes coincide
with the $w$-classes. Therefore the number $\tau=2^t$
is completely analogous to Faddeev's result
$\tau=2^t$ \cite[Corollary 4.1]{Fa65-2} 
(cited in Theorem \ref{t5.7})
for the separable 3-dimensional case.
\item[(ii)] 
By the beginning of the proof of Theorem \ref{t10.4} (f) 
an order $\Lambda$ is cyclic if and only if $\Delta(L)=1$. 
By part (i) this is equivalent to $\tau=1$.
The implication $(\textup{cyclic}\Rightarrow\tau=1)$
is a special case of Theorem \ref{t4.12}.
\item[(iii)] 
The semigroup structure on the set $\EE(A)$ of
$\varepsilon$-classes is not easy to describe.
Suppose that two full lattices
$L$ and $\www{L}$ with $\Z$-bases in normal form
$$\uuuu{a}\begin{pmatrix}1&0&0\\0&\gamma_{22}&0\\
0&\gamma_{32}&\gamma_{33}\end{pmatrix}
\quad\textup{and}\quad
\uuuu{a}\begin{pmatrix}1&0&0\\0&\www{\gamma}_{22}&0\\
0&\www{\gamma}_{32}&\www{\gamma}_{33}\end{pmatrix}$$
are given. Their product $L\cdot \www{L}$ is generated by the
tuple
$$\uuuu{a}\begin{pmatrix}
1&0&0&0&0&0\\0&\gamma_{22}&\www{\gamma}_{22}&0&0&0\\
0&\gamma_{32}&\www{\gamma}_{32}&\gamma_{22}\www{\gamma}_{22}&
\gamma_{33}&\www{\gamma}_{33}\end{pmatrix}
\sim \uuuu{a}\begin{pmatrix}1&0&0\\
0&\varepsilon_{22}&0\\0&\varepsilon_{32}&\varepsilon_{33}
\end{pmatrix}$$
with 
$\varepsilon_{22}={\gcd}_\Q(\gamma_{22},\www{\gamma}_{22})$.
But $\varepsilon_{33}$ and $\varepsilon_{32}$
are difficult to determine.
\end{enumerate}
\end{remarks}

Consider the automorphism
$$(\textup{multiplication with }a):A\to A$$
and the cyclic order $\Lambda_a=\Z\cdot 1 +\Z\cdot a 
+\Z\cdot a^2=\Z[a]$. Each $\varepsilon$-class $[L]_\varepsilon$
of full lattices with $\OO([L]_\varepsilon)\supset\Lambda_a$
gives rise to one conjugacy class of nilpotent integer
$3\times 3$ matrices with one $3\times 3$ Jordan block.
One such matrix $M$ is obtained by choosing a full lattice
$L$ and a $\Z$-basis $\BB=(b_0,b_1,b_2)$ of $L$ and defining
$M$ by 
$$a\cdot\BB=\BB\cdot M.$$
The  normal form in Theorem \ref{t10.7} (d) is also observed in
\cite[11. and 12.]{AO82}. Though the beautiful way, how one can
extract from it the corresponding $\varepsilon$-class of
full lattices and their order, is not contained in \cite{AO82}.

\begin{theorem}\label{t10.7}
Let $A$ be as in \eqref{10.7}. Consider the cyclic order
$\Lambda_a=\Z[a]\subset A$. 

\begin{enumerate}
\item[(a)] 
Each order $\Lambda\supset\Lambda_a$ has a unique 
$\Z$-basis of the following shape,
\begin{eqnarray}\label{10.13}
\uuuu{a}\begin{pmatrix}1&0&0\\0&\alpha_{22}&0\\
0&\alpha_{32}&\alpha_{33}\end{pmatrix}\\
\textup{with }\alpha_{22}=\frac{1}{n_2},\quad 
\alpha_{33}=\frac{1}{n_2^2n_3},\quad 
\alpha_{32}=\frac{n_4}{n_2^3n_3}\nonumber\\
\textup{for some }n_2,n_3\in\N,\ n_4\in [0,n_2)\cap\Z.
\nonumber
\end{eqnarray}
Then $\Delta(\Lambda)=n_3$.
\item[(b)] 
Consider an order $\Lambda$ as in part (a) and a 
decomposition $\Delta(\Lambda)=n_3=n_1d_1$ with
$n_1,d_1\in\N$ and $\gcd(n_1,d_1)=1$.
By Theorem \ref{t10.4} (f) this decomposition gives rise to one
$\varepsilon$-class $[L]_\varepsilon$ of exact $\Lambda$-ideals
with $\Delta([L]_\varepsilon)=\frac{n_1}{d_1}$.
The representative $L$ in Lemma \ref{t10.5} of this
$\varepsilon$-class has the $\Z$-basis
\begin{eqnarray}\label{10.14}
\uuuu{a}\begin{pmatrix}1&0&0\\0&\gamma_{22}&0\\
0&\gamma_{32}&\gamma_{33}\end{pmatrix}\quad\textup{with }\quad\\
\gamma_{22}=\frac{1}{d_1}\alpha_{22}=\frac{1}{d_1n_2},\ 
\gamma_{33}=\alpha_{33}=\frac{1}{n_2^2n_3},\ 
\gamma_{32}=\frac{1}{d_1}\alpha_{32}=\frac{n_4}{d_1n_2^3n_3}.
\nonumber
\end{eqnarray}
\item[(c)] 
Call the $\Z$-basis in \eqref{10.14} $(b_0,b_1,b_2)$. 
The nilpotent integer $3\times 3$ matrix $M$ with a single
Jordan block which is determined by
$$a\cdot (b_2,b_1,b_0)=(b_2,b_1,b_0)\cdot M$$
is
\begin{eqnarray}\label{10.15}
M&=& \begin{pmatrix}0&n_1n_2&-n_4\\0&0&d_1n_2\\0&0&0 
\end{pmatrix}.
\end{eqnarray}
Here $n_2\in\N$, $n_3(=n_1d_1)\in \N$ and $n_4\in [0,n_2)\cap\Z$
determine the associated order $\Lambda$, with
$\Delta(\Lambda)=n_3$, and the decomposition
$n_3=n_1d_1$ with $n_1,d_1\in\N$ and $\gcd(n_1,d_1)=1$ 
fixes the associated $\varepsilon$-class of exact 
$\Lambda$-ideals.
\item[(d)] 
The set 
\begin{eqnarray*}
\{\begin{pmatrix}0&m_1&-m_3 \\ 0&0&m_2\\ 0&0&0\end{pmatrix}\,|\,
m_1,m_2\in\N,m_3\in[0,\gcd(m_1,m_2))\cap\Z\}
\end{eqnarray*}
is a set of representatives of the $\GL_3(\Z)$-conjugacy classes
of nilpotent integer $3\times 3$ matrices with a $3\times 3$
Jordan block, one representative for each set.
From one representative the order $\OO(L)$ of the corresponding
$\varepsilon$-class $[L]_\varepsilon$ can be read off. 
It is as in \eqref{10.13}
with $n_2=\gcd(m_1,m_2)$, $n_3=\frac{m_1m_2}{n_2^2}$ and
$n_4=m_3$. Furthermore, $n_1$ and $d_1$ in \eqref{10.14}
are $n_1=\frac{m_1}{n_2}$ and $d_1=\frac{m_2}{n_2}$.
\end{enumerate}
\end{theorem}

{\bf Proof:}
(a) $a\in\Lambda_a\subset \Lambda$ 
implies $\alpha_{22}=\frac{1}{n_2}$
for some $n_2\in\N$, and 
$$(\alpha_{22}a+\alpha_{32}a^2)^2=\alpha_{22}^2a^2\in 
\Lambda\cap\Q a^2=\Z\alpha_{33}a^2$$
implies $\alpha_{22}^2/\alpha_{33}\in\N$, so
$\alpha_{33}=\frac{1}{n_2^2n_3}$ for some $n_3\in\N$.
Then $\Delta(L)=\alpha_{22}^2/\alpha_{33}=n_3$.
$a\in\Lambda_a\subset\Lambda$ implies also
\begin{eqnarray*}
a\in n_2(\alpha_{22}a+\alpha_{32}a^2)+\Z\alpha_{33}a^2
=a+(n_2\alpha_{32}+\Z\alpha_{33})a^2,
\end{eqnarray*}
so $\alpha_{32}\in\Z\cdot \frac{\alpha_{33}}{n_2}$, so
$\alpha_{32}=\frac{n_4}{n_2^3n_3}$ for some 
$n_4\in\Z.$
$\alpha_{32}$ can be changed by adding multiples of 
$\alpha_{33}$, so $n_4$ can be chosen uniquely in
$[0,n_2)\cap\Z$.

(b) The $\Z$-basis in \eqref{10.14} is in the normal form in
Lemma \ref{t10.5} because 
\begin{eqnarray*}
{\gcd}_\Q(\gamma_{22}^2,\gamma_{33})
={\gcd}_\Q(\frac{1}{d_1^2n_2^2},\frac{1}{n_1d_1n_2^2})
=\frac{1}{n_1d_1^2n_2^2}
=\frac{1}{d_1n_2^2n_3},\\
\gamma_{32}=\frac{n_4}{n_2}\cdot \frac{1}{d_1n_2^2n_3}
\in [0,\frac{1}{d_1n_2^2n_3})\cap\Q
=[0,{\gcd}_\Q(\gamma_{22}^2,\gamma_{33}))\cap\Q.
\end{eqnarray*}
The full lattice $L$ which is generated by this $\Z$-basis
satisfies
$\Delta(L)=\frac{\gamma_{22}^2}{\gamma_{33}}=\frac{n_1}{d_1}$.
The  inclusion $\OO(L)\supset \Lambda$ is obvious from
\begin{eqnarray*}
(\alpha_{22}a+\alpha_{32}a^2)\cdot 1 &=&
d_1(\gamma_{22}a+\gamma_{32}a^2),\\
(\alpha_{22}a+\alpha_{32}a^2)(\gamma_{22}a+\gamma_{32}a^2)
&=& \alpha_{22}\gamma_{22}a^2
=\frac{1}{d_1n_2^2}a^2
=n_1\cdot\gamma_{33}a^2,\\
\alpha_{33}a^2\cdot 1&=& \gamma_{33}a^2.
\end{eqnarray*}
The third equation, the factors $d_1$ and $n_1$ in the first
two equations and ${\gcd}(n_1,d_1)=1$ show also
$\OO(L)=\Lambda$.

(c) The calculation of $M$ is straightforward.

(d) This follows from part (c). \hfill$\Box$

\section{The rank 3 cases with Jordan blocks of sizes
$2$ and $1$}\label{s11}
\setcounter{equation}{0}
\setcounter{table}{0}
\setcounter{figure}{0}

\noindent 
A regular integer $3\times 3$ matrix with a $2\times 2$
Jordan block has two different integer eigenvalues.
By subtracting a suitable multiple of the unit matrix
we can restrict to the case with eigenvalue 
$\alpha\in\Z-\{0\}$ for the $1\times 1$ Jordan block
and the eigenvalue $0$ for the $2\times 2$ Jordan block.
By multiplying the matrix with the sign of $\alpha$ 
we can restrict to the case $\alpha\in\N$. 
The corresponding algebra is
\begin{eqnarray}\label{11.1}
A=\Q e_1+\Q e_2+\Q a\quad\textup{with }e_ie_j=\delta_{ij}e_i,\,
e_1a=0,\, e_2a=a,\, a^2=0.
\end{eqnarray}
It splits into the summands $\Q e_1$ and $\Q e_2+\Q a$.
Its separable part is $F=\Q e_1+\Q e_2$, its radical is
$\NN=\Q a$. The projection 
$$\pr_F:A\to F$$
with respect to the decomposition $A=F\oplus\NN$ is an
algebra homomorphism. 

The conjugacy classes of integer $3\times 3$ matrices with
a $2\times 2$ Jordan block with eigenvalue 0 and a 
$1\times 1$ Jordan block with eigenvalue $\alpha\in\N$ 
correspond by Theorem \ref{t6.3} 
1:1 to the $\varepsilon$-classes of $\Lambda_\alpha$-ideals
$\{[L]_\varepsilon\,|\, L\in\LL(A),\OO(L)\supset\Lambda_\alpha\}$
where $\Lambda_\alpha$ is the cyclic order
\begin{eqnarray*}
\Lambda_\alpha:= \Z[\alpha e_1+a]=\Z(e_1+e_2)+\Z(\alpha e_1+a)
+\Z\alpha^2 e_1\subset A.
\end{eqnarray*}
Here a $\Z$-basis $\BB$ of an arbitrary full lattice $L$ gives rise
to a matrix $M_\BB\in M_{3\times 3}(\Q)$ with eigenvalue
$\alpha$ on the $1\times 1$ Jordan block and eigenvalue $0$
on the $2\times 2$ Jordan block by the formula
\begin{eqnarray}\label{11.2}
(\alpha e_1+a)\BB=\BB\cdot M_\BB.
\end{eqnarray}
The matrix $M_\BB$ has integer entries if and only if
$\OO(L)\supset\Lambda_\alpha$. 

Lemma \ref{t11.1} provides a $\Z$-basis in normal form for each
order $\Lambda$ in $A$, a unique representative in each
$\varepsilon$-class of full lattices and a normal form for each
$\GL_3(\Z)$-conjugacy class of integer $3\times 3$ matrices
with a $2\times 2$ Jordan block with eigenvalues 0 
and a $1\times 1$ Jordan block with eigenvalue $\alpha\in\N$.
Though we do not determine the semigroup structure.

\begin{lemma}\label{t11.1}
Let $A$ be as in \eqref{11.1}, and let $\alpha\in\N$.
Denote $\uuuu{e}:=(e_1,e_2,a)$. 

\begin{enumerate}
\item[(a)] 
Each order in $A$ has a unique $\Z$-basis of the shape
\begin{eqnarray}\label{11.3}
\uuuu{e}\begin{pmatrix}0&\alpha_1&1\\0&0&1\\\alpha_2&\alpha_3
&0\end{pmatrix}\quad\textup{with}\\
\alpha_1\in\N,\ \alpha_2\in\Q_{>0},\ 
\alpha_3\in\frac{\alpha_2}{\alpha_1}([0,\alpha_1)\cap\Z).
\nonumber
\end{eqnarray}
\item[(b)] 
An order $\Lambda$ with $\Z$-basis as in \eqref{11.3}
contains $\Lambda_\alpha$ if and only if there exist
\begin{eqnarray}\nonumber
n_1&\in& \{\www{n}_1\in\N\,|\, \www{n}_1|\alpha\},\\
n_3&\in& [0,\frac{\alpha}{n_1})\cap\Z,\label{11.4}\\
n_2&\in& \bigl(n_1^2n_3+\Z\alpha\bigr)\cap\N\nonumber
\end{eqnarray}
with
\begin{eqnarray}\label{11.5}
\alpha_1=\frac{\alpha}{n_1}(\in\N),\ 
\alpha_2=\frac{\alpha}{n_2}(\in\Q_{>0}),\ 
\alpha_3=\frac{n_1n_3}{n_2}(=\frac{\alpha_2}{\alpha_1}n_3).
\end{eqnarray}
Then $\Lambda=\Lambda_\alpha$ if and only if 
$n_1=n_2=n_3=1$. 
\item[(c)] 
Each $\varepsilon$-class of full lattices contains a unique
full lattice $L$ with a $\Z$-basis of the following shape,
\begin{eqnarray}\label{11.6}
\uuuu{e}\begin{pmatrix}0&1&\delta_1\\0&0&1\\
\delta_2&\delta_3&0\end{pmatrix}
\quad\textup{with }\delta_1\in[0,\frac{1}{2}]\cap\Q,\ 
\delta_2\in\Q_{>0},\\
\begin{array}{ll}
\delta_3\in(-\frac{1}{2}\delta_2,\frac{1}{2}\delta_2]\cap\Q&
\textup{ if }\delta_1\in(0,\frac{1}{2}),\\
\delta_3\in [0,\frac{1}{2}\delta_2]\cap\Q&
\textup{ if }\delta_1\in\{0,\frac{1}{2}\}.
\end{array}\nonumber
\end{eqnarray}
\item[(d)] 
An order $\Lambda$ with $\Z$-basis in \eqref{11.3}
and a full lattice $L$ with $\Z$-basis in \eqref{11.6}
satisfy $\Lambda\cdot L=L$ if and only if
\begin{eqnarray}\label{11.7}
\frac{\alpha_2}{\delta_2}\in\Z,\ 
\alpha_1\delta_1\in\Z,\ 
\alpha_1\frac{\delta_3}{\delta_2}\in\Z,\ 
\frac{\alpha_3}{\delta_2}-\alpha_1\delta_1
\frac{\delta_3}{\delta_2}\in\Z.
\end{eqnarray}
\item[(e)] 
Let $L$ be a full lattice with $\Z$-basis in \eqref{11.6}.
The $\Z$-basis in \eqref{11.3} of the order $\OO(L)$ satisfies
\begin{eqnarray}\label{11.8}
\alpha_2=\delta_2,\ \alpha_1=\lcm\bigl(\delta_\Q(\delta_1),
\delta_\Q(\frac{\delta_3}{\delta_2})\bigr),\\
\alpha_3\equiv \frac{\alpha_2}{\alpha_1}\Bigl(
\alpha_1\delta_1\cdot\alpha_1\frac{\delta_3}{\delta_2}\Bigr)
\equiv \alpha_1\delta_1\delta_3\mmod\Z\alpha_2.
\nonumber
\end{eqnarray}
\item[(f)] 
\begin{enumerate}
\item[(i)] Let $L$ be a full lattice with $\Z$-basis $\BB$
in \eqref{11.6}. The matrix $M_\BB\in M_{3\times 3}(\Q)$
with $(\alpha e_1+a)\BB=\BB\cdot M_\BB$ is
\begin{eqnarray}\label{11.9}
M_\BB=\begin{pmatrix}0&-\alpha\frac{\delta_3}{\delta_2}&
\frac{1}{\delta_2}-\alpha\delta_1\frac{\delta_3}{\delta_2}\\
0&\alpha&\alpha\delta_1\\0&0&0\end{pmatrix}.
\end{eqnarray}
It is in $M_{3\times 3}(\Z)$ if and only if
$\OO(L)\supset \Lambda_\alpha$.
\item[(ii)] 
The matrices $M_\BB$ in \eqref{11.9} for
$(\delta_1,\delta_2,\delta_3)$ in \eqref{11.6} with
$M_\BB\in M_{3\times 3}(\Z)$ are representatives of the 
conjugacy classes of integer $3\times 3$ matrices with 
eigenvalue $\alpha$ on the $1\times 1$ Jordan block and
eigenvalue $0$ on the $2\times 2$ Jordan block.
(We do not list the triples $(\delta_1,\delta_2,\delta_3)$
in \eqref{11.6} with $M_\BB\in M_{3\times 3}(\Z)$ in a
more explicit way.)
\end{enumerate}
\end{enumerate}
\end{lemma}

{\bf Proof:}
(a) Let $\Lambda$ be an order in $A$. By Theorem \ref{t9.1} (a)
the order $\pr_F\Lambda$ in $A$ has a unique $\Z$-basis
$(e_1,e_2)\begin{pmatrix}\alpha_1&1\\0&1\end{pmatrix}$
for some $\alpha_1\in\N$.

The unit element $1_A=e_1+e_2$ is in $F$ and in $A$. 
The $\Z$-basis
of $\pr_F\Lambda$ lifts and extends to a $\Z$-basis of the shape
in \eqref{11.3} with $\alpha_1\in\N$, $\alpha_2\in\Q_{>0}$ and 
$\alpha_3\in [0,\alpha_2)\cap\Q$. The additional constraint
$\alpha_3\in \frac{\alpha_2}{\alpha_1}\bigl([0,\alpha_1)\cap\Z
\bigr)$ follows from $\Lambda$ being an order:
\begin{eqnarray*}
\Lambda\owns (\alpha_1e_1+\alpha_3a)^2&=&\alpha_1^2e_1 
=\alpha_1\cdot (\alpha_1e_1+\alpha_3 a)
-\frac{\alpha_1\alpha_3}{\alpha_2}\cdot\alpha_2a,\\
\textup{so}&&\frac{\alpha_1\alpha_3}{\alpha_2}\in\Z.
\end{eqnarray*}

(b) \eqref{11.10} puts the $\Z$-basis in \eqref{11.3} for the
order $\Lambda_\alpha$ and the $\Z$-basis in \eqref{11.3} for
an arbitrary order side by side,
\begin{eqnarray}\label{11.10}
\Lambda_\alpha:\quad
\uuuu{e}\begin{pmatrix}0&\alpha&1\\0&0&1\\ \alpha&1&0
\end{pmatrix},\qquad \Lambda:\quad 
\begin{pmatrix}0&\alpha_1&1\\0&0&1\\ \alpha_2&\alpha_3&0
\end{pmatrix} .
\end{eqnarray}
Write $\alpha_3=\frac{\alpha_2}{\alpha_1}n_3$ for some
$n_3\in[0,\alpha_1)\cap\Z$. Then 
$\Lambda\supset\Lambda_\alpha$ is equivalent to
\begin{eqnarray*}
\alpha_2&=& \frac{\alpha}{n_2}\textup{ for some }n_2\in\N,\\
\alpha_1&=&\frac{\alpha}{n_1}\textup{ for some }n_1\in\N
\textup{ with }n_1|\alpha,\\
1&\in& n_1\alpha_3+\Z\alpha_2 
=\frac{n_1^2n_3}{n_2}+\Z\frac{\alpha}{n_2}
=\frac{1}{n_2}\bigl(n_1^2n_3+\Z\alpha\bigr).
\end{eqnarray*}
Therefore $n_2\in (n_1^2n_3+\Z\alpha )\cap\Z.$

(c) Consider a full lattic $L_1$ in $A$. 
By Theorem \ref{t9.1} (b) the $\varepsilon$-class of 
$\pr_FL_1$ contains a unique full lattice with a $\Z$-basis
of the shape $(e_1,e_2)\begin{pmatrix}1&\delta_1\\0&1
\end{pmatrix}$
for some $\delta_0\in[0,\frac{1}{2}]\cap\Q$.
This full lattice lifts to a full lattice $L_2$ in the
$\varepsilon$-class of $L_1$ with a $\Z$-basis of the shape
\begin{eqnarray*}
\uuuu{e}\begin{pmatrix}0&1&\delta_1\\ 0&0&1\\ 
\delta_2&\delta_3&\delta_4\end{pmatrix}\textup{ with}\\
\delta_2\in\Q_{>0},\ \delta_3,\delta_4\in(-\frac{1}{2}\delta_2,
\frac{1}{2}\delta_2]\cap\Q.
\end{eqnarray*}
The full lattice $L_3= (e_1+e_2-\delta_4 a)L_2$ has the 
$\Z$-basis 
$\uuuu{e}\begin{pmatrix} 0&1&\delta_1\\ 0&0&1\\ 
\delta_2&\delta_3&0\end{pmatrix}$.

If $\delta_1=0$ the full lattice $(-e_1+e_2)L_3$ has the
$\Z$-basis $\uuuu{e}
\begin{pmatrix}0&1&0\\ 0&0&1\\ \delta_2&-\delta_3&0\end{pmatrix}$.

If $\delta_1=\frac{1}{2}$ we have the following full lattices
and $\Z$-bases,
\begin{eqnarray*}
&&(-e_1+e_2)L_3:\ 
\uuuu{e}\begin{pmatrix}0&1&-\frac{1}{2}\\ 0&0&1\\ 
\delta_2&-\delta_3&0\end{pmatrix} \textup{ and }
\uuuu{e}\begin{pmatrix}0&1&\frac{1}{2}\\ 0&0&1\\ 
\delta_2&-\delta_3&-\delta_3\end{pmatrix},\\
&&(e_1+e_2+\delta_3 a)(-e_1+e_2)L_3:\ 
\uuuu{e}\begin{pmatrix}0&1&\frac{1}{2}\\ 0&0&1\\ 
\delta_2&-\delta_3&0\end{pmatrix}.
\end{eqnarray*}
Therefore in the boundary cases $\delta_1\in\{0,\frac{1}{2}\}$
we can choose $\delta_3\in [0,\frac{1}{2}\delta_2]\cap\Q$.
We leave it to the reader to show that no further reductions
are possible.

(d) $\Lambda\cdot L$ is generated by the following nine products
of the three generators of $\Lambda$ in \eqref{11.3} and the
three generators of $L$ in \eqref{11.6},
$$\Lambda\cdot L:\quad \uuuu{e}\begin{pmatrix}
0&0&0&0&\alpha_1&\alpha_1\delta_1&0&1&\delta_1\\
0&0&0&0&0&0&0&0&1\\
0&0&\alpha_2&0&0&\alpha_3&\delta_2&\delta_3&0\end{pmatrix}.$$
So $\Lambda\cdot L\supset L$, and $\Lambda\cdot L=L$ is
equivalent to 
$$\alpha_2a\in L,\ \alpha_1e_1\in L,\ 
\alpha_1\delta_1e_1+\alpha_3a\in L.$$
which is equivalent to \eqref{11.7}

(e) The minimal $\alpha_1\in\N$ and $\alpha_2\in\Q_{>0}$
which satisfy the parts 
$$\frac{\alpha_2}{\delta_2}\in\Z,\ 
\alpha_1\delta_1\in\Z,\ 
\alpha_1\frac{\delta_3}{\delta_2}\in\Z$$
of \eqref{11.7} are $\alpha_2=\delta_2$ and 
$\alpha_1=\lcm(\delta_\Q(\delta_1),
\delta_\Q(\frac{\delta_3}{\delta_2}))$.
Then the last condition 
$\frac{\alpha_3}{\delta_2}-\alpha_1\delta_1
\frac{\delta_3}{\delta_2}\in\Z$ amounts to the last condition
in \eqref{11.8} for $\alpha_3$.

(f) (i) is a straightforward calculation.
(ii) is clear. \hfill$\Box$

\bigskip
Theorem \ref{t11.2} concludes section \ref{s11}.
It states which full lattices in $A$ are invertible.
This allows a short proof in the given case of Theorem 
\ref{t4.4}, namely that $L^2$ is invertible for each full $L$ 
lattice in $A$. But the main point is to prove Theorem \ref{t5.8}
in the given case. On the one hand, the proof is technical.
On the other hand, earlier results play well together in the
proof.

\begin{theorem}\label{t11.2}
Let $A$ be as in \eqref{11.1}.

\begin{enumerate}
\item[(a)] 
Consider a full lattice $L$ in $A$ with a $\Z$-basis as
in \eqref{11.6}. The following conditions are equivalent:
\begin{enumerate}
\item[(i)]
$L$ is invertible.
\item[(ii)]
$\OO(\pr_FL)=\pr_F\OO(L)$.
\item[(iii)]
The $\Z$-basis in \eqref{11.3} of the order $\OO(L)$ satisfies
$\alpha_1=\delta_\Q(\delta_1)$.
\item[(iv)]
$\delta_\Q(\frac{\delta_3}{\delta_2})$ divides 
$\delta_\Q(\delta_1)$.
\end{enumerate}
\item[(b)] 
Consider an order $\Lambda$ as in \eqref{11.3}.
Recall $\frac{\alpha_3\alpha_1}{\alpha_2}\in\Z\cap 
[0,\alpha_1)$ from \eqref{11.3}. Define
\begin{eqnarray*}
\mu&:=& \gcd(\alpha_1,\frac{\alpha_3\alpha_1}{\alpha_2})\in\N,\\
t&:=& |\{p\in\P\,|\, p\textup{ divides }\mu\}|,\\
\tau&:=& 2^t.
\end{eqnarray*}

\begin{enumerate}
\item[(i)] Then $\tau$ is the number of $w$-classes in the set 
$\{[L]_\varepsilon\,|\, L\in\LL(A),\OO(L)=\Lambda\}$.
\item[(ii)] 
For $L\in\LL(A)$ with $\OO(L)=\Lambda$ and a $\Z$-basis
as in \eqref{11.6} and \eqref{11.8} define
$$\mu_1:=\gcd(\alpha_1,\alpha_1\delta_1)\in\N,\quad
\mu_2:=\gcd(\alpha_1,\alpha_1\frac{\delta_3}{\alpha_2})\in\N.$$
Then
$$\mu=\mu_1\cdot \mu_2,\quad \gcd(\mu_1,\mu_2)=1.$$
\item[(iii)] 
$L$ and $\www{L}\in\LL(A)$ with $\OO(L)=\OO(\www{L})
=\Lambda$ and $\Z$-bases as in \eqref{11.6} and \eqref{11.8} 
satisfy 
$$L\sim_w\www{L}\ \iff\ \mu_1=\www{\mu}_1,\mu_2=\www{\mu}_2.$$
\end{enumerate}
(c) For each full lattice $L$ its square $L^2$ is invertible.
\end{enumerate}
\end{theorem}

{\bf Proof:}
(a) Consider the $\Z$-basis in \eqref{11.3} for $\OO(L)$
and the $\Z$-basis in \eqref{11.6} for $L$,
and recall Lemma \ref{t11.1} (e).

(i)$\Rightarrow$(ii) is part of Lemma \ref{t5.9} (c).

(ii)$\iff$(iii): 
$\pr_F\OO(L)$ has the $\Z$-basis
$(e_1,e_2)\begin{pmatrix}\alpha_1&1\\0&1\end{pmatrix}$.
$\pr_FL$ has the $\Z$-basis 
$(e_1,e_2)\begin{pmatrix}1&\delta_1\\0&1\end{pmatrix}$.
By Theorem \ref{t9.1} (d) $\OO(\pr_FL)$ has the $\Z$-basis
$(e_1,e_2)\begin{pmatrix}\delta_\Q(\delta_1)&1\\0&1
\end{pmatrix}$.

(iii)$\iff$(iv): By Lemma \ref{t11.1} (e) 
$\alpha_1=\lcm(\delta_\Q(\delta_1),
\delta_\Q(\frac{\delta_3}{\delta_2}))$.
Therefore (iii)$\iff$(iv).

(iv)$\Rightarrow$(i): 
By \eqref{11.8} and (iv) $\alpha_1=\delta_\Q(\delta_1)$. 
Write $n_\delta:=\nu_\Q(\delta_1)$, so 
$\delta_1=\frac{n_\delta}{\alpha_1}$.
Choose $l_1,l_2\in\Z$ with 
$n_\delta l_1=1+\alpha_1l_2$. 
By \eqref{11.8} $\alpha_2=\delta_2$,
$\alpha_3\equiv\alpha_1\delta_1\delta_3
=\delta_3n_\delta\mmod\Z\alpha_2$. 

We will show that the full lattice
$\www{L}$ with the following $\Z$-basis satisfies 
$L\cdot\www{L}=\OO(L)$, so then $L$ is invertible. 
$$\www{L}:\quad \uuuu{e}\begin{pmatrix}
0&\alpha_1^2&\alpha_1l_1\\ 0&0&1\\ 
\alpha_2&\alpha_3n_\delta&\alpha_3l_2\end{pmatrix}.$$
$L\cdot\www{L}$ is generated by the following elements,
\begin{eqnarray}\label{11.11}
\uuuu{e}\begin{pmatrix}
0&\alpha_1^2&\alpha_1^2\delta_1&\alpha_1l_1&\alpha_1l_1\delta_1\\
0&0&0&0&1\\
\alpha_2&0&\alpha_3n_\delta&\delta_3&\alpha_3l_2\end{pmatrix}
=\uuuu{e}\begin{pmatrix}
0&\alpha_1^2&\alpha_1n_\delta&\alpha_1l_1&n_\delta l_1\\
0&0&0&0&1\\
\alpha_2&0&\alpha_3n_\delta&\delta_3&\alpha_3l_2\end{pmatrix}.
\end{eqnarray}
The product $L\cdot\www{L}$ contains $\alpha_2a$, it contains
\begin{eqnarray*}
\uuuu{e}\begin{pmatrix}\alpha_1l_1n_\delta-\alpha_1^2l_2\\
0-0\\ \delta_3n_\delta-0\end{pmatrix}
=\uuuu{e}\begin{pmatrix}\alpha_1\\0\\ \delta_3n_\delta
\end{pmatrix},\ \textup{ so }
\uuuu{e}\begin{pmatrix}\alpha_1\\0\\ \alpha_3\end{pmatrix},
\end{eqnarray*}
and it contains 
\begin{eqnarray}\label{11.12}
\uuuu{e}\begin{pmatrix}n_\delta l_1 - \alpha_1l_2\\
1-0\\ \alpha_3l_2-\alpha_3l_2\end{pmatrix}
=\uuuu{e}\begin{pmatrix}1\\1\\0\end{pmatrix}.
\end{eqnarray}
So it contains the $\Z$-basis
$\uuuu{e}\begin{pmatrix}0&\alpha_1&1\\0&0&1\\
\alpha_2&\alpha_3&0\end{pmatrix}$ of $\OO(L)$, so it contains
$\OO(L)$.

The fifth element $n_\delta l_1e_1+e_2+\alpha_3l_2a$ 
in \eqref{11.11} is in $\OO(L)$ because of \eqref{11.12}.
The second element $\alpha_1^2e_1$ is in $\OO(L)$ because of
$\alpha_1\alpha_3\in\Z\alpha_2$.
The fourth element $\alpha_1l_1e_1+\delta_3a$ is in $\OO(L)$
because
$$\alpha_3l_1-\delta_3\equiv \delta_3n_\delta l_1-\delta_3
=\delta_3\alpha_1 l_2\equiv 0\mmod\Z\alpha_2,$$
where $\delta_3\alpha_1\in\Z\alpha_2$ follows from \eqref{11.7}
and $\delta_2=\alpha_2$.

(b) (i) follows from (ii) and (iii).

(ii) Recall $\alpha_2=\delta_2$ from \eqref{11.8},
$\alpha_1\delta_1\in\Z\cap [0,\frac{\alpha_1}{2}]$, 
$\alpha_1\frac{\delta_3}{\alpha_2}\in\Z\cap (-\frac{\alpha_1}{2},
\frac{\alpha_1}{2}]$ from \eqref{11.7}, and from \eqref{11.8}
\begin{eqnarray}
\alpha_1 &=& \lcm(\delta_\Q(\delta_1),\delta_\Q(
\frac{\delta_3}{\alpha_2})),\label{11.13}\\
\frac{\alpha_3\alpha_1}{\alpha_2}&\equiv& \alpha_1\delta_1\cdot 
\alpha_1\frac{\delta_3}{\alpha_2}\mmod \Z\alpha_1.
\label{11.14}
\end{eqnarray}
\eqref{11.13} implies $\gcd(\mu_1,\mu_2)=1$. 
Of course
\begin{eqnarray*}
1&=& \gcd(\frac{\alpha_1\delta_1}{\mu_1},\frac{\alpha_1}{\mu_1})
= \gcd(\frac{\alpha_1\delta_3/\alpha_2}{\mu_2},
\frac{\alpha_1}{\mu_2}) \\
&=& \gcd(\frac{\alpha_1\delta_1}{\mu_1}\cdot 
\frac{\alpha_1\delta_3/\alpha_2}{\mu_2},
\frac{\alpha_1}{\mu_1\mu_2}).
\end{eqnarray*}
With \eqref{11.14} this shows $\mu=\mu_1\cdot\mu_2$. 

(iii) The case $\alpha_1=1$: Then
$\mu=\mu_1=\mu_2=\www{\mu}_1=\www{\mu}_2=1$. 
Also, then $\alpha_3=0$ and $\Lambda=\Z[e_1+\alpha_2 a]$
is a cyclic order. By Theorem \ref{t4.12} 
each full lattice $L$ with $\OO(L)=\Lambda$ is invertible,
so there is only one $w$-class, the group 
$G([\Lambda]_\varepsilon)$. This fits to 
$\mu=\mu_1=\mu_2=\www{\mu}_1=\www{\mu}_2=1$.

The case $\alpha=2$: It is similar to the cases $\alpha_1\geq 3$.
Only $\frac{\varphi(\alpha_1)}{2}$ must be replaced by $1$.

The cases $\alpha\geq 3$: We state and will show three claims.

{\bf Claims:}
\begin{enumerate}
\item[(I)]
Each $w$-class within the set $\{[L]_\varepsilon\,|\, 
L\in\LL(A),\OO(L)=\Lambda\}$ contains 
$\frac{\varphi(\alpha_1)}{2}$ $\varepsilon$-classes.
\item[(II)]
$L$ and $\www{L}\in\LL(A)$ with $\OO(L)=\OO(\www{L})=\Lambda$
and $L\sim_w\www{L}$ satisfy $\mu_1=\www{\mu}_1$ and
$\mu_2=\www{\mu}_2$.
\item[(III)]
For each decomposition $\mu=\mu_1^0\cdot\mu_2^0$ with
$\mu_1^0,\mu_2^0\in\N$ with $\gcd(\mu_1^0,\mu_2^0)=1$,
the number of full lattices $L\in\LL(A)$ with $\OO(L)=\Lambda$,
with $\Z$-bases as in \eqref{11.6} and with 
$\mu_1=\mu_1^0$, $\mu_2=\mu_2^0$ is 
$\frac{\varphi(\alpha_1)}{2}$.
\end{enumerate}

Together (I), (II) and (III) show (iii).

{\bf Proof of the claims:}
{\bf (I)}
By Theorem \ref{t9.1} (d) $[G([\pr_F\Lambda]_\varepsilon)|
=\frac{\varphi(\alpha_1)}{2}$. 
By Theorem \ref{t5.10} 
the homomorphism $\pr_F:G([\Lambda]_\varepsilon)\to 
G([\pr_F\Lambda]_\varepsilon)$ is an isomorphism, so
$|G([\Lambda]_\varepsilon)|=\frac{\varphi(\alpha_1)}{2}$.
By Theorem \ref{t5.5} (d) 
each $w$-class within the set $\{[L]_\varepsilon\,|\, 
L\in\LL(A),\OO(L)=\Lambda\}$ contains
$|G([\Lambda]_\varepsilon)|$ many elements.

{\bf (II)}
Let $L$ and $\www{L}\in\LL(A)$ with $\OO(L)=\OO(\www{L})
=\Lambda$ and $L\sim_w\www{L}$. 
Let $L$ have a $\Z$-basis as in \eqref{11.6}
with invariants $\delta_1,\delta_2=\alpha_2,\delta_3$.
Then $\www{L}:L$ is invertible with $\OO(\www{L}:L)=\Lambda$
by Theorem \ref{t5.5} (b).
Let $\oooo{L}$ be the unique full lattice in the 
$\varepsilon$-class $[\www{L}:L]_\varepsilon$ with a 
$\Z$-basis as in \eqref{11.6} with invariants
$\oooo{\delta}_1,\oooo{\delta}_2,\oooo{\delta}_3$.
Then $\oooo{\delta}_2=\alpha_2$ by \eqref{11.8} and
$\alpha_1=\delta_\Q(\oooo{\delta}_1)$ by Theorem \ref{t11.2} (a).

Write $n_\delta:=\nu_\Q(\oooo{\delta}_1)$, so
$\oooo{\delta}_1=\frac{n_\delta}{\alpha_1}$, and choose
$l_1,l_2\in\Z$ with $n_\delta l_1=1+\alpha_1l_2$.
Now $(\alpha_1e_1+e_2)\oooo{L}$ and $L$ have the following
$\Z$-bases,
\begin{eqnarray*}
(\alpha_1e_1+e_2)\oooo{L}:\quad
\uuuu{e}\begin{pmatrix}0&\alpha_1&n_\delta\\ 0&0&1\\ 
\alpha_2&\oooo{\delta}_3&0\end{pmatrix},\qquad
L:\quad \uuuu{e}\begin{pmatrix}0&1&\delta_1\\ 0&0&1\\ 
\alpha_2&\delta_3&0\end{pmatrix}.
\end{eqnarray*}
Their product $(\alpha_1e_1+e_2)\oooo{L}\cdot L$ is in the
$\varepsilon$-class $[\www{L}]_\varepsilon$. It is generated
by the following elements,
\begin{eqnarray}
(\alpha_1e_1+e_2)\oooo{L}\cdot L:\quad 
\uuuu{e}\begin{pmatrix}0&\alpha_1&\alpha_1\delta_1&n_\delta & 
n_\delta\delta_1\\ 0&0&0&0&1\\ \alpha_2&0&\oooo{\delta}_3
&\delta_3&0\end{pmatrix}.\label{11.15}
\end{eqnarray}
We claim that it has the $\Z$-basis
\begin{eqnarray}
\uuuu{e}\begin{pmatrix}0&1&n_\delta\delta_1\\ 0&0&1\\
\alpha_2 & \delta_3 l_1 & 0\end{pmatrix}.\label{11.16}
\end{eqnarray}
The second element in \eqref{11.16} is obtained from the 
fourth and second element in \eqref{11.15} and
$n_\delta l_1=1+\alpha_1l_2$. 
The second, third and fourth element in \eqref{11.15}
are generated by the first and second element in \eqref{11.16}
because $\frac{\alpha_1\delta_3}{\alpha_2}\in\Z$ and 
$\frac{\alpha_1\oooo{\delta}_3}{\alpha_2}\in\Z$ imply 
\begin{eqnarray*}
\alpha_1\delta_3l_1&\equiv& 0\mmod \Z\alpha_2,\\
\alpha_1\delta_1\delta_3l_1- \oooo{\delta}_3
&\equiv& \alpha_3l_1-\oooo{\delta}_3
\equiv \alpha_1\oooo{\delta}_1\oooo{\delta}_3l_1-\oooo{\delta}_3
\\
&=&n_\delta\oooo{\delta}_3l_1-\oooo{\delta}_3
=\alpha_1\oooo{\delta}_3l_2 \\
&\equiv& 0\mmod \Z\alpha_2,\\
n_\delta\delta_3 l_1-\delta_3=\delta_3\alpha_1l_2&\equiv& 0
\mmod \Z\alpha_2.
\end{eqnarray*}

Then
\begin{eqnarray*}
\www{\mu}_1:=\gcd(\alpha_1,\alpha_1n_\delta \delta_1)
=\gcd(\alpha_1,\alpha_1\delta_1)=\mu_1,
\end{eqnarray*}
so $\www{\mu}_1=\mu_1$, $\www{\mu}_2=\mu_2$.

{\bf (III)} 
The number in question is the number of pairs
$(\delta_1,\delta_3)\in\Q^2$ with 
\begin{eqnarray*}
\alpha_1\delta_1\in\Z\cap [0,\frac{\alpha_1}{2}],\quad 
\alpha_1\frac{\delta_3}{\alpha_2}\in\Z\cap 
(-\frac{\alpha_1}{2},\frac{\alpha_1}{2}],\\
\frac{\alpha_1\alpha_3}{\alpha_2}\equiv \alpha_1\delta_1\cdot 
\alpha_1\frac{\delta_3}{\alpha_2}\mmod\Z\alpha_1,\\
\mu_1^0 = \gcd(\alpha_1,\alpha_1\delta_1),\quad
\mu_2^0=\gcd(\alpha_1,\alpha_1\frac{\delta_3}{\alpha_2}).
\end{eqnarray*}
Write $\alpha_1=\www{\alpha}_1\www{\mu}_1\www{\mu}_2$ with the
following primes $p\in\P$ as factors,
\begin{eqnarray*}
p\,|\, \www{\mu}_1 &\iff& p\,|\, \mu_1^0,\\
p\,|\, \www{\mu}_2 &\iff& p\,|\, \mu_2^0,\\
\textup{so }p\,|\, \www{\alpha}_1&\iff& p \not|\ \mu_1^0
\textup{ and }p\not|\ \mu_2^0.
\end{eqnarray*}
Then $\mu_1^0\,|\, \www{\mu}_1$, $\mu_2^0\,|\, \www{\mu}_2$,
\begin{eqnarray}
1 &=& \gcd(\frac{\alpha_1}{\mu_1^0},
\frac{\alpha_1\delta_1}{\mu_1^0})
=\gcd(\frac{\alpha_1}{\mu_1^0\mu_2^0},
\frac{\alpha_1\delta_1}{\mu_1^0}),\label{11.17}\\
1 &=& \gcd(\frac{\alpha_1}{\mu_2^0},
\frac{\alpha_1\delta_3/\alpha_2}{\mu_2^0})
=\gcd(\frac{\alpha_1}{\mu_1^0\mu_2^0},
\frac{\alpha_1\delta_3/\alpha_2}{\mu_2^0}),\label{11.18}\\
1 &=& \gcd(\frac{\alpha_1}{\mu_1^0\mu_2^0},
\frac{\alpha_1\delta_1}{\mu_1^0}\cdot 
\frac{\alpha_1\delta_3/\alpha_2}{\mu_2^0}),\nonumber\\
\frac{\alpha_1\alpha_3/\alpha_2}{\mu} 
&=&\frac{\alpha_1\alpha_3/\alpha_2}{\mu_1^0\mu_2^0}
\equiv \frac{\alpha_1\delta_1}{\mu_1^0}\cdot 
\frac{\alpha_1\delta_3/\alpha_2}{\mu_2^0}\mmod \Z
\frac{\alpha_1}{\mu_1^0\mu_2^0}.\nonumber
\end{eqnarray}
The number of $\frac{\alpha_1\delta_1}{\mu_1^0}\in\Z\cap 
[0,\frac{\alpha_1}{2\mu_1^0}]$ with \eqref{11.17} is
$\frac{1}{2}\varphi(\frac{\alpha_1}{\mu_1^0})$.
The number of $\frac{\alpha_1\delta_3/\alpha_2}{\mu_2^0}
\in\Z\cap (-\frac{1}{2}\frac{\alpha_1}{\mu_2^0},
\frac{1}{2}\frac{\alpha_1}{\mu_2^0}]$ with \eqref{11.18}
is $\varphi(\frac{\alpha_1}{\mu_2^0})$. 

Their products take modulo $\Z\frac{\alpha_1}{\mu_1^0\mu_2^0}$
each value in $\Z^*_{\alpha_1/(\mu_1^0\mu_2^0)}$ with equal
multiplicity. Therefore the number in question is
\begin{eqnarray*}
\frac{\frac{1}{2}\varphi(\frac{\alpha_1}{\mu_1^0})
\varphi(\frac{\alpha_1}{\mu_2^0})}
{\varphi(\frac{\alpha_1}{\mu_1^0\mu_2^0})}
= \frac{1}{2}\frac{\varphi(\www{\alpha}_1)
\varphi(\frac{\www{\mu}_1}{\mu_1^0}) \varphi(\www{\mu}_2)\cdot
\varphi(\www{\alpha}_1)\varphi(\www{\mu}_1)
\varphi(\frac{\www{\mu}_2}{\mu_2^0})}
{\varphi(\www{\alpha}_1)\varphi(\frac{\www{\mu}_1}{\mu_1^0})
\varphi(\frac{\www{\mu}_2}{\mu_2^0})}
=\frac{1}{2}\varphi(\alpha_1).
\end{eqnarray*}

(c) It is sufficient to consider a lattice $L$ with a $\Z$-basis
as in \eqref{11.6}. Then $L^2$ is generated by the following
vectors,
$$\uuuu{e}\begin{pmatrix}0&1&\delta_1&\delta_1^2\\
0&0&0&1\\ \delta_2&0&\delta_3&0\end{pmatrix}.$$
Write $d_\delta:=\delta_\Q(\delta_1)$ and 
$n_\delta:=\nu_\Q(\delta_1)$, so
$\delta_1=\frac{n_\delta}{d_\delta}$.
Choose $l_1,l_2\in\Z$ with $n_\delta l_1+d_\delta l_2=1$.
Then $L^2$ is generated by the following vectors,
\begin{eqnarray*}
\uuuu{e}\begin{pmatrix}0&\frac{1}{d_\delta}&0&0&\delta_1^2\\
0&0&0&0&1\\ \delta_2&\delta_3l_1 & \delta_3l_1d_\delta&
\delta_3l_2d_\delta&0 
\end{pmatrix},\textup{ so by }
\uuuu{e}\begin{pmatrix}0&\frac{1}{d_\delta}&\delta_1^2\\
0&0&1\\ {\gcd}_\Q(\delta_2,\delta_3d_\delta)& 
\delta_3l_1&0\end{pmatrix}.
\end{eqnarray*}
$(d_\delta e_1+e_2)L^2$ is generated by 
$$\uuuu{e}\begin{pmatrix}0&1&\frac{n_\delta^2}{d_\delta}\\
0&0&1\\ {\gcd}_\Q(\delta_2,\delta_3d_\delta) & 
\delta_3l_1&0\end{pmatrix}.$$
Because of part (a) it remains to show that
$\delta_\Q(\frac{\delta_3l_1}{{\gcd}_\Q(\delta_2,
\delta_3d_\delta)})$ divides
$\delta_\Q(\frac{n_\delta^2}{d_\delta})=d_\delta$. 
But if $\delta_3\neq 0$ then there is a number $l_3\in\N$ with
\begin{eqnarray*}
\delta_\Q(\frac{\delta_3l_1}{{\gcd}_\Q(\delta_2,
\delta_3d_\delta)})
=\delta_\Q(\frac{1}{{\gcd}_\Q(\frac{\delta_2}{\delta_3l_1},
\frac{d_\delta}{l_1})})
=\delta_\Q(\frac{l_3}{d_\delta}),
\end{eqnarray*}
so this divides $d_\delta$. 
If $\delta_3=0$ then $\delta_\Q(0)=1$ which divides $d_\delta$.
\hfill$\Box$

\section{The 3-dimensional algebra which is not cyclic}
\label{s12}
\setcounter{equation}{0}
\setcounter{table}{0}
\setcounter{figure}{0}

\noindent 
All 3-dimensional algebras except one are cyclic and therefore
relevant for the study of $\GL_3(\Z)$-conjugacy classes of
regular integer $3\times 3$ matrices. The exception is the
algebra $\Q[x,y]/(x^2,xy,y^2)$.
It is therefore and because of the simplicity of the
structures of its full lattices, orders and the semigroup
$\EE(A)$ the least interesting of all 3-dimensional algebras.
Lemma \ref{t12.1} treats it for completeness sake, and
in order to prove Theorem \ref{t5.8} in this case.

\begin{lemma}\label{t12.1}
Let $A$ by the algebra $A=\Q 1_A\oplus\Q a\oplus \Q b$
with $a^2=ab=b^2=0$ (it is isomorphic to $\Q[x,y]/(x^2,xy,y^2)$).
Then $A=F\oplus R$ with $F=\Q 1_A$ and $R=\Q a\oplus \Q b$.

Each full lattice is invertible.
Each $\varepsilon$-class of full lattices contains an order.
The orders are the full lattices $\Z 1_A\oplus K$ with
$K\in\LL(R)$. Therefore the semigroup $\EE(A)$ is 
\begin{eqnarray*}
\EE(A)=W(\EE(A))\cong (\{\textup{orders}\},\cdot)
\cong (\LL(R),+).
\end{eqnarray*}
\end{lemma}

{\bf Proof:} Each full lattice $L\in\LL(A)$ has a unique 
$\Z$-basis of the shape
\begin{eqnarray*}
(1_A,a,b)\begin{pmatrix}\beta_{11}&0&0\\ 
\beta_{21}&\beta_{22}&0\\ \beta_{31}&\beta_{32}&\beta_{33}
\end{pmatrix}\quad\textup{with }\beta_{11},\beta_{22},
\beta_{33}\in\Q_{>0}\\
\hspace*{2cm}\beta_{ij}\in(-\frac{1}{2}\beta_{ii},
\frac{1}{2}\beta_{ii}]\cap\Q\textup{ for }i>j.
\end{eqnarray*}
The first basis element $u:=\beta_{11}1_a+\beta_{21}a+\beta_{31}b$
is in $A^{*}$. The full lattice $L$ is an order if and only if
$u=1_A$. The full lattice $u^{-1}L$ is an order.
The remaining statements are clear, too. \hfill$\Box$


\begin{thebibliography}{AAA9}
\bibitem[AO81]{AO81} H. Appelgate, H. Onishi: \quad
   Continued fractions and the conjugacy problem in 
   $\SL_2(\Z)$.
   Commun. Algebra {\bf 9(11)} (1981), 1121--1130.
\bibitem[AO82]{AO82} H. Appelgate, H. Onishi: \quad 
   The similarity problem for $3\times 3$ integer matrices.
   Linear Algebra Appl. {\bf 42} (1982), 159--174.
\bibitem[AMD69]{AMD69} M.F. Atiyah, I.G. MacDonald: \quad
   Introduction to commutative algebra. 
   Addison-Wesley Publishing Company, Inc., 1969.
\bibitem[BHJ22]{BHJ22} W. Bley, T. Hofmann, H. Johnston: \quad
   Computation of lattice isomorphisms and the integral matrix
   similarity problem.
   Forum of Mathematics, Sigma {\bf 10:e87} (2022), 1--36.
\bibitem[BSh73]{BSh73} Z.I. Borevich, I.R. Shafarevich: \quad 
   Number theory. Academic Press Inc. 1973.
\bibitem[BF65]{BF65} Z.I. Borevi\v{c}, D.K. Faddeev: \quad 
   Representations of orders with cyclic index.
   Trudy Mat. Inst. Steklov {\bf 80} (1965), 51--65 (russian).
   Proc. Steklov Inst. Math. {\bf 80} (1965), 56--72 (translation).
\bibitem[Bo90]{Bo90} N. Bourbaki: \quad 
   Algebra II. Springer, 1990. 
\bibitem[Coh78]{Coh78} H. Cohn:\quad
   A classical invitation to algebraic numbers and class fields
   (with two appendices by Olga Taussky). Springer 1978.
\bibitem[Co97]{Co97} J.H. Conway: \quad
   The Sensual (Quadratic) Form.
   The Carus Mathematical Monographs {\bf 26}.
   Mathematical Association of America, 1997.
\bibitem[CR62]{CR62} C.W. Curtis, I. Reiner:\quad 
   Representation theory of finite groups and 
   associative algebras.
   Interscience Publishers 1962.
\bibitem[DTZ62]{DTZ62} E.C. Dade, O. Taussky, H. Zassenhaus:\quad
   On the theory of orders, in particular on the semigroup 
   of ideal classes and genera of an order in an algebraic 
   number field.
   Math. Ann. {\bf 148} (1962), 31--64.
\bibitem[De72]{De72} R. Dedekind: \quad
   \"Uber die Theorie der ganzen algebraischen Zahlen.
   Supplement to the book of P.G. Lejeune-Dirichlet: 
   Vorlesungen \"uber Zahlentheorie. 2nd edition, 
   Braunschweig 1872.
   (Also in Ges. math. Werke III, Vieweg 1932). 
\bibitem[EHO19]{EHO19} B. Eick, T. Hofmann, E.A. O'Brien: 
   \quad The conjugacy problem in $\GL(n,\Z)$.
   J. London Math. Soc. (2) {\bf 100.3} (2019), 731--756.
\bibitem[Fa65-1]{Fa65-1} D.K. Faddeev: \quad 
   An introduction to multiplicative theory of modules of
   integral representations.
   Trudy Mat. Inst. Steklov {\bf 80} (1965), 145--182 (russian).
   Proc. Steklov Inst. Math. {\bf 80} (1965), 164--210 
   (translation).
\bibitem[Fa65-2]{Fa65-2} D.K. Faddeev: \quad 
   On the theory of cubic $\Z$-rings. 
   Trudy Mat. Inst. Steklov {\bf 80} (1965), 183--187 (russian).
   Proc. Steklov Inst. Math. {\bf 80} (1965), 211--215 
   (translation).
\bibitem[Fa67]{Fa67} D.K. Faddeev: \quad 
   The number of classes of exact ideals for $\Z$-rings.
   Mat. Zametki {\bf 1} (1967), 625--632 (russian).
   Math. Notes {\bf 1} (1967), 415--419 (translation). 
\bibitem[Fa68]{Fa68} D.K. Faddeev: \quad
   Equivalence of systems of integer matrices.
   Amer. Math. Soc. Transl. (2) {\bf 71} (1968),
   43--48.
\bibitem[Fr65]{Fr65} A. Fr\"ohlich: \quad 
   Invariants for modules over commutative separable orders.
   Quart. J. Math. Oxford Ser. (2) {\bf 16} (1965), 193--232. 
\bibitem[Ga01]{Ga01} C.F. Gau{\ss}: \quad 
   Disquisitiones arithmeticae. G\"ottingen, 1801.
   German translation (and enlargement 
   by further papers of Gau{\ss}):
   Untersuchungen \"uber h\"ohere Arithmetik.
   G\"ottingen, 1889 (2nd edition Chelsea 1965).
\bibitem[Gr80]{Gr80} F.J. Grunewald: \quad 
   Solution of the conjugacy problem in certain arithmetic 
   groups. In: Word problems, II (Conf. on decision problems
   in algebra, Oxford, 1976), Stud. Logic Foundations Math.
   vol. 95, North Holland, Amsterdam-New York, 1980, 
   pp. 101--139.
\bibitem[Ha21]{Ha21} A. Hatcher: \quad 
   Topology of numbers. Book manuscript, 2021.
\bibitem[HL26]{HL26} C. Hertling, K. Larabi: \quad
   Semigroups from full lattices in commutative
   $\Q$-algebras. Accepted for publication in the 
   Journal of Commutative Algebra.
   arXiv:2602.14973v1, 48 pages, February 2026.
\bibitem[Hu17]{Hu17} D. Husert: \quad
   Similarity of integer matrices. Ph.D. thesis,
   University of Paderborn, 2017, 147 pages. 
\bibitem[Ku12]{Ku12} R. Kurbel:\quad
   Konjugationsklassen ganzzahliger Matrizen.
   PhD thesis, University Mannheim, 2012. 
\bibitem[LMD33]{LMD33} C.G. Latimer, C.C. MacDuffee:\quad
   A correspondence between classes of ideals and classes 
   of matrices.
   Annals of Mathematics {\bf 34} (1933), 313--316.
\bibitem[LMFDB23]{LMFDB23} The LMFDB collaboration: 
   The $L$-functions and modular forms database. 2023.
   Available at https://www.lmfdb.org.
\bibitem[Ma20]{Ma20} S. Marseglia: \quad 
   Computing the ideal class monoid of an order.
   J. London Math. Soc {\bf 101.3} (2020), 984--1007.
\bibitem[Ne99]{Ne99} J. Neukirch: \quad 
   Algebraic number theory. 
   Grundlehren der mathematischen Wissenschaften 322,    
   Springer 1999.
\bibitem[Re70]{Re70} I. Reiner:\quad
   A survey of integral representation theory.
   Bull. Am. Math. Soc. {\bf 76} (1970), 157--227. 
\bibitem[Re03]{Re03} I. Reiner: \quad 
   Maximal orders. Clarendon Press 2003 or Academic Press 1975.
\bibitem[Se73]{Se73} J.-P. Serre: \quad
   A course in arithmetic.
   Graduate texts in mathematics 7. Springer, 1973. 
\bibitem[Si70]{Si70} M. Singer: \quad 
   Invertible powers of ideals over orders in commutative
   separable algebras.
   Proc. Cambridge Philos. Soc. {\bf 67.2} (1970), 237--242. 
\bibitem[St11]{St11} E. Steinitz: \quad 
   Rechteckige Systeme und Moduln in algebraischen Zahlk\"orpern.
   I, II, Math. Ann. {\bf 71} (1911), 328--354,
   {\bf 72} (1912), 297--345. 
\bibitem[Ta60]{Ta60} O. Taussky:\quad 
   Matrices of rational integers. 
   Bull. Am. Math. Soc. {\bf 66} (1960), 327--345. 
\bibitem[Ta74]{Ta74} O. Taussky: \quad
   A result concerning classes of matrices.
   Journal of Number Theory {\bf 6} (1974), 64--71.
\bibitem[Ta78]{Ta78} O. Taussky:\quad 
   Introduction into connections between algebraic number 
   theory and integral matrices. 
   2nd appendix in \cite{Coh78}, 305--326. 
\bibitem[Tr13]{Tr13} M. Trifkovi\'c: \quad
   Algebraic theory of quadratic numbers. Universitext,
   Springer, 2013. 
\bibitem[We17]{We17} M.H. Weissman: \quad
   An illustrated theory of numbers.
   American Mathematical Society, 2017.
\bibitem[Zag81]{Zag81} D.B. Zagier:\quad 
   Zetafunktionen und quadratische K\"orper.
   Eine Einf\"uhrung in die h\"ohere Zahlentheorie.
   Springer, 1981. 
\bibitem[ZZ94]{ZZ94} P. Zanardo, U. Zannier: \quad 
   The class semigroup of orders in number fields.
   Math. Proc. Cambridge Philos. Soc. {\bf 115.3} (1994), 
   379--391. 
\bibitem[Za38]{Za38} H. Zassenhaus:\quad
   Neuer Beweis der Endlichkeit der Klassenzahl bei unimodularer
   \"Aquivalenz endlicher ganzzahliger Substitutionsgruppen.
   Abh. Math. Sem. Univ. Hamburg {\bf 12} (1938), 276--288.
\end{thebibliography}
\end{document}